\documentclass[10pt,english,reqno]{amsart} 

\usepackage[usenames,dvipsnames]{color}

\usepackage[hidelinks]{hyperref}
\usepackage{graphicx}
\usepackage{amssymb}
\usepackage{epstopdf}
\usepackage{enumerate}
\usepackage{mathabx}
\usepackage{tikz-cd}
\usepackage{MnSymbol}
\usepackage{mathrsfs}

\DeclareMathSymbol{A}{\mathalpha}{operators}{`A}
\DeclareMathSymbol{B}{\mathalpha}{operators}{`B}
\DeclareMathSymbol{C}{\mathalpha}{operators}{`C}
\DeclareMathSymbol{D}{\mathalpha}{operators}{`D}
\DeclareMathSymbol{E}{\mathalpha}{operators}{`E}
\DeclareMathSymbol{F}{\mathalpha}{operators}{`F}
\DeclareMathSymbol{G}{\mathalpha}{operators}{`G}
\DeclareMathSymbol{H}{\mathalpha}{operators}{`H}
\DeclareMathSymbol{I}{\mathalpha}{operators}{`I}
\DeclareMathSymbol{J}{\mathalpha}{operators}{`J}
\DeclareMathSymbol{K}{\mathalpha}{operators}{`K}
\DeclareMathSymbol{L}{\mathalpha}{operators}{`L}
\DeclareMathSymbol{M}{\mathalpha}{operators}{`M}
\DeclareMathSymbol{N}{\mathalpha}{operators}{`N}
\DeclareMathSymbol{O}{\mathalpha}{operators}{`O}
\DeclareMathSymbol{P}{\mathalpha}{operators}{`P}
\DeclareMathSymbol{Q}{\mathalpha}{operators}{`Q}
\DeclareMathSymbol{R}{\mathalpha}{operators}{`R}
\DeclareMathSymbol{S}{\mathalpha}{operators}{`S}
\DeclareMathSymbol{T}{\mathalpha}{operators}{`T}
\DeclareMathSymbol{U}{\mathalpha}{operators}{`U}
\DeclareMathSymbol{V}{\mathalpha}{operators}{`V}
\DeclareMathSymbol{W}{\mathalpha}{operators}{`W}
\DeclareMathSymbol{X}{\mathalpha}{operators}{`X}
\DeclareMathSymbol{Y}{\mathalpha}{operators}{`Y}
\DeclareMathSymbol{Z}{\mathalpha}{operators}{`Z}

\newcommand{\integers}{\mathbf Z}
\newcommand{\rationals}{\mathbf Q}
\newcommand{\complexes}{\mathbf C}
\newcommand{\residue}{\textnormal{\textsf{f}}}
\newcommand{\Isoc}{\textnormal{\textsf{Isoc}}}
\newcommand{\Basic}{\textnormal{\textsf{Basic}}}
\newcommand{\unr}{\mathrm{unr}}
\newcommand{\Gal}{\operatorname{Gal}}
\newcommand{\Kottwitz}{\textnormal{Kott}}
\newcommand{\ord}{\mathrm{ord}}
\newcommand{\characteristic}{\operatorname{char}}

\newcommand{\Maps}{\operatorname{Maps}}
\newcommand{\Cov}{\textnormal{\textsf{Cov}}}

\newcommand{\cycle}{\textnormal{\textsf{Z}}}
\newcommand{\extension}{\textnormal{\textsf{E}}}
\newcommand{\SMaps}{\mathscr{M}aps}
\newcommand{\Hom}{\operatorname{Hom}}

\newcommand{\SHom}{\mathscr{H}om}
\newcommand{\Ker}{\mathrm{Ker}}

\newcommand{\id}{\mathrm{id}}
\newcommand{\Fib}{\mathrm{Fib}}

\newcommand{\deloop}{\textnormal{\textsf{B}}}
\DeclareMathOperator*\colim{colim}

\newcommand{\Spec}{\operatorname{Spec}}

\newcommand{\sconn}{\textnormal{sc}}
\newcommand{\adjoint}{\textnormal{ad}}
\newcommand{\abelian}{\textnormal{ab}}
\newcommand{\GL}{\mathrm{GL}}

\newcommand{\SQuad}{\mathscr{Q}uad}
\newcommand{\strict}{\textnormal{st}}
\newcommand{\LLC}{\textnormal{LLC}}
\newcommand{\tors}{\textnormal{-tors}}
\newcommand{\sgn}{\mathrm{sgn}}
\newcommand{\Hilbert}{\mathrm{Hilb}}
\newcommand{\Ktheory}{\textnormal{\textsf{K}}}
\newcommand{\Langlands}{\textnormal{\textsf{L}}}
\newcommand{\Translation}{\mathrm{T}}

\newtheorem{thm}[subsubsection]{Theorem}
\newtheorem*{thm*}{Theorem}
\newtheorem{thmx}{Theorem}

\newtheorem{prop}[subsubsection]{Proposition}
\newtheorem{lem}[subsubsection]{Lemma}
\newtheorem{conj}[subsubsection]{Conjecture}
\newtheorem{cor}[subsubsection]{Corollary}

\theoremstyle{definition}

\newtheorem{rem}[subsubsection]{Remark}

\numberwithin{equation}{section}

\newtheorem{untitledsubsubsection}[subsubsection]{}

\makeatletter
\renewcommand{\@secnumfont}{\bfseries}
\makeatother

\newenvironment{void}
{\begin{untitledsubsubsection}}
{\end{untitledsubsubsection}}

\title[Extended pure inner forms of covers]{What are the extended pure inner forms of a cover?}

\author{Luozi Shi \and Yifei Zhao}
\date{\today}

\thanks{
The project was funded by the Deutsche Forschungsgemeinschaft (DFG, German Research Foundation) Project-ID 427320536 -- SFB 1442, as well as under Germany's Excellence Strategy EXC 2044 390685587, Mathematics Münster: Dynamics--Geometry--Structure.
}

\begin{document}

\begin{abstract}
Kottwitz suggested to study all extended pure inner forms together in the local Langlands correspondence for linear reductive groups. We extend this philosophy to a large class of covers, including those defined by Brylinski and Deligne, and explain its relation with Weissman's observation that L-packets for covers are sometimes empty.
\end{abstract}

\maketitle

\setcounter{tocdepth}{2}
\tableofcontents


\section*{Introduction}

The goal of this article is to define a notion of ``extended pure inner forms'' of a covering group and argue that it is relevant for the local Langlands program for covers.

Let us begin by describing the puzzle that motivated our consideration.

\subsection{The ``missing" $L$-packets}

In the usual local Langlands program, one takes as input a local field $F$ and a reductive group $F$-scheme $G$. To these data, one attaches the set $\Pi(G(F))$ of isomorphism classes of irreducible smooth $G(F)$-representations and the set $\Phi({}^LG)$ of L-parameters, and posits the existence of a natural map
\begin{equation}
\label{eq-local-langlands-correspondence-intro}
\LLC : \Pi(G(F)) \rightarrow \Phi({}^LG).
\end{equation}

The local Langlands program for covers requires a cohomological input $\mu$, in addition to $F$ and $G$. Traditionally, $\mu$ is defined in terms of algebraic K-theory (\emph{cf.}~\cite{MR1896177, MR3802418}). For now, let us ignore the precise meaning of $\mu$ and accept that it gives rise to a topological central extension
$$
1 \rightarrow A \rightarrow \widetilde G \rightarrow G(F) \rightarrow 1
$$
for a finite subgroup $A$ of $\complexes^{\times}$, as well as an ``L-group" $\widetilde H$. Imitating the linear situation, one posits the existence of a natural map
\begin{equation}
\label{eq-local-langlands-correspondence-cover-intro}
\LLC : \Pi(\widetilde G) \rightarrow \Phi(\widetilde H),
\end{equation}
where $\Pi(\widetilde G)$ is the set of isomorphism classes of irreducible genuine smooth $\widetilde G$-representations and $\Phi(\widetilde H)$ is the set of L-parameters defined in terms of $\widetilde H$. (The adjective ``genuine" means that $A$ acts through its inclusion in $\complexes^{\times}$.) We refer the reader to \cite{MR3802417, MR3802418, MR3802419} where foundations of this program are laid out.

The map \eqref{eq-local-langlands-correspondence-cover-intro} has been constructed by Weissman when $G$ is a split torus. He observed an intriguing phenomenon: It may \emph{not} be surjective. This stands in contrast with \eqref{eq-local-langlands-correspondence-intro}, which is expected to be surjective when $G$ is quasi-split.

The goal of our article is to relate this phenomenon to Kottwitz's philosophy of treating all extended pure inner forms of $G$ together in the formulation of \eqref{eq-local-langlands-correspondence-intro} (\emph{cf.}~\cite{MR809866, MR1485921, MR3163358}). Namely, for each basic $G$-isocrystal $\beta$, we shall construct a cover $\widetilde G_{\beta}$ of the extended pure inner form $G_{\beta}(F)$ associated to $\beta$. We expect \eqref{eq-local-langlands-correspondence-cover-intro} to fit into a family of maps
\begin{equation}
\label{eq-local-langlands-correspondence-cover-isocrystal-intro}
\LLC_{\beta} : \Pi(\widetilde G_{\beta}) \rightarrow \Phi(\widetilde H),
\end{equation}
parametrized by $\beta$. Weissman's observation above, in fact, leads to an obstruction $\Omega_{\beta}(\sigma)$ for the nonemptiness of $\LLC_{\beta}^{-1}(\sigma)$. Our main result determines the set of $\beta$ for which $\Omega_{\beta}(\sigma)$ vanishes. The key point is that this set of $\beta$ is always \emph{nonempty}, though it may not contain the trivial element. When $G$ is a torus, we prove that $\LLC_{\beta}^{-1}(\sigma)$ is indeed nonempty and finite whenever $\Omega_{\beta}(\sigma)$ vanishes.

Informally, our results indicate that the ``missing" L-packets observed by Weissman may appear on an ``extended pure inner form" $\widetilde G_{\beta}$ of the cover $\widetilde G$.

\subsection{Results}

Let us give a more precise account of the content of this article.

In the remainder of this introduction, we fix a local field $F$, a reductive group $F$-scheme $G$, and a finite subgroup $A$ of $\complexes^{\times}$ whose order is invertible in $F$.

Our first task is to define the cover $\widetilde G_{\beta}$ for an arbitrary $G$-isocrystal $\beta$, given the cohomological input $\mu$. It turns out that the K-theoretic formalism of Brylinski--Deligne (\emph{cf.}~\cite{MR1896177}) is too restrictive for this purpose. Instead, we take $\mu$ to be an \emph{\'etale metaplectic cover}, \emph{i.e.}~a morphism of pointed (higher) \'etale stacks (\emph{cf.}~\cite{MR1441006, MR3769731, zhao2022metaplectic})
$$
\deloop G \rightarrow \deloop^4A(1).
$$

The reason is as follows: An \'etale metaplectic cover $\mu$ of $G$ induces an \'etale metaplectic cover $\mu_{\beta}$ of each $G_{\beta}$, hence a cover $\widetilde G_{\beta}$ of $G_{\beta}(F)$. However, even if $\mu$ comes from algebraic K-theory, $\mu_{\beta}$ may not (\emph{cf.}~Remark \ref{rem-kazhdan-patterson-linear-difference}). In other words, \'etale metaplectic covers are necessary even if one is only interested in Brylinski--Deligne covers.

Next, we turn our attention to Langlands duality. The L-group of an \'etale metaplectic cover is defined in \cite{zhao2022metaplectic}. Let us sketch (a minor variant of) this construction, as it is important for the formulation of our main results.

This construction consists of three steps.

In the first step, we replace the pair $(G, \mu)$ by another one $(G^{\sharp}, \mu_{G^{\sharp}})$. Here, $G^{\sharp}$ is a  reductive group $F$-scheme, endowed with an \'etale metaplectic cover $\mu_{G^{\sharp}}$ which is ``as commutative as possible".

To explain the last phrase, we recall that every reductive group $F$-scheme $G$ maps to the stack quotient $G_{\abelian} := G/G_{\sconn}$, where $G_{\sconn}$ is the simply connected form of $G$. In fact, $G_{\abelian}$ is a commutative group stack and the gentlest kind of \'etale metaplectic covers are pulled back from ``$\integers$-linear" morphisms\footnote{In homotopical terms, this means morphisms of sheaves of $H\integers$-module spectra.}
$$
\deloop G_{\abelian} \rightarrow \deloop^4A(1)
$$
Such morphisms are parametrized by maps of complexes $\pi_1 G \rightarrow A[2]$. Kaletha (\emph{cf.}~\cite{kaletha2022}) has studied the covers defined by them, at least for quasi-split $G$, and reduced the Langlands correspondence for them to that for linear reductive groups.\footnote{However, Kaletha's construction of the covers is different from ours. His uses the Langlands duality for tori, and ours does not. The equivalence of these two constructions is a consequence of our results in \S\ref{sec-sharp-covers}.}

The next, and slightly less gentle, kind of \'etale metaplectic covers is pulled back from symmetric monoidal morphisms\footnote{This means morphisms of sheaves of (grouplike) $\mathbb E_{\infty}$-monoids, or equivalently of spectra.} from $\deloop G_{\abelian}$ to $\deloop^4A(1)$. The \'etale metaplectic cover $\mu_{G^{\sharp}}$ is of this kind. The passage from $(G, \mu)$ to $(G^{\sharp}, \mu_{G^{\sharp}})$ is analogous to the ``sharp cover" construction of Weissman (\emph{cf.}~\cite{MR3802418}).

The second step has to do with the subtle, but important difference between symmetric monoidal and $\integers$-linear morphisms from $\deloop G^{\sharp}_{\abelian}$ to $\deloop^4A(1)$. Namely, there is a canonical decomposition
	\begin{equation}
	\label{eq-intro-commutative-cover-canonical-deomposition}
	\mu_{G^{\sharp}} \cong \mu_{G^{\sharp}}^{(1)} + \mu_{G^{\sharp}}^{(2)}
	\end{equation}
where $\mu_{G^{\sharp}}^{(1)}$ is $2$-torsion and $\mu_{G^{\sharp}}^{(2)}$ comes from a $\integers$-linear morphism. The decomposition \eqref{eq-intro-commutative-cover-canonical-deomposition} appeared first in Gaitsgory and Lysenko's work (\emph{cf.}~\cite{MR3769731}), who used it to explain a sign occurring in the twisted geometric Satake equivalence.

The third step is the passage to the Galois side: We take $H$ to be the Langlands dual group of $G^{\sharp}$ and define a sum of $Z_H(\complexes)$-gerbes over the \'etale site of $\Spec F$
\begin{equation}
\label{eq-intro-ell-group-center}
\widetilde Z_H := \widetilde Z_H^{(1)} + \widetilde Z_H^{(2)},
\end{equation}
where $Z_H$ is the center of $H$. The summands in \eqref{eq-intro-ell-group-center} are constructed from the respective summands in \eqref{eq-intro-commutative-cover-canonical-deomposition}. The L-group $\widetilde H$ is obtained formally from $\widetilde Z_H$, by rewriting an \'etale gerbe as a Galois cocycle. In the K-theoretic context, Weissman defined the L-group as a Baer sum similar to the above (\emph{cf.}~\cite{MR3802418}). However, the decomposition \eqref{eq-intro-commutative-cover-canonical-deomposition} has no K-theoretic counterpart, so our formalism renders the situation more symmetric.

Slightly extending Kaletha's work, one can relate the Langlands duality for the ``sharp cover" $(G^{\sharp}, \mu_{G^{\sharp}})$ to that for linear reductive groups (\emph{cf.}~\S\ref{sec-duality-for-the-cocenter}). The passage from $(G, \mu)$ to $(G^{\sharp}, \mu_{G^{\sharp}})$ is more mysterious and is responsible for the ``missing" L-packets.

However, before we can go any further, we must first construct the Langlands duality for sharp covers of tori. The following result is proved in \S\ref{sec-sharp-covers}.

\begin{thmx}
\label{thmx-duality-sharp-torus}
Let $T$ be an $F$-torus equipped with a symmetric monoidal morphism $\mu : \deloop T \rightarrow \deloop^4A(1)$, defining a cover $\widetilde T$ and an L-group $\widetilde H$. There is a canonical bijection
$$
\Pi(\widetilde T) \xrightarrow{\simeq} \Phi(\widetilde H),
$$
where $\Pi(\widetilde T)$ is the set of genuine smooth characters of $\widetilde T$ and $\Phi(\widetilde H)$ is the set of L-parameters defined in terms of $\widetilde H$.
\end{thmx}

The special case of Theorem \ref{thmx-duality-sharp-torus} where $\mu$ comes from algebraic K-theory and $T$ is split is established by Weissman (\emph{cf.}~\cite[Part 3]{MR3802418}). This serves as justification for his definition of the L-group and is not at all a trivial consequence of class field theory.

Our proof of Theorem \ref{thmx-duality-sharp-torus} is independent of \emph{op.cit.}. It establishes that, in a precise sense, the decompositions \eqref{eq-intro-commutative-cover-canonical-deomposition} and \eqref{eq-intro-ell-group-center} match under Langlands duality. The main novelty in our proof is the treatment of a subtle $2$-torsion phenomenon, which explains how Gaitsgory and Lysenko's ``sign gerbe" (\emph{cf.}~\cite[\S4.8]{MR3769731}) and Weissman's meta-Galois group (\emph{cf.}~\cite[\S4]{MR3802418}) are interchanged under Langlands duality (\emph{cf.}~Corollary \ref{cor-sign-cover-dual-identification}).

We shall use Theorem \ref{thmx-duality-sharp-torus} to formulate the compatibility of the conjectural local Langlands correspondence for $(G, \mu)$ with ``central core characters". According to Weissman's vision, this is a substitute for the compatibility with central characters for the usual local Langlands correspondence.

Let us be more precise. There is a natural map from the center $Z^{\sharp}$ of $G^{\sharp}$ to the center $Z$ of $G$. (There is, however, no natural maps between $G^{\sharp}$ and $G$.) This map is compatible with their \'etale metaplectic covers, so it induces a map on the covers of their $F$-points
\begin{equation}
\label{eq-intro-induced-map-center-cover}
\widetilde Z^{\sharp} \rightarrow \widetilde Z.
\end{equation}

Given an irreducible genuine smooth representation $V$ of $\widetilde G$, the $\widetilde Z^{\sharp}$-action on $V$ through \eqref{eq-intro-induced-map-center-cover} is given by a genuine smooth character: This is the \emph{central core character} of $V$. Compatibility of the local Langlands correspondence for $(G, \mu)$ with central core characters asserts that the following diagram commutes:
\begin{equation}
\label{eq-intro-compatibility-central-character}
\begin{tikzcd}[column sep = 1.5em]
	\Pi(\widetilde G) \ar[d] \ar[r, "\LLC"] & \Phi(\widetilde H) \ar[d] \\
	\Pi(\widetilde Z^{\sharp}) \ar[r, "\simeq"] & \Phi(\widetilde H_{\abelian})
\end{tikzcd}
\end{equation}
Here, the left vertical arrow extracts the central core character, the right vertical arrow is the ``abelianization" of an L-parameter $\sigma$, and the lower horizontal equivalence is defined by Theorem \ref{thmx-duality-sharp-torus}, or rather, its mild generalization to disconnected groups.

Assuming the commutativity of \eqref{eq-intro-compatibility-central-character}, we can now explain the failure of surjectivity of $\LLC$. Let $K$ be the kernel of \eqref{eq-intro-induced-map-center-cover}. Given an L-parameter $\sigma \in \Phi(\widetilde H)$, whose abelianization corresponds to a genuine smooth character $\chi_{\sigma} : \widetilde Z^{\sharp} \rightarrow \complexes^{\times}$, if the restriction
\begin{equation}
\label{eq-intro-weissman-obstruction}
\Omega(\sigma) := \chi_{\sigma} |_K
\end{equation}
is nonzero, then the fiber of $\LLC$ at $\sigma$ is empty. This is because the central core character of any $V \in \Pi(\widetilde G)$ must annihilate $K$.

We call \eqref{eq-intro-weissman-obstruction} \emph{Weissman's obstruction}, as he first discovered it for tori (\emph{cf.}~\cite{MR2485462, MR3595494}). The main goal of this article is to generalize it to $\widetilde G_{\beta}$ for every $G$-isocrystal $\beta$ and to characterize those $\beta$ for which this obstruction vanishes.

The first task is easy, given our definition of $\widetilde G_{\beta}$. Indeed, there is a natural map $\widetilde Z^{\sharp} \rightarrow \widetilde G_{\beta}$ for every $G$-isocrystal $\beta$ with central image. This allows us to formulate the compatibility of the conjectural local Langlands correspondence \eqref{eq-local-langlands-correspondence-cover-isocrystal-intro} with central core characters, in analogy with \eqref{eq-intro-compatibility-central-character}. For each L-parameter $\sigma \in \Phi(\widetilde H)$, we may then define a character
$$
\Omega_{\beta}(\sigma) : K \rightarrow \complexes^{\times}.
$$
If $\Omega_{\beta}(\sigma)\neq 1$, then the fiber of $\LLC_{\beta}$ at $\sigma$ is empty.

Let us now state our main theorem. It will be established in \S\ref{sec-weissman-obstruction}.

\begin{thmx}
\label{thmx-kottwitz-invariant}
Let $G$ be a reductive group $F$-scheme endowed with an \'etale metaplectic cover $\mu$. Let $G^{\sharp}$ and $K$ be defined as above.
\begin{enumerate}
	\item There is a canonical exact sequence of abelian groups
	$$
	(\pi_1G^{\sharp})_{\Gal_F} \rightarrow (\pi_1G)_{\Gal_F} \xrightarrow{\gamma} \Hom(K, \complexes^{\times}) \rightarrow 1.
	$$
	
	\item For each $G$-isocrystal $\beta$ and L-parameter $\sigma$, the character $\Omega_{\beta}(\sigma)$ vanishes if and only if the Kottwitz invariant of $\beta$ maps to $\Omega(\sigma)^{-1}$ under $\gamma$.
\end{enumerate}
\end{thmx}

In particular, for any L-parameter $\sigma$, there always exists a basic $G$-isocrystal $\beta$ for which $\Omega_{\beta}(\sigma)$ vanishes and the set of isomorphism classes of such $\beta$ is a torsor under the image of $(\pi_1G^{\sharp})_{\Gal_F}$. When $G$ is a torus, we prove that the vanishing of $\Omega_{\beta}(\sigma)$ is necessary and sufficient for $\LLC_{\beta}^{-1}(\sigma)$ to be nonempty (\emph{cf.}~Proposition \ref{prop-local-langlands-correspondence-tori}).

Let us say one word about the proof of Theorem \ref{thmx-kottwitz-invariant}.

The key ingredient is the ``canonical quadratic structure" of an \'etale metaplectic cover $\mu$ with respect to the $\deloop Z$-action on $\deloop G$ (\emph{cf.}~Proposition \ref{prop-canonical-quadratic-structure}). This appears to be a fundamental piece of structure of \'etale characteristic classes, valid over an arbitrary base scheme. Informally, it expresses
$$
\mu(\mathscr E\otimes \mathscr Z) - \mu(\mathscr E) -\mu(\mathscr Z_G),
$$
for any $G$-torsor $\mathscr E$ and $Z$-torsor $\mathscr Z$ (with induced $G$-torsor $\mathscr Z_G$), in terms of an explicit bilinear expression in $\mathscr Z$ and the $G_{\abelian}$-torsor induced from $\mathscr E$. One classical manifestation of this phenomenon is the formula of Chern classes
$$
c_2(\mathscr E \otimes \mathscr L) - c_2(\mathscr E) - c_2(\mathscr L^{\oplus n}) = (n-1) \cdot c_1(\det\mathscr E) \cup c_1(\mathscr L)
$$
for any rank-$n$ vector bundle $\mathscr E$ and any line bundle $\mathscr L$.

\subsection{Conventions}

This paper uses homotopical algebra as developed by Lurie (\emph{cf.}~\cite{MR2522659, lurie2017higher}). While certain aspects of the theory of covers can be handled using traditional methods of homological algebra\footnote{For example, by taking a simplicial resolution of $\deloop G$, one can encode an \'etale metaplectic cover as an \'etale hypercocycle, which is how this notion was originally conceived of by Deligne (\emph{cf.}~\cite{MR1441006}). The $\integers$-linear part of $\mu_{G^{\sharp}}$ can also be encoded by a map of complexes over $\Spec F$, hence by a Galois hypercocycle. This is how Kaletha describes them (\emph{cf.}~\cite{kaletha2022}).}, manipulations of higher symmetric monoidal structures in this article are infeasible without Lurie's theory.

Following Lurie's convention, we refer to $\infty$-groupoids as \emph{spaces}. We invoke the equivalence between connective spectra and grouplike $\mathbb E_{\infty}$-monoids (\emph{cf.}~\cite[Remark 5.2.6.26]{lurie2017higher}) and view them as spaces with additional structure. We abbreviate ``connective $H\integers$-module spectra" as \emph{$\integers$-linear spaces}. In particular, there are forgetful functors from $\integers$-linear spaces to grouplike $\mathbb E_{\infty}$-monoids, and from grouplike $\mathbb E_{\infty}$-monoids to pointed spaces.\footnote{For a usual groupoid, \emph{i.e.}~a $1$-truncated space, a $\integers$-linear structure is the structure of a strictly commutative Picard groupoid, whereas a grouplike $\mathbb E_{\infty}$-monoid structure is the structure of a Picard groupoid.} These are essentially the only higher algebraic structures we will need.

Given a scheme $S$, we use $\deloop$ to denote the deloop functor for fppf sheaves over $S$. In particular, for a group $S$-scheme $G$, $\deloop G$ is the usual classifying stack of $G$. If $G$ is smooth and affine, then its deloops in the fppf and \'etale topologies coincide.

Given an fppf sheaf of abelian groups $\mathscr A$ over $S$ and an integer $n\ge 1$, we view the $n$-fold deloop $\deloop^n\mathscr A$ as an fppf sheaf of $\integers$-linear spaces over $S$. If $\mathscr A$ is pulled back from the small \'etale site of $S$ (\emph{e.g.}~$\mathscr A\cong A(1)$ where $A$ is a finite abelian group of invertible order), then its deloops in the fppf and \'etale topologies coincide (\emph{cf.}~\cite[0DDT]{stacks-project}). We invoke this equivalence when applying the formalism of \cite{zhao2022metaplectic}.

We shall use ``Kummer theory" extensively in the following form. For any integer $n\ge 1$, we have the coboundary $\mathbb G_m \rightarrow \deloop \integers/n(1)$ of the Kummer exact sequence. This yields a morphism in the pro-category of fppf sheaves of $\integers$-linear spaces
$$
\Psi : \mathbb G_m \rightarrow \deloop \hat{\integers}(1) := \lim_n \deloop \integers/n(1),
$$
where the formal inverse limit is taken over the divisibility poset of positive integers. We shall also frequently use the deloop $\deloop\Psi : \deloop\mathbb G_m \rightarrow \deloop^2\hat{\integers}(1)$ of $\Psi$.

\subsection{Acknowledgements}
L.S.~would like to thank Jessica Fintzen for introducing him to the topic of covers. Y.Z.~thanks Tianyi Feng, Dennis Gaitsgory, Wee Teck Gan, and Tasho Kaletha for illuminating conversations about the subject of this article.

\medskip

\section{Covers of $G_{\beta}(F)$}

Let $F$ be a local field with residue field $\residue$. We fix an algebraic closure $\overline{\residue}$ of $\residue$.

In this section, we first recall the notion of $G$-isocrystals following Kottwitz (\emph{cf.}~\cite{MR809866, MR1485921}). Then we define the cover $\widetilde G_{\beta}$ for a reductive group $F$-scheme $G$ together with an \'etale metaplectic cover $\mu$ and a $G$-isocrystal $\beta$. Next, we recall some combinatorial data associated to $\mu$ in order to define the set of L-parameters and extend Weissman's conjectural local Langlands correspondence to $\widetilde G_{\beta}$  (\emph{cf.}~Conjecture \ref{conj-enhanced-local-langlands-correspondence}).

\subsection{$G$-isocrystals}

\begin{void}
Denote by $\breve F$ the completed maximal unramified extension of $F$ determined by $\overline{\residue}$. Denote by $q$ the cardinality of $\residue$. The $q$th power Frobenius of $\overline{\residue}$ extends by functoriality to an automorphism of $\breve F$, which we denote by $\sigma$.

Denote by $X$ the prestack quotient $\Spec(\breve F)/\sigma^{\integers}$. The inclusion $F \subset \breve F$ induces a morphism of prestacks
\begin{equation}
\label{eq-unramified-quotient-to-local-field}
	X \rightarrow \Spec F.
\end{equation}

In what follows, we treat $\Spec F$ as the base scheme, so fiber products taken over $\Spec F$ will be written without $\Spec F$.
\end{void}

\begin{void}
For any affine group $F$-scheme $G$ of finite type, we write $\Isoc_G$ for the groupoid of $G$-torsors over $X$, which we refer to as \emph{$G$-isocrystals}.

Equivalently, a $G$-isocrystal $\beta$ consists of a $G$-torsor $\mathscr E$ over $\Spec\breve F$ and an isomorphism of $G$-torsors $\varphi : \sigma^*\mathscr E \xrightarrow{\simeq} \mathscr E$.

Note that any element $g$ of $G(\breve F)$ defines a $G$-isocrystal $(\mathscr E, \varphi)$, where $\mathscr E$ is the trivial $G$-torsor over $\Spec \breve F$ and $\varphi$ is multiplication by $g$. Conversely, if a $G$-isocrystal is endowed with a trivialization over $\Spec \breve F$, then $\varphi$ is given by multiplication by an element of $G(\breve F)$.
\end{void}

\begin{rem}
\label{rem-isocrystal-conjugacy-class}
If $F$ is of characteristic zero and $G$ is connected, then any $G$-torsor over $\Spec\breve F$ is trivial (\emph{cf.}~\cite[Theorem 1.9]{MR180554}). If $F$ is of characteristic $p\neq 0$ and $G$ is reductive\footnote{Our convention is that ``reductive" implies ``connected".}, then any $G$-torsor over $\Spec\breve F$ is trivial (\emph{cf.}~\cite[\S8.6]{MR244259}).

For our purposes, however, we will sometimes need the case for disconnected $G$ such as the center of a reductive group $F$-scheme.
\end{rem}

\begin{void}
Given a $G$-isocrystal $\beta$, we write $G_{\beta}$ for the group $F$-sheaf of automorphisms of $\beta$. Namely, for any $F$-algebra $R$, an $R$-point of $G_{\beta}$ is an automorphism of the pullback of $\beta$ to the prestack $X\times\Spec R$.

By \cite[\S3.3]{MR1485921}, the group $F$-sheaf $G_{\beta}$ is represented by an affine group $F$-scheme.
\end{void}

\begin{void}
\label{void-reductive-group-notation}
For a reductive group scheme $G$ over any base scheme $S$, we employ the following notation: $G_{\sconn}$ (respectively $G_{\adjoint}$) stands for the simply connected (respectively adjoint) form of $G$. Write $T$ (respectively $T_{\sconn}$, $T_{\adjoint}$) for the abstract Cartan of $G$ (respectively $G_{\sconn}$, $G_{\adjoint}$). Denote by $\Lambda$ (respectively $\check{\Lambda}$) the fppf sheaf of cocharacters (respectively characters) of $T$. We use the same notation $\Lambda_{\sconn}$, $\Lambda_{\adjoint}$, \emph{etc.}~for $T_{\sconn}$ and $T_{\adjoint}$.

Denote by $\Delta \subset \Lambda$ (respectively $\check{\Delta} \subset \check{\Lambda}$) the subsheaf of simple coroots (respectively simple roots). Thus $\Delta$ generates $\Lambda_{\sconn}$ and $\check{\Delta}$ generates the subsheaf dual to $\Lambda_{\adjoint}$.

Denote by $Z$ the center of $G$, so we have a canonical isomorphism
$$
Z \cong \Fib(\Lambda \rightarrow \Lambda_{\adjoint}) \otimes \mathbb G_m,
$$
where $\Fib$ stands for the fiber of complexes of fppf sheaves of abelian groups, and the tensor product is understood in the derived sense.

Denote by $\pi_1 G$ the quotient of fppf sheaves $\Lambda/\Lambda_{\sconn}$. We view
$$
G_{\abelian} := \pi_1 G\otimes \mathbb G_m
$$
as an fppf sheaf of Picard groupoids and refer to it as the \emph{cocenter} of $G$. (It coincides with the abelianization of $G$ when $\pi_1 G$ is torsion-free.) The identification $Z/Z_{\sconn}\cong G/G_{\sconn}$ induces a monoidal morphism
\begin{equation}
\label{eq-group-scheme-cocenter}
G \rightarrow G_{\abelian}.
\end{equation}
\end{void}

\begin{void}
\label{void-kottwitz-invariant}
For a reductive group $F$-scheme $G$, we have the Kottwitz invariant
\begin{equation}
\label{eq-kottwitz-invariant}
\Kottwitz : \Isoc_G \rightarrow (\pi_1G)_{\Gal_F},
\end{equation}
where $(\pi_1 G)_{\Gal_F}$ denotes the group of Galois coinvariants of $\pi_1 G$.\footnote{The formation of Galois coinvariants does not require the choice of an algebraic closure of $F$.}

Since \eqref{eq-kottwitz-invariant} plays a major role in this text, let us recall its definition.

Denote by $\Isoc_{G_{\abelian}}$ the space of $G_{\abelian}$-torsors over $X$, \emph{i.e.}~maps from $X$ to $\deloop G_{\abelian}$. Since $T_{\sconn}$ is reductive, every $T_{\sconn}$-torsor over $\Spec\breve F$ is trivial (\emph{cf.}~Remark \ref{rem-isocrystal-conjugacy-class}). Combining this with the fact that $\integers$ has cohomological dimension $1$, we see that $H^2(X, T_{\sconn}) \cong 0$, so the quotient map yields an equivalence
\begin{equation}
\label{eq-cocenter-isocrystal}
\Isoc_T/\Isoc_{T_{\sconn}} \xrightarrow{\simeq} \Isoc_{G_{\abelian}}.
\end{equation}

The functorial isomorphism $\pi_0(\Isoc_T) \cong (X_*T)_{\Gal_F}$ for tori (\emph{cf.}~\cite[\S2.4]{MR809866}) induces an isomorphism $\pi_0(\Isoc_{G_{\abelian}}) \cong (\pi_1 G)_{\Gal_F}$ via \eqref{eq-cocenter-isocrystal}. We set \eqref{eq-kottwitz-invariant} to be the composition
\begin{align*}
\Isoc_G &\rightarrow \Isoc_{G_{\abelian}} \\
& \rightarrow \pi_0 (\Isoc_{G_{\abelian}}) \xrightarrow{\simeq} (\pi_1 G)_{\Gal_F},
\end{align*}
where the first map is defined by functoriality with respect to \eqref{eq-group-scheme-cocenter}.
\end{void}

\begin{void}
We keep the assumption that $G$ is reductive.

Recall that a $G$-isocrystal $\beta$ is \emph{basic} if its induced $G_{\adjoint}$-isocrystal is the pullback of a $G_{\adjoint}$-torsor over $\Spec F$ (\emph{cf.}~\cite[\S4.5, \S5.1]{MR809866}). Thus, if $\beta$ is basic, then $G_{\beta}$ is an inner form of $G$. Inner forms arising in this manner are called \emph{extended pure inner forms} of $G$. Denote by $\Basic_G$ the full subgroupoid of $\Isoc_G$ consisting of basic $G$-isocrystals.

According to \cite[Proposition 5.6]{MR809866}, \eqref{eq-kottwitz-invariant} induces a bijection
\begin{equation}
\label{eq-basic-isocrystal-kottwitz-invariant-bijection}
\pi_0(\Basic_G) \xrightarrow{\simeq} (\pi_1 G)_{\Gal_F}.
\end{equation}
\end{void}

\subsection{Construction of covers}
\label{sec-construction-of-covers}

\begin{void}
Let $G$ be an affine group $F$-scheme of finite type and $A$ be a finite abelian group whose order is invertible in $F$.

For each $n\in\integers$, write $A(n)$ for the corresponding Tate twist of $A$, viewed as an \'etale sheaf of finite abelian groups over $\Spec F$.
\end{void}

\begin{void}
Denote by $\Maps_e(\deloop G, \deloop^4A(1))$ the space of rigidified morphisms $\deloop G \rightarrow \deloop^4A(1)$, \emph{i.e.}~morphisms of pointed $F$-stacks. It admits a $\integers$-linear structure induced from the abelian group structure on $A(1)$.

For a topological group $K$, we refer to a topological central extension
$$
1 \rightarrow A \rightarrow \widetilde K \rightarrow K \rightarrow 1,
$$
where $\widetilde K \rightarrow K$ is a local homeomorphim, as a \emph{cover} of $K$. The collection of covers of $K$ form a $\integers$-linear groupoid under Baer sum, which we denote by $\Cov(K, A)$.

Let us equip $G(F)$ with the topology inherited from $F$. The construction of \cite[\S2.1]{zhao2022metaplectic} yields a $\integers$-linear functor
\begin{equation}
\label{eq-construction-of-covers}
\int_F : \Maps_e(\deloop G, \deloop^4A(1)) \rightarrow \Cov(G(F), A).
\end{equation}
\end{void}

\begin{void}
\label{void-isocrystal-translation-functor}
For any $G$-isocrystal $\beta$, we shall construct a functor
\begin{equation}
\label{eq-isocrystal-translation-functor}
\Translation_{\beta} : \Maps_e(\deloop G, \deloop^4A(1)) \rightarrow \Maps_e(\deloop G_{\beta}, \deloop^4A(1)),
\end{equation}
to be conceived of as ``translation by $\beta$".

If $G_{\beta}$ is of finite type, then by composing \eqref{eq-isocrystal-translation-functor} with the functor \eqref{eq-construction-of-covers} applied to $G_{\beta}$, we obtain a functor
\begin{equation}
\label{eq-construction-of-covers-isocrystal}
	\int_{F, \beta} : \Maps_e(\deloop G, \deloop^4A(1)) \rightarrow \Cov(G_{\beta}(F), A).
\end{equation}

Under \eqref{eq-construction-of-covers-isocrystal}, every rigidified morphism $\deloop G \rightarrow \deloop^4A(1)$ defines a cover $\widetilde G_{\beta}$ of $G_{\beta}(F)$.
\end{void}

\begin{void}\emph{Construction of $\Translation_{\beta}$.}
Let us view $\beta$ as a morphism $X\rightarrow \deloop G$. Since $G_{\beta}$ is the group $F$-sheaf of its automorphisms, $\beta$ extends to a morphism of fppf stacks
\begin{equation}
\label{eq-isocrystal-classifying-stack-translation-map}
X\times \deloop G_{\beta} \rightarrow \deloop G.
\end{equation}
More precisely, there is a natural morphism of group $X$-sheaves $X\times G_{\beta} \rightarrow X\times_{\deloop G}X$, where the target is the fiber product of $\beta$ with itself, and \eqref{eq-isocrystal-classifying-stack-translation-map} is obtained as its deloop.

Given a rigidified morphism $\mu : \deloop G \rightarrow \deloop^4A(1)$, the pullback of $\mu$ along \eqref{eq-isocrystal-classifying-stack-translation-map} is a morphism $X\times\deloop G_{\beta} \rightarrow \deloop^4A(1)$ whose restriction along the neutral section $e : X \rightarrow X\times\deloop G_{\beta}$ is isomorphic to $\beta^*\mu$. Thus, sending $\mu$ to the difference $\mu - p^*\beta^*\mu$, where $p : X\times\deloop G_{\beta} \rightarrow X$ is the projection, defines a functor
\begin{equation}
\label{eq-etale-metaplectic-cover-subtraction}
\Maps_e(\deloop G, \deloop^4A(1)) \rightarrow \Maps_e(X\times \deloop G_{\beta}, \deloop^4A(1)),
\end{equation}
where the target is the space of maps $X\times\deloop G_{\beta} \rightarrow \deloop^4A(1)$ rigidified along $e$.

Thanks to Lemma \ref{lem-weil-galois-torsion} below, pullback along the projection $X\times \deloop G_{\beta} \rightarrow \deloop G_{\beta}$ induces an isomorphism
\begin{equation}
\label{eq-isoc-classifying-stack-cohomology}
\Maps_e(\deloop G_{\beta}, \deloop^4A(1)) \xrightarrow{\simeq} \Maps_e(X\times\deloop G_{\beta}, \deloop^4A(1)).
\end{equation}

The desired functor \eqref{eq-isocrystal-translation-functor} is the composition of \eqref{eq-etale-metaplectic-cover-subtraction} with the inverse of \eqref{eq-isoc-classifying-stack-cohomology}.
\end{void}

\begin{lem}
\label{lem-weil-galois-torsion}
For any $F$-scheme $S$, pullback along the projection $S\times X \rightarrow S$ induces an isomorphism of \'etale cochains
\begin{equation}
\label{eq-weil-galois-torsion}
\Gamma(S, A(1)) \xrightarrow{\simeq} \Gamma(S\times X, A(1)).
\end{equation}
\end{lem}
\begin{proof}
Denote by $\nu : \Spec \check F \rightarrow \Spec F$ the natural map. The complex $\nu_*A(1)$ is endowed with an automorphism $\sigma^*$ defined by pullback along $\sigma : \Spec \check F \rightarrow \Spec \check F$. We claim that it is sufficient to identify the fiber of
\begin{equation}
\label{eq-automorphism-derived-invariant-fiber}
\sigma^* - \id : \nu_* A(1) \rightarrow \nu_* A(1)
\end{equation}
with $A(1)$, along the unit map $A(1) \rightarrow \nu_* A(1)$.

Indeed, the complex $\Gamma(S\times X, A(1))$ is the (derived) $\integers$-invariants of the complex $\Gamma(S\times\Spec\breve F, A(1))$, with $1\in\integers$ acting by $\sigma^*$. On the other hand, base change (\emph{cf.}~\cite[0F1I]{stacks-project}) yields an isomorphism
$$
\Gamma(S\times\Spec\breve F, A(1)) \xrightarrow{\simeq} \Gamma(S, \nu_*A(1)).
$$
Hence, $\Gamma(S\times X, A(1))$ is identified with the image under $\Gamma(S, -)$ of the fiber of \eqref{eq-automorphism-derived-invariant-fiber}. If the latter is identified with $A(1)$ along the unit map, then $\Gamma(S\times X, A(1))$ is identified with $\Gamma(S, A(1))$ along the pullback map.

We now identify the fiber of \eqref{eq-automorphism-derived-invariant-fiber}. Let $F^{\unr}$ denote the maximal unramified extension of $F$ determined by $\bar{\residue}$, so $\breve F$ is the completion of $F^{\unr}$. Pulling back along $\Spec \breve F \rightarrow \Spec F^{\unr}$ induces an equivalence of \'etale sites, so we may replace $\Spec \breve F$ by $\Spec F^{\unr}$.

The map $\Spec F^{\unr} \rightarrow \Spec F$ is a $\hat{\integers}$-torsor and the desired isomorphism can be verified at a geometric point of $\Spec F$. We thus reduce to the following assertion: For any torsion $\hat{\integers}$-module $M$, its $\hat{\integers}$-invariants coincide with its $\integers$-invariants along the natural map $\integers \rightarrow \hat{\integers}$. This follows from the computation of group cohomology of $\hat{\integers}$ (\emph{cf.}~\cite[XIII, \S1]{MR554237}).
\end{proof}

\begin{rem}
The proof of Lemma \ref{lem-weil-galois-torsion} applies when $A(1)$ is replaced by any torsion \'etale sheaf over $\Spec F$ of order invertible in $F$. It expresses the fact that the morphism \eqref{eq-unramified-quotient-to-local-field} induces, universally, an isomorphism on \'etale cohomology with such coefficients.
\end{rem}

\subsection{Combinatorics of covers}

\begin{void}
Let $G$ be a reductive group scheme over a base scheme $S$. Let $A$ be a finite abelian group whose order is invertible on $S$. We employ the notation of \S\ref{void-reductive-group-notation}.

Denote by $\Maps_e(\deloop G, \deloop^4A(1))$ the space of rigidified morphisms $\deloop G \rightarrow \deloop^4A(1)$. We shall recall certain combinatorial data associated to it.
\end{void}

\begin{void}
\label{void-strict-weyl-invariant-forms}
Given a quadratic form $Q : \Lambda \rightarrow A(-1)$, we write
$$
b : \Lambda \otimes \Lambda \rightarrow A(-1)
$$
for the associated symmetric form, sending $\lambda_1, \lambda_2 \in \Lambda$ to
$$
b(\lambda_1, \lambda_2) := Q(\lambda_1 + \lambda_2) - Q(\lambda_1) - Q(\lambda_2).
$$

We say that $Q$ is \emph{strictly Weyl-invariant} if the equality
\begin{equation}
\label{eq-strict-weyl-invariance}
b(\alpha, \lambda) = Q(\alpha) \langle \check{\alpha}, \lambda\rangle
\end{equation}
holds for any $\lambda \in \Lambda$ and any simple coroot $\alpha \in \Delta$.

The right-hand-side of \eqref{eq-strict-weyl-invariance} makes sense for any $\lambda \in \Lambda_{\adjoint}$, if we understand $\langle\cdot, \cdot\rangle$ as the canonical pairing between the root lattice and $\Lambda_{\adjoint}$. Thus it extends to a bilinear form
$$
b_1 : \Lambda_{\sconn} \otimes \Lambda_{\adjoint} \rightarrow A(-1).
$$
\end{void}

\begin{void}\emph{The pairing $b_2$.}
The coincidence between $b$ and $b_1$ over $\Lambda_{\sconn}\otimes \Lambda$ implies that their adjoints make the following diagram commute:
\begin{equation}
\label{eq-strict-weyl-invariant-form-extension}
\begin{tikzcd}[column sep = 1em]
	\Lambda \ar[r]\ar[d] & \SHom(\Lambda, A(-1)) \ar[d] \\
	\Lambda_{\adjoint} \ar[r] & \SHom(\Lambda_{\sconn}, A(-1))
\end{tikzcd}
\end{equation}

Taking fibers of the vertical maps, we obtain a map
\begin{equation}
\label{eq-center-cocenter-pairing-adjoint}
\Fib(\Lambda \rightarrow \Lambda_{\adjoint}) \rightarrow \SHom(\pi_1G, A(-1)).
\end{equation}
Denote by $b_2$ the bilinear pairing obtained from \eqref{eq-center-cocenter-pairing-adjoint} by passing to the adjoint
\begin{equation}
\label{eq-center-cocenter-pairing}
b_2 : \pi_1 G\otimes \Fib(\Lambda \rightarrow \Lambda_{\adjoint}) \rightarrow A(-1).
\end{equation}
\end{void}

\begin{void}
\label{void-sharp-lattices}
Denote by $\Lambda^{\sharp} \subset \Lambda$ the kernel of $b$. Denote by $\Lambda^{\sharp}_{\sconn} \subset \Lambda_{\sconn}$, $\Lambda^{\sharp}_{\adjoint} \subset \Lambda_{\adjoint}$ the kernels of $b_1$. Write $\check{\Lambda}^{\sharp}$ for the dual of $\Lambda^{\sharp}$.

For each simple coroot $\alpha \in \Delta$, we shall also write
$$
\alpha^{\sharp} := \ord(Q(\alpha))\cdot\alpha,\quad \check{\alpha}^{\sharp} := \ord(Q(\alpha))^{-1}\cdot\check{\alpha},
$$
where $\ord(Q(\alpha))$ denotes the order of $Q(\alpha) \in A(-1)$. The set $\Delta^{\sharp}$ of $\alpha^{\sharp}$ (respectively $\check{\Delta}^{\sharp}$ of $\check{\alpha}^{\sharp}$) forms a subsheaf of $\Lambda^{\sharp}$ (respectively $\check{\Lambda}^{\sharp}$).

Observe that $\Lambda_{\sconn}^{\sharp}$ is the span of $\Delta^{\sharp}$: An element $\sum_{\alpha \in \Delta} d_{\alpha}\cdot\alpha$ of $\Lambda_{\sconn}$ belongs to $\Lambda^{\sharp}_{\sconn}$ if and only if it pairs to zero under $b_1$ against each fundamental coweight $\omega_{\alpha}$, and this occurs if and only if $d_{\alpha}\cdot Q(\alpha) = 0$ for each $\alpha\in\Delta$. Likewise, $\Lambda^{\sharp}_{\adjoint}$ is dual to the span of $\check{\Delta}^{\sharp}\subset\check{\Lambda}^{\sharp}$.

Moreover, the quadruple 
\begin{equation}
\label{eq-metaplectic-dual-root-data}
(\Delta^{\sharp} \subset \Lambda^{\sharp}, \check{\Delta}^{\sharp} \subset \check{\Lambda}^{\sharp})
\end{equation}
is a locally constant \'etale sheaf of based root data over $S$. In particular, \eqref{eq-metaplectic-dual-root-data} is the root data of a reductive group $F$-scheme $G^{\sharp}$ with sheaf of cocharacters $\Lambda^{\sharp}$. We decorate with $(\cdot)^{\sharp}$ all the objects associated to $G^{\sharp}$ in \S\ref{void-reductive-group-notation}.
\end{void}

\begin{void}
\label{void-etale-metaplectic-cover-fiber-sequence}
Write $\SMaps_e(\deloop G, \deloop^4A(1))$ for the \'etale sheaf over $S$ whose sections over an $S$-scheme $S_1$ are rigidified morphisms $\deloop G\times_S S_1 \rightarrow \deloop^4A(1)$.

By \cite[Proposition 5.1.11]{zhao2022metaplectic}, there is a canonical fiber sequence
\begin{equation}
\label{eq-etale-metaplectic-cover-fiber-sequence}
\SHom_{\integers}(\pi_1 G, \deloop^2 A) \rightarrow \SMaps_e(\deloop G, \deloop^4A(1)) \rightarrow \SQuad(\Lambda, A(-1))_{\strict},
\end{equation}
where $\SHom_{\integers}$ denotes the \'etale sheaf of $\integers$-linear morphisms and $\SQuad(\Lambda, A(-1))_{\strict}$ denotes the \'etale sheaf of strictly Weyl-invariant quadratic forms on $\Lambda$. The first map in \eqref{eq-etale-metaplectic-cover-fiber-sequence} is defined by tensoring with $\deloop\Psi$ and pulling back along $\deloop G \rightarrow \deloop G_{\abelian}$.

In particular, to each rigidified morphism $\mu : \deloop G \rightarrow \deloop^4A(1)$, we may associate a strictly Weyl-invariant quadratic form $Q$ and pairings $b$, $b_1$, $b_2$ as well as the \'etale sheaf of based root data \eqref{eq-metaplectic-dual-root-data}.
\end{void}

\begin{prop}
\label{prop-sharp-center-symmetric-monoidal}
Let $\mu$ be a rigidified morphism $\deloop G \rightarrow \deloop^4A(1)$. The restriction of $\mu$ to $\deloop Z^{\sharp}$ canonically lifts to an $\mathbb E_{\infty}$-monoidal morphism
\begin{equation}
\label{eq-sharp-center-symmetric-monoidal}
\mu_{Z^{\sharp}} : \deloop Z^{\sharp} \rightarrow \deloop^4A(1),
\end{equation}
equipped with a trivialization over $\deloop Z^{\sharp}_{\sconn}$.
\end{prop}
\begin{proof}
In \cite[\S6.1]{zhao2022metaplectic}, we construct from $\mu$ a canonical $\mathbb E_{\infty}$-monoidal morphism $\mu_{T^{\sharp}} : \deloop T^{\sharp} \rightarrow \deloop^4A(1)$ endowed with a trivialization over $\deloop T^{\sharp}_{\sconn}$. It is enough to identify the restriction of $\mu_{T^{\sharp}}$ to $\deloop Z^{\sharp}$ with the restriction of $\mu$.

To do this, we recall that $\mu_{T^{\sharp}}$ is constructed, \'etale locally over $\Spec F$, by choosing a Borel subgroup $B\subset G$ and restricting $\mu$ to $\deloop B$. The latter descends to $\deloop T$ and $\mu_{T^{\sharp}}$ is its pullback to $\deloop T^{\sharp}$. This provides an identification between the restrictions of $\mu_{T^{\sharp}}$ and $\mu$ to $\deloop Z^{\sharp}$, which \emph{a priori} depends on $B$. The independence is proved as in \cite[\S5.2.6]{zhao2022metaplectic}.
\end{proof}

\begin{void}
\label{void-sharp-cocenter-induced-cover}
Note that $Z^{\sharp}/Z^{\sharp}_{\sconn}$ is canonically identified with $G^{\sharp}_{\abelian}$, so the $\mathbb E_{\infty}$-monoidal morphism \eqref{eq-sharp-center-symmetric-monoidal} together with its trivialization over $\deloop Z^{\sharp}_{\sconn}$ defines an $\mathbb E_{\infty}$-monoidal morphism
\begin{equation}
\label{eq-sharp-cocenter-induced-cover}
\mu_{G^{\sharp}_{\abelian}} : \deloop G^{\sharp}_{\abelian} \rightarrow \deloop^4A(1).
\end{equation}

The sheaf of $\mathbb E_{\infty}$-monoidal morphisms from $\deloop G^{\sharp}_{\abelian}$ to $\deloop^4A(1)$ fits into a canonical fiber sequence
\begin{equation}
\label{eq-sharp-cocenter-fiber-sequence}
\SHom_{\integers}(\pi_1G^{\sharp}, \deloop^2A) \rightarrow \SMaps_{\mathbb E_{\infty}}(\deloop G^{\sharp}_{\abelian}, \deloop^4A(1)) \rightarrow \SHom(\pi_1G^{\sharp}, A(-1)_{2\tors}),
\end{equation}
where $A(-1)_{2\tors}$ denotes the subsheaf of $2$-torsion elements of $A(-1)$. Indeed, this follows from expressing $\deloop G^{\sharp}_{\abelian}$ as the cofiber of $\deloop T_{\sconn}^{\sharp} \rightarrow \deloop T^{\sharp}$ and reducing to the analogous statement for tori (\emph{cf.}~\cite[Proposition 4.6.2]{zhao2022metaplectic}).

In particular, it follows from \emph{op.cit.}~that the image of $\mu_{G_{\abelian}^{\sharp}}$ along the second map of \eqref{eq-sharp-cocenter-fiber-sequence} is the restriction of $Q$ to $\Lambda^{\sharp}$, which is valued in $A(-1)_{2\tors}$ and annihilates $\Lambda^{\sharp}_{\sconn}$.
\end{void}

\subsection{The local Langlands correspondence}
\label{sec-local-langlands-correspondence}

\begin{void}
\label{void-local-langlands-correspondence-context}
We specialize to the case where $G$ is a reductive group $F$-scheme. Fix a finite abelian group $A$ whose order is invertible in $F$, equipped with an injective character
$$
\zeta : A \rightarrow \complexes^{\times}.
$$
Note that $\zeta$ identifies $A$ with the subgroup $\mu_N(\complexes)$ for $N := |A|$.

Let $\mu$ be a rigidified morphism $\deloop G \rightarrow \deloop^4A(1)$. We shall recall Weissman's conjectural local Langlands correspondence for the cover of $G(F)$ defined by $\mu$ and explain its extension to extended pure inner forms of $G$.
\end{void}

\begin{void}
\label{void-local-langlands-correspondence-cover-construction}
For each $\beta \in \Isoc_G$, we apply the construction functor \eqref{eq-construction-of-covers-isocrystal} to $\mu$ to obtain a cover
$$
\widetilde G_{\beta} := \int_{F, \beta} \mu.
$$

Denote by $\Pi(\widetilde G_{\beta})$ the set of isomorphism classes of irreducible $\zeta$-genuine smooth representations of $\widetilde G_{\beta}$. Being ``$\zeta$-genuine" means that $A$ acts through the character $\zeta$.

As above, we omit the subscript $\beta$ when it is the trivial $G$-isocrystal.
\end{void}

\begin{void}
\label{void-metaplectic-dual-data}
On the other hand, the rigidified morphism $\mu$ defines the reductive group $F$-scheme $G^{\sharp}$ (\emph{cf.}~\S\ref{void-sharp-lattices}) and the $\mathbb E_{\infty}$-monoidal morphism $\mu_{G_{\abelian}^{\sharp}}$ (\emph{cf.}~\S\ref{void-sharp-cocenter-induced-cover}). The Galois side of the local Langlands correspondence depends only on $(G^{\sharp}, \mu_{G_{\abelian}^{\sharp}})$, as opposed to $(G, \mu)$.

Denote by $H$ the Langlands dual of $G^{\sharp}$, viewed as a locally constant \'etale sheaf of pinned split reductive group $\integers$-schemes. In particular, $H$ is equipped with a Killing pair $T_H\subset B_H \subset H$, where $T_H$ has sheaf of characters $\Lambda^{\sharp}$.
\end{void}

\begin{void}
\label{void-sharp-cocenter-fiber-sequence-splitting}
We shall construct a canonical splitting of the fiber sequence \eqref{eq-sharp-cocenter-fiber-sequence}. The idea of this construction is originally due to Gaitsgory and Lysenko (\emph{cf.}~\cite[\S4.8]{MR3769731}).

If $A$ has odd degree, then $A(-1)_{2\tors}$ vanishes and \eqref{eq-sharp-cocenter-fiber-sequence} trivially splits.

If $A$ has even degree\footnote{By the assumption that $|A|$ is invertible in $F$, this implies that $F$ has characteristic $\neq 2$.}, then $\zeta$ identifies $A(-1)_{2\tors}$ with $\integers/2$. To split \eqref{eq-sharp-cocenter-fiber-sequence}, we associate to each character $\epsilon : \pi_1G^{\sharp} \rightarrow \integers/2$ the $\mathbb E_{\infty}$-monoidal morphism
\begin{equation}
\label{eq-sign-cover-composition-with-character}
\deloop G^{\sharp}_{\abelian} \xrightarrow{\epsilon \otimes \deloop\Psi} \deloop^2\{\pm 1\} \xrightarrow{\sgn} \deloop^4\{\pm 1\}^{\otimes 2} \rightarrow \deloop^4A(1),
\end{equation}
where $\sgn$ is the $\mathbb E_{\infty}$-monoidal morphism constructed below.
\end{void}

\begin{void}\emph{Construction of $\sgn$.}
\label{void-sign-cover}
We work over the base scheme $S := \Spec \integers[\frac{1}{2}]$. The \'etale sheaf $\SMaps_{\mathbb E_{\infty}}(\deloop^2\{\pm 1\}, \deloop^4\{\pm 1\}^{\otimes 2})$ is the fiber of the map
\begin{equation}
\label{eq-multiplicative-cover-squaring-map}
\SMaps_{\mathbb E_{\infty}}(\deloop\mathbb G_m, \deloop^4\{\pm 1\}^{\otimes 2}) \rightarrow \SMaps_{\mathbb E_{\infty}}(\deloop\mathbb G_m, \deloop^4\{\pm 1\}^{\otimes 2})
\end{equation}
given by pullback along $(\cdot)^2 : \deloop\mathbb G_m \rightarrow \deloop\mathbb G_m$.

On the other hand, the functor of taking loop spaces and applying $\SMaps_e(\mathbb G_m, \cdot)$ yields an equivalence (\emph{cf.}~\cite[Proposition 4.6.6]{zhao2022metaplectic})
\begin{equation}
\label{eq-multiplicative-cover-parametrization}
\SMaps_{\mathbb E_{\infty}}(\deloop\mathbb G_m, \deloop^4\{\pm 1\}^{\otimes 2})\xrightarrow{\simeq}\SMaps_{\mathbb E_{\infty}}(\integers, \deloop^2\{\pm 1\}).
\end{equation}
This induces an identification of the fiber of \eqref{eq-multiplicative-cover-squaring-map}
\begin{equation}
\label{eq-sign-cover-equivalence}
\SMaps_{\mathbb E_{\infty}}(\deloop^2\{\pm 1\}, \deloop^4\{\pm 1\}^{\otimes 2})\xrightarrow{\simeq} \SMaps_{\mathbb E_{\infty}}(\integers/2, \deloop^2\{\pm 1\}).
\end{equation}

Note that an $\mathbb E_{\infty}$-monoidal morphism $\integers/2 \rightarrow \deloop^2\{\pm 1\}$ is equivalent to a symmetric monoidal extension of $\integers/2$ by $\deloop\{\pm 1\}$. We define $\sgn : \deloop^2\{\pm 1\} \rightarrow \deloop^4\{\pm 1\}^{\otimes 2}$ to be the $\mathbb E_{\infty}$-monoidal morphism whose image under \eqref{eq-sign-cover-equivalence} is the trivial monoidal extension of $\integers/2$ by $\deloop\{\pm 1\}$, with commutativity constraint specified by the pairing
$$
\integers/2 \otimes \integers/2 \rightarrow \{\pm 1\},\quad a, b\mapsto (-1)^{ab}.
$$
\end{void}

\begin{void}
Under the splitting of \eqref{eq-sharp-cocenter-fiber-sequence}, we may write $\mu_{G^{\sharp}_{\abelian}}$ as a sum
\begin{equation}
\label{eq-sharp-cover-decomposition}
\mu_{G^{\sharp}_{\abelian}} \xrightarrow{\simeq} \mu_{G^{\sharp}_{\abelian}}^{(1)} + \mu_{G^{\sharp}_{\abelian}}^{(2)},
\end{equation}
where $\mu_{G^{\sharp}_{\abelian}}^{(1)}$ is the composition \eqref{eq-sign-cover-composition-with-character} applied to the character $\epsilon$ defined by the restriction of $Q$ to $\Lambda^{\sharp}$ (\emph{cf.}~\S\ref{void-sharp-cocenter-induced-cover}) and $\mu_{G^{\sharp}_{\abelian}}^{(2)}$ is defined by a $\integers$-linear morphism
\begin{equation}
\label{eq-sharp-cover-linear-component}
\pi_1G^{\sharp} \rightarrow \deloop^2 A.
\end{equation}
\end{void}

\begin{void}
\label{void-group-cochain-as-space}
We shall now convert the data $\epsilon : \pi_1G^{\sharp} \rightarrow \integers/2$ and \eqref{eq-sharp-cover-linear-component} to the Galois side. For this, it helps to introduce a bit of formalism.

Given a pro-space $X = \lim_{i\in I} X_i$ and a sheaf of abelian groups $\mathscr A$ over some $X_i$, we write
$$
\Gamma(X, \mathscr A[n]) := \colim_{j\in I_{/i}} \Gamma(X_j, \mathscr A[n]),
$$
where the transition maps are given by pullbacks. We refer to objects of the $\integers$-linear space underlying $\Gamma(X, \mathscr A[2])$ as \emph{$\mathscr A$-gerbes} over $X$. Thus, the total space of an $\mathscr A$-gerbe over $X$ is a pro-space over $X$.

Givan pro-group $\Sigma = \lim_{i\in I}\Sigma_i$, we may apply the above formalism to the pro-space $*/\Sigma := \lim_{i\in I} */\Sigma_i$. Any $\Sigma_i$-module $\mathscr A$ may be regarded as a sheaf of abelian groups over $*/\Sigma_i$, and we use $\cycle^n(\Sigma, \mathscr A)$ to denote the $\integers$-linear space underlying $\Gamma(*/\Sigma, \mathscr A[n])$.  In particular, we have an isomorphism whenever $0\le m\le n$:
$$
\pi_m \cycle^n(\Sigma, \mathscr A) \xrightarrow{\simeq} H^{n - m}(\Sigma, \mathscr A),
$$
the right-hand-side being the continuous group cohomology of $\Sigma$ with coefficients in $\mathscr A$.
\end{void}

\begin{void}
\label{void-weil-group-classifying-space}
Fix an algebraic closure $\bar F$ of $F$ lifting $\bar{\residue}$.

Denote by $W_F$ the Weil group of $F$, which we view as a pro-group $W_F := \lim \Sigma$, where the formal limit is taken over discrete quotients $W_F \twoheadrightarrow \Sigma$.

By taking fibers at the geometric point $\Spec\bar F$, we may view the \'etale sheaf $H(\complexes)$ as a group with a $W_F$-action through a finite quotient $\Sigma$. In particular, $H(\complexes)$ may be viewed as a sheaf of groups over $*/\Sigma$.

Analogously, $Z_H(\complexes)$ may be viewed as a sheaf of abelian groups over $*/\Sigma$. The formalism of \S\ref{void-group-cochain-as-space} allows us to make sense of $Z_H(\complexes)$-gerbes over $*/W_F$.
\end{void}

\begin{void}\emph{The meta-Weil group.}
Consider the central extension
\begin{equation}
\label{eq-hilbert-cover}
1 \rightarrow \{\pm 1\} \rightarrow \widetilde F^{\times}_{\Hilbert} \rightarrow F^{\times} \rightarrow 1
\end{equation}
defined by the quadratic Hilbert symbol $\{\cdot, \cdot\}$ as cocycle, \emph{i.e.}~we have $\widetilde F^{\times}_{\Hilbert} := F^{\times} \times \{\pm 1\}$ with the group structure $(a, 1)\cdot (b, 1) := (ab, \{a, b\})$.

The \emph{meta-Weil group} is defined to be the pullback of \eqref{eq-hilbert-cover} along the Artin reciprocity map $W_F \rightarrow F^{\times}$ (\emph{cf.}~\cite[\S4]{MR3802418})
\begin{equation}
\label{eq-meta-weil-group}
1 \rightarrow \{\pm 1\} \rightarrow \widetilde W_F \rightarrow W_F \rightarrow 1.
\end{equation}

Taking classifying spaces, \eqref{eq-meta-weil-group} yields a $\{\pm 1\}$-gerbe over $*/W_F$.
\end{void}

\begin{void}
\label{void-metaplectic-dual-etale-gerbes}
Denote by $\widetilde Z_H^{(1)}$ the $Z_H(\complexes)$-gerbe over $*/W_F$ induced from \eqref{eq-meta-weil-group} along the dual $\epsilon^{\vee} : \{\pm 1\} \rightarrow Z_H(\complexes)$ of the character $\epsilon$.

Denote by $\widetilde Z_H^{(2)}$ the $Z_H(\complexes)$-gerbe over $*/W_F$ defined by composing \eqref{eq-sharp-cover-linear-component} with $\zeta$. Here, we invoked the passage from \'etale $Z_H(\complexes)$-gerbes over $\Spec F$ to $Z_H(\complexes)$-gerbes over $*/W_F$.

Consider the sum of $Z_H(\complexes)$-gerbes
\begin{equation}
\label{eq-metaplectic-ell-group-center}
\widetilde Z_H := \widetilde Z_H^{(1)} + \widetilde Z_H^{(2)}.
\end{equation}
\end{void}

\begin{void}
Inducing (the total space of) $\widetilde Z_H$ along the morphism $Z_H(\complexes) \rightarrow H(\complexes)$ of sheaves of groups over $*/W_F$, we obtain a pro-space $\widetilde H$ over $*/W_F$.

By an \emph{L-parameter}, we shall mean a section of the projection $\widetilde H \rightarrow */W_F$.

For any standard parabolic subgroup $P_H \subset H$ with standard Levi subgroup $M_H \subset P_H$, one may induce $\widetilde Z_H$ along $Z_H(\complexes) \rightarrow M_H(\complexes) \rightarrow P_H(\complexes)$ to obtain pro-spaces $\widetilde M_H$ and $\widetilde P_H$ over $*/W_F$. An L-parameter $*/W_F \rightarrow \widetilde H$ is called \emph{semisimple} if, whenever it factors through $\widetilde P_H$ for some standard parabolic subgroup $P_H$, it factors through $\widetilde M_H$.

Denote by $\Phi(\widetilde H)$ the set of isomorphism classes of semisimple L-parameters. We shall often refer to elements of $\Phi(\widetilde H)$ simply as ``L-parameters".
\end{void}

\begin{rem}
Let us remark on why we define L-parameters in terms of $\widetilde H$ rather than the more concrete definition in terms of an ``$L$-group". The discrepancy has to do with the choice of base points.

Indeed, choosing a base point of $\widetilde Z_H$ and taking loop spaces, we obtain an extension of pro-groups
$$
1 \rightarrow Z_H(\complexes) \rightarrow \Omega(\widetilde Z_H) \rightarrow W_F \rightarrow 1.
$$

Likewise, the induced base point of $\widetilde H$ gives rise to an extension $\Omega(\widetilde H)$ of $W_F$ by $H(\complexes)$, which may be considered as the ``$L$-group".

An L-parameter $*/W_F \rightarrow \widetilde H$ is (non-canonically) isomorphic to one which preserves the base point, which is equivalent to a section $W_F \rightarrow \Omega(\widetilde H)$. Note that $\Omega(\widetilde H)$ is the pullback of some extension $\Omega(\widetilde H)_f$ of a discrete quotient of $W_F$ by $H(\complexes)$. The composite 
$$
W_F \rightarrow \Omega(\widetilde H) \twoheadrightarrow \Omega(\widetilde H)_f
$$
is a morphism of pro-groups, so it factors through a discrete quotient of $W_F$. Thus, we recover the classical notion of an L-parameter (or more precisely, a Weil parameter).\footnote{Weissman provides a different recipe for restoring the independence of base points, by explicitly lifting the ``category of L-groups" to a $2$-category (\emph{cf.}~\cite[\S5.1]{MR3802418}).}
\end{rem}

\begin{void}
The following is Weissman's version of the local Langlands correspondence (\emph{cf.}~\cite[Conjecture 0.1]{MR3802418}).
\end{void}

\begin{conj}[Weissman]
\label{conj-local-langlands-correspondence}
There is a natural finite-to-one map
\begin{equation}
\label{eq-local-langlands-correspondence}
\LLC : \Pi(\widetilde G) \rightarrow \Phi(\widetilde H).
\end{equation}
\end{conj}

\begin{void}
Let us include the covers $\widetilde G_{\beta}$ in the formulation of Conjecture \ref{conj-local-langlands-correspondence}.
\end{void}

\begin{conj}
\label{conj-enhanced-local-langlands-correspondence}
For each $\beta \in \Basic_G$, there is a natural finite-to-one map
\begin{equation}
\label{eq-local-langlands-correspondence-isocrystal}
\LLC_{\beta} : \Pi(\widetilde G_{\beta}) \rightarrow \Phi(\widetilde H).
\end{equation}
\end{conj}

\medskip

\section{Sharp covers}
\label{sec-sharp-covers}

Let $F$ be a local field with a fixed algebraic closure $\bar F$.

The goal of this subsection is to construct the local Langlands correspondence for sharp tori (\emph{cf.}~\eqref{eq-torus-commutative-cover-map-on-trivializations}). This is a consequence of Theorem \ref{thm-torus-commutative-cover-duality} whose proof occupies \S\ref{sec-sign-cover} and \S\ref{sec-torus-linear-cover}.

In \S\ref{sec-duality-for-center}, we will use this result to establish the Langlands correspondence for the ``sharp center". This will be needed for the formulation of compatibility with central core character (\emph{cf.}~\S\ref{sec-weissman-obstruction}). In \S\ref{sec-duality-for-the-cocenter}, we explain another consequence of Theorem \ref{thm-torus-commutative-cover-duality} which will not be needed later. Its purpose is to justify why the local Langlands correspondence for sharp covers is not far from the local Langlands correspondence for linear reductive groups.

\subsection{Duality for tori}
\label{sec-sharp-tori-duality}

\begin{void}
\label{void-class-field-theory-functor}
Given topological abelian groups $A_1$, $A_2$, we write $\extension^1(A_1, A_2)$ for the groupoid of commutative extensions of $A_1$ by $A_2$. We endow $\complexes^{\times}$ with the discrete topology.

For any $F$-torus $T$, we shall construct a $\integers$-linear functor (natural in $T$)
\begin{equation}
\label{eq-class-field-theory-functor}
\Langlands_T : \cycle^2(W_F, \check T(\complexes)) \rightarrow \extension^1(T(F), \complexes^{\times}),
\end{equation}
where $\check T$ stands for the Langlands dual of $T$ and the left-hand-side is defined in \S\ref{void-group-cochain-as-space}.
\end{void}

\begin{void}\emph{Construction of \eqref{eq-class-field-theory-functor}.}
Since the group $H^2(W_F, \check T(\complexes))$ vanishes (\emph{cf.}~\cite[Theorem 3.2.2]{MR3003999}), the space $\cycle^2(W_F, \check T(\complexes))$ is connected and thus identified with the classifying space of $\cycle^1(W_F, \check T(\complexes))$.

The automorphism group of the zero object in $\extension^1(T(F), \complexes^{\times})$ is the group $\Hom(T(F), \complexes^{\times})$ of continuous characters. To define \eqref{eq-class-field-theory-functor}, it suffices to define a $\integers$-linear functor
\begin{equation}
\label{eq-class-field-theory-functor-loop}
\cycle^1(W_F, \check T(\complexes)) \rightarrow \Hom(T(F), \complexes^{\times}).
\end{equation}

The functor \eqref{eq-class-field-theory-functor-loop} is set to be the projection $\cycle^1(W_F, \check T(\complexes)) \rightarrow H^1(W_F, \check T(\complexes))$, followed by Langlands duality for tori (\emph{cf.}~\cite[\S7.5]{MR2508725})
\begin{equation}
\label{eq-langlands-duality-for-tori}
H^1(W_F, \check T(\complexes)) \xrightarrow{\simeq} \Hom(T(F), \complexes^{\times}).
\end{equation}
\end{void}

\begin{rem}
\label{rem-class-field-theory-functor-loop}
By construction, $\pi_1\Langlands_T$ is the Langlands duality \eqref{eq-langlands-duality-for-tori} for $T$.
\end{rem}

\begin{rem}
There is some asymmetry in the way we defined $\Langlands_T$: The left-hand-side is a $2$-groupoid, while the right-hand-side is a $1$-groupoid. This is due to similar asymmetry in the usual formulation of Langlands duality for tori \eqref{eq-langlands-duality-for-tori}. A better formulation would be an equivalence of groupoids
\begin{equation}
\label{eq-langlands-duality-for-tori-categorical}
\cycle^1(W_F, \check T(\complexes)) \xrightarrow{\simeq} \Hom(\Isoc_T, */\complexes^{\times}),
\end{equation}
where $\Isoc_T$ is understood as a pro-Picard groupoid. The equivalence \eqref{eq-langlands-duality-for-tori-categorical} ought to recover \eqref{eq-langlands-duality-for-tori} on $\pi_0$ and the Pontryagin duality between $\check T(\complexes)^{\Gal_F}$ and $(\pi_1T)_{\Gal_F}$ on $\pi_1$.

We will not adopt this point of view in the present article, since the benefits it brings are not visible at the level of our results.
\end{rem}

\begin{void}
\label{void-torus-commutative-cover}
We now let $A$ be a finite abelian group of order invertible in $F$, equipped with an injective character $\zeta : A \rightarrow \complexes^{\times}$.

Let $T$ be an $F$-torus endowed with an $\mathbb E_{\infty}$-monoidal morphism $\mu : \deloop T \rightarrow \deloop^4A(1)$. Applying the construction functor \eqref{eq-construction-of-covers} to $\mu$, we obtain a cover $\widetilde T$ of $T(F)$. It is commutative since $\mu$ is $\mathbb E_{\infty}$-monoidal. Inducing along $\zeta$, we obtain a commutative extension
\begin{equation}
\label{eq-torus-commutative-cover}
1 \rightarrow \complexes^{\times} \rightarrow \widetilde T_{\zeta} \rightarrow T(F) \rightarrow 1.
\end{equation}
We shall view $\widetilde T_{\zeta}$ as an object of $\extension^1(T(F), \complexes^{\times})$.

We shall now apply the construction of the dual datum \eqref{eq-metaplectic-ell-group-center} to $(T, \mu)$. In the present context, we have $Z_H \cong H \cong \check T$. Thus \eqref{eq-metaplectic-ell-group-center} is a $\check T(\complexes)$-gerbe $\widetilde{\check T}$ over $*/W_F$, which we view as an object of $\cycle^2(W_F, \check T(\complexes))$.
\end{void}

\begin{thm}
\label{thm-torus-commutative-cover-duality}
There is a canonical isomorphism in $\extension^1(T(F), \complexes^{\times})$ functorial in $(T, \mu)$:
\begin{equation}
\label{eq-torus-commutative-cover-duality}
	\Langlands_T(\widetilde{\check T}) \xrightarrow{\simeq} \widetilde T_{\zeta}.
\end{equation}
\end{thm}

\begin{void}
The construction of \eqref{eq-torus-commutative-cover-duality} requires some effort and will be completed in \S\ref{sec-torus-linear-cover}.

The idea is as follows: The decomposition \eqref{eq-sharp-cover-decomposition} exhibits $\mu$ as the sum of a ``sign component'' $\mu^{(1)}$ and a $\integers$-linear component $\mu^{(2)}$. The resulting cover of $\widetilde T_{\zeta}$ is thus a Baer sum of two covers. Correspondingly, $\widetilde{\check T}$ is also the sum of two $\check T(\complexes)$-gerbes. We will construct the isomorphism \eqref{eq-torus-commutative-cover-duality} for these two summands separately and obtain the general case by adding them up, using the $\integers$-linearity of $\Langlands_T$.
\end{void}

\begin{void}
\label{void-torus-commutative-cover-local-langlands-correspondence}
Theorem \ref{thm-torus-commutative-cover-duality} yields the local Langlands correspondence \eqref{eq-local-langlands-correspondence} for $(T, \mu)$.

Indeed, the functor $\Langlands_T$ carries trivializations of $\widetilde{\check T}$ to trivializations of $\widetilde T_{\zeta}$. The latter are in bijection with the set $\Pi(\widetilde T)$ of $\zeta$-genuine characters of $\widetilde T$.

Furthermore, this map intertwines the $\cycle^1(W_F, \check T(\complexes))$-action on trivializations of $\widetilde{\check T}$ with the $\Hom(T(F), \complexes^{\times})$-action on trivializations of $\widetilde T_{\zeta}$, via the map \eqref{eq-class-field-theory-functor-loop}. Since \eqref{eq-class-field-theory-functor-loop} induces an isomorphism on $\pi_0$, $\Langlands_T$ induces an isomorphism
\begin{equation}
\label{eq-torus-commutative-cover-map-on-trivializations}
\Phi(\widetilde{\check T}) \xrightarrow{\simeq} \Pi(\widetilde T).
\end{equation}

The local Langlands correspondence for $(T, \mu)$ is define to be the inverse of \eqref{eq-torus-commutative-cover-map-on-trivializations}.
\end{void}

\subsection{The sign component}
\label{sec-sign-cover}

\begin{void}
In this subsection, we assume $\characteristic F\neq 2$. Our goal is to construct \eqref{eq-torus-commutative-cover-duality} when $\mu = \mu^{(1)}$, \emph{i.e.}~when it arises from the $\mathbb E_{\infty}$-monoidal morphism $\sgn : \deloop^2\{\pm 1\} \rightarrow \deloop^4\{\pm 1\}^{\otimes 2}$ (\emph{cf.}~\S\ref{void-sign-cover}) by pre-composing with $\epsilon\otimes\Psi$ and post-composing with the inclusion of $\{\pm 1\}$ in $A$.

Let us begin by treating the ``universal" case, where $\epsilon$ is the identity on $\integers/2$.
\end{void}

\begin{void}
Viewing $\sgn$ as a section of the fiber of \eqref{eq-multiplicative-cover-squaring-map} and applying the construction functor \eqref{eq-construction-of-covers} for $G := \mathbb G_m$, we obtain a cover of $F^{\times}$ whose pullback along $(\cdot)^2 : F^{\times} \rightarrow F^{\times}$ is endowed with a splitting.

These data can be packaged by a diagram of topological groups
\begin{equation}
\label{eq-sign-cover}
\begin{tikzcd}[column sep = 1.5em]
	& & & F^{\times} \ar[d, "(\cdot)^2"]\ar[dl, swap, "\tau"] \\
	1 \ar[r] & \{\pm 1\} \ar[r] & \widetilde F^{\times}_{\sgn} \ar[r] & F^{\times} \ar[r] & 1
\end{tikzcd}
\end{equation}
where the lower row is a double cover of $F^{\times}$.

Our main result of this subsection is the explicit identification of \eqref{eq-sign-cover}. The answer involves the cover \eqref{eq-hilbert-cover} defined by the quadratic Hilbert symbol.
\end{void}

\begin{prop}
\label{prop-sign-cover-identification}
There is a canonical isomorphism of covers
\begin{equation}
\label{eq-sign-cover-isomorphism-hilbert-cover}
\widetilde F^{\times}_{\sgn} \xrightarrow{\simeq} \widetilde F^{\times}_{\Hilbert}
\end{equation}
such that $\tau$ corresponds, under the natural bijection $\widetilde F^{\times}_{\Hilbert} \cong F^{\times} \times \{\pm 1\}$, to the map
$$
F^{\times} \rightarrow F^{\times} \times\{\pm 1\},\quad a\mapsto (a^2, \{a, a\}).
$$
\end{prop}

\begin{void}
\label{void-hilbert-cover}
In order to construct the isomorphism \eqref{eq-sign-cover-isomorphism-hilbert-cover}, we first need to describe the rigidified morphism $\deloop\mathbb G_m \rightarrow \deloop^4\{\pm 1\}^{\otimes 2}$ defining the cover $\widetilde F^{\times}_{\Hilbert}$.

Recall that the fiber sequence \eqref{eq-etale-metaplectic-cover-fiber-sequence} for $G := \mathbb G_m$ and $A := \{\pm 1\}$ admits a canonical splitting (\emph{cf.}~\cite[Remark 4.2.8]{zhao2022metaplectic})
\begin{equation}
\label{eq-multiplicative-cover-splitting}
\SMaps_e(\deloop\mathbb G_m, \deloop^4\{\pm 1\}^{\otimes 2}) \xrightarrow{\simeq} \deloop^2\{\pm 1\} \oplus \integers/2.
\end{equation}
The inclusion of $\integers/2$ is defined by sending $1$ to cup product $\deloop\Psi \cup \deloop\Psi$, where $\deloop\Psi : \deloop\mathbb G_m \rightarrow \deloop^2\{\pm 1\}$ is the deloop of the Kummer map.\footnote{We temporarily depart from our convention where $\Psi$ has coefficients in $\hat{\integers}(1)$.}

\emph{Claim}: The image of $\deloop\Psi \cup \deloop\Psi$ under the construction functor \eqref{eq-construction-of-covers} for $G := \mathbb G_m$ is canonically identified with $\widetilde F^{\times}_{\Hilbert}$.
\end{void}

\begin{void}
\begin{proof}[Proof of Claim]
Let us make the functor \eqref{eq-construction-of-covers} more explicit. Given a rigidified morphism $\mu : \deloop \mathbb G_m \rightarrow \deloop^4\{\pm 1\}^{\otimes 2}$, we obtain a $\mathbb E_1$-monoidal morphism $\mathbb G_m \rightarrow \deloop^3\{\pm 1\}^{\otimes 2}$ by taking loop spaces. The fiber $\mathbb G_m^{\dagger}$ of the latter fits into a fiber sequence of $\mathbb E_1$-monoidal stacks
\begin{equation}
\label{eq-multiplicative-group-monoidal-fiber-sequence}
\deloop^2\{\pm 1\}^{\otimes 2} \rightarrow \mathbb G_m^{\dagger} \rightarrow \mathbb G_m.
\end{equation}
Evaluating \eqref{eq-multiplicative-group-monoidal-fiber-sequence} at $\Spec F$ and using the vanishing of $H^3(\Spec F, \{\pm 1\}^{\otimes 2})$, we obtain a short exact sequence of groups
\begin{equation}
\label{eq-multiplicative-group-monoidal-fiber-sequence-evaluation}
1\rightarrow H^2(\Spec F, \{\pm 1\}^{\otimes 2}) \rightarrow \widetilde{\mathbb G}_m \rightarrow F^{\times} \rightarrow 1.
\end{equation}
The image of $\mu$ under \eqref{eq-construction-of-covers} is given by \eqref{eq-multiplicative-group-monoidal-fiber-sequence-evaluation} under Tate duality $H^2(\Spec F, \{\pm 1\}^{\otimes 2}) \cong \{\pm 1\}$, endowed with the topology defined by distinguished sections (\emph{cf.}~\cite[\S2.1.4]{zhao2022metaplectic}).

In the special case $\mu := \deloop\Psi\cup \deloop\Psi$, the fiber sequence \eqref{eq-multiplicative-group-monoidal-fiber-sequence} canonically splits as a fiber sequence of \emph{pointed} stacks (\emph{cf.}~\cite[Proposition 4.4.5]{zhao2022metaplectic}). Its monoidal product can thus be described by a cocycle $\mathbb G_m \times \mathbb G_m \rightarrow \deloop^2\{\pm 1\}^{\otimes 2}$, which one identifies with the external cup product of $\Psi$ with itself. This implies that the induced short exact sequence \eqref{eq-multiplicative-group-monoidal-fiber-sequence-evaluation} has a canonical set-theoretic splitting, with cocycle given by the Galois symbol
\begin{equation}
\label{eq-kummer-cup-product-pairing}
F^{\times} \otimes F^{\times} \rightarrow H^2(\Spec F, \{\pm 1\}^{\otimes 2}),\quad a\otimes b\mapsto [\Psi(a)]\cup[\Psi(b)],
\end{equation}
where $[\Psi(a)]$ is the Kummer class of $a \in F^{\times}$. However, \eqref{eq-kummer-cup-product-pairing} becomes the quadratic Hilbert symbol after identifying $H^2(\Spec F, \{\pm 1\}^{\otimes 2})$ with $\{\pm 1\}$ under Tate duality.
\end{proof}
\end{void}

\begin{void}
\label{void-hilbert-cover-symmetric-monoidal-parameter}
Note that every rigidified morphism $\deloop\mathbb G_m \rightarrow \deloop^4\{\pm 1\}^{\otimes 2}$ is canonically $\mathbb E_{\infty}$-monoidal because its associated symmetric form vanishes (\emph{cf.}~\cite[Proposition 4.6.2]{zhao2022metaplectic}), so we may view $\deloop\Psi\cup \deloop\Psi$ as an $\mathbb E_{\infty}$-monoidal morphism $\deloop\mathbb G_m \rightarrow \deloop^4\{\pm 1\}^{\otimes 2}$. Let us identify its image in $\SMaps_{\mathbb E_{\infty}}(\integers, \deloop^2\{\pm 1\})$ under \eqref{eq-multiplicative-cover-parametrization}, viewed as a symmetric monoidal extension:
\begin{equation}
\label{eq-hilbert-cover-symmetric-monoidal-parameter}
\deloop\{\pm 1\} \rightarrow \widetilde{\integers}_{\Hilbert} \rightarrow \integers.
\end{equation}

By construction, \eqref{eq-hilbert-cover-symmetric-monoidal-parameter} is related to \eqref{eq-multiplicative-group-monoidal-fiber-sequence} (for $\mu := \deloop\Psi\cup \deloop\Psi$) as follows: We apply the functor $\SMaps_e(\mathbb G_m, -)$ to \eqref{eq-multiplicative-group-monoidal-fiber-sequence} and form the pullback and pushout along the maps
$$
\integers \rightarrow \SMaps_e(\mathbb G_m, \mathbb G_m),\quad \SMaps_e(\mathbb G_m, \deloop^2\{\pm 1\}^{\otimes 2}) \xrightarrow{\simeq} \deloop\{\pm 1\},
$$
where the first map sends $a \in \integers$ to the character $x\mapsto x^a$ and the second map is defined by the \'etale cohomology of $\mathbb G_m$, \emph{i.e.}~the inverse to tensoring with $\Psi$.
\end{void}

\begin{void}\emph{Description of $\widetilde{\integers}_{\Hilbert}$.}
Since \eqref{eq-multiplicative-group-monoidal-fiber-sequence} (for $\mu := \deloop\Psi\cup \deloop\Psi$) admits a canonical splitting as a sequence of pointed stacks, so does \eqref{eq-hilbert-cover-symmetric-monoidal-parameter}. Let us record this splitting as an isomorphism of pointed stacks
\begin{equation}
\label{eq-hilbert-cover-pointed-splitting}
\widetilde{\integers}_{\Hilbert} \xrightarrow{\simeq} \integers \times \deloop\{\pm 1\}.
\end{equation}

Using \eqref{eq-hilbert-cover-pointed-splitting}, we may write the monoidal product on $\widetilde{\integers}_{\Hilbert}$ as a cocycle
\begin{equation}
\label{eq-hilbert-cover-cocycle}
\integers\times \integers \rightarrow \deloop \{\pm 1\}.
\end{equation}
Let us write $\Psi(-1)$ for the $\{\pm 1\}$-torsor of square roots of $-1$. There is a natural isomorphism $\Psi\cup\Psi \cong \Psi \otimes p^*\Psi(-1)$ in $\SMaps_e(\mathbb G_m, \deloop^2\{\pm 1\}^{\otimes 2})$, where $p : \mathbb G_m \rightarrow \Spec F$ is the projection (\emph{cf.}~\cite[Theorem 3.1.5]{zhao2022metaplectic}). Thus the cocycle \eqref{eq-hilbert-cover-cocycle} sends $(a, b) \in \integers\times \integers$ to the $ab$-multiple of $\Psi(-1)$. The associator of the monoidal product is given by the bilinearity of \eqref{eq-hilbert-cover-cocycle}.

It remains to describe the commutativity constraint on $\widetilde{\integers}_{\Hilbert}$. By the above description of the monoidal product, this is specified by an isomorphism $ab\cdot\Psi(-1) \cong ba \cdot \Psi(-1)$ for each $a, b\in \integers$, \emph{i.e.}~by a bilinear pairing
\begin{equation}
\label{eq-hilbert-cover-commutativity-constraint}
\integers \otimes \integers \rightarrow \{\pm 1\}.
\end{equation}
(The bilinearity is a consequence of the hexagon axiom.) By \cite[Proposition 4.6.6]{zhao2022metaplectic}, the value of this pairing at $1\otimes 1$ is $-1$. Thus \eqref{eq-hilbert-cover-commutativity-constraint} is given by $a\otimes b\mapsto (-1)^{ab}$.
\end{void}

\begin{rem}
It is also possible to arrive at the above description of the monoidal structure on $\widetilde{\integers}_{\Hilbert}$ by comparing with Brylinski and Deligne's classification of central extensions of $\mathbb G_m$ by $\Ktheory_2$ (\emph{cf.}~\cite[\S3]{MR1896177}).

Indeed, $\deloop\Psi\cup\deloop \Psi$ is the image, under \'etale realization (\emph{cf.}~\cite[\S2.3.2]{zhao2022metaplectic}), of the central extension $E$ of $\mathbb G_m$ by $\Ktheory_2$ defined using the canonical pairing $\mathbb G_m\otimes \mathbb G_m \rightarrow \Ktheory_2$ as cocycle. The \'etale realization is compatible with second Brylinski--Deligne invariants, in the sense that we have a commutative square of $\mathbb E_1$-monoidal stacks
\begin{equation}
\label{eq-brylinski-deligne-etale-monoidal-compatibility}
\begin{tikzcd}[column sep = 2em]
	\integers \ar[r]\ar[d, "\simeq"] & \deloop\mathbb G_m \ar[d, "\Psi"] \\
	\integers \ar[r] & \deloop^2\{\pm 1\}
\end{tikzcd}
\end{equation}
where the top horizontal arrow is the second Brylinski--Deligne invariant of $E$ and the bottom horizontal arrow is the $\mathbb E_1$-monoidal morphism corresponding to $\widetilde{\integers}_{\Hilbert}$.\footnote{We make an important cautionary remark. Since the cocycle $a, b\mapsto (-1)^{ab}$ is commutative, the top horizontal arrow in \eqref{eq-brylinski-deligne-etale-monoidal-compatibility} is symmetric monoidal. The bottom horizontal arrow is also $\mathbb E_{\infty}$-monoidal because $\deloop\Psi\cup\deloop\Psi$ is. However, \eqref{eq-brylinski-deligne-etale-monoidal-compatibility} is \emph{not} a commutative diagram of $\mathbb E_{\infty}$-monoidal stacks: The top circuit is $\integers$-linear while the bottom circuit is not.} Now, the second Brylinski--Deligne invariant of $E$ is the central extension of $\integers$ by $\mathbb G_m$, defined using $a,b\mapsto(-1)^{ab}$ as cocycle. This implies the above description of $\widetilde{\integers}_{\Hilbert}$ as a monoidal stack.
\end{rem}

\begin{void}\emph{Monoidal splitting of $\widetilde{\integers}_{\Hilbert}$.}
\label{void-hilbert-cover-monoidal-splitting}
Let us construct a splitting of \eqref{eq-hilbert-cover-symmetric-monoidal-parameter} as a fiber sequence of \emph{monoidal} stacks. Under the identification \eqref{eq-hilbert-cover-pointed-splitting}, this splitting is given by
\begin{equation}
\label{eq-hilbert-cover-monoidal-splitting}
\integers \rightarrow \widetilde{\integers}_{\Hilbert},\quad
a\mapsto (a, \binom{a}{2}\cdot\Psi(-1)).
\end{equation}

Because the cocycle of $\widetilde{\integers}_{\Hilbert}$ is given by $a, b\mapsto ab\cdot \Psi(-1)$, the fact that this is a monoidal splitting follows from the equality of integers
$$
\binom{a + b}{2} - \binom{a}{2} - \binom{b}{2} = ab.
$$

Denote by $\widetilde{\integers}_{\sgn}$ the trivial monoidal extension of $\integers$ by $\deloop\{\pm 1\}$ with commutativity constraint specified by $\integers \otimes\integers \rightarrow \{\pm 1\}$, $a, b\mapsto (-1)^{ab}$. The monoidal splitting \eqref{eq-hilbert-cover-monoidal-splitting} exhibits an isomorphism of symmetric monoidal extensions of $\integers$ by $\deloop\{\pm 1\}$:
\begin{equation}
\label{eq-hilbert-cover-sign-cover-identification}
\widetilde{\integers}_{\sgn} \xrightarrow{\simeq} \widetilde{\integers}_{\Hilbert}.
\end{equation}
\end{void}

\begin{void}
\label{void-hilbert-cover-square-splitting}
Finally, the construction of \eqref{eq-hilbert-cover-sign-cover-identification} renders it \emph{incompatible} with the natural splittings of the two sides over $2 : \integers \rightarrow \integers$. Let us be more precise.

The extension $\widetilde{\integers}_{\sgn}$ is monoidally equivalent to $\integers \times \deloop\{\pm 1\}$ by construction, so it admits a splitting over $2 : \integers \rightarrow \integers$ sending $a \in \integers$ to $(2a, 1)$. In other words, this is the splitting induced from $\sgn$ (as a symmetric monoidal extension of $\integers/2$ by $\deloop\{\pm 1\}$), by pulling back along $\integers \rightarrow \integers/2$.

The composition of this splitting with \eqref{eq-hilbert-cover-sign-cover-identification} is the map
\begin{equation}
\label{eq-hilbert-cover-square-splitting}
\integers \rightarrow \widetilde{\integers}_{\Hilbert},\quad a\mapsto (2a, a\cdot\Psi(-1)),
\end{equation}
because of the identity
$$
\binom{2a}{2} = a \mod 2.
$$
\end{void}

\begin{void}
We are now ready to construct the isomorphism \eqref{eq-sign-cover-isomorphism-hilbert-cover}.

\begin{proof}[Proof of Proposition \ref{prop-sign-cover-identification}]
We shall construct an isomorphism of $\mathbb E_{\infty}$-monoidal morphism $\deloop \mathbb G_m \rightarrow \deloop^4\{\pm 1\}^{\otimes 2}$ which produces \eqref{eq-sign-cover-isomorphism-hilbert-cover} under the construction functor \eqref{eq-construction-of-covers}. Using the equivalence \eqref{eq-multiplicative-cover-parametrization}, it suffices to construct an isomorphism of $\mathbb E_{\infty}$-monoidal morphisms $\integers \rightarrow \deloop\{\pm 1\}$ classifying the ``sign", respectively ``Hilbert" covers. The desired isomorphism is supplied by \eqref{eq-hilbert-cover-sign-cover-identification}.

It remains to identify the section $\tau$. By \S\ref{void-hilbert-cover-square-splitting}, this section is defined by the section of symmetric monoidal stacks
$$
\begin{tikzcd}[column sep = 1.5em]
	& & \integers \ar[d, "2"]\ar[dl, swap, "\eqref{eq-hilbert-cover-square-splitting}"] \\ 
	\deloop\{\pm 1\} \ar[r] & \widetilde{\integers}_{\Hilbert} \ar[r] & \integers
\end{tikzcd}
$$

Recall the extension $\mathbb G_m^{\dagger}$ associated to $\deloop\Psi\cup\deloop\Psi$ (\emph{cf.}~\eqref{eq-multiplicative-group-monoidal-fiber-sequence}) endowed with its natural splitting $\mathbb G_m^{\dagger} \cong \mathbb G_m \times \deloop^2\{\pm 1\}^{\otimes 2}$ as a pointed stack. We want to identify the section
\begin{equation}
\label{eq-hilbert-cover-square-splitting-geometric}
\mathbb G_m \rightarrow \mathbb G_m^{\dagger}
\end{equation}
which produces \eqref{eq-hilbert-cover-square-splitting} under the construction of \S\ref{void-hilbert-cover-symmetric-monoidal-parameter}. (Recall that the construction of \emph{loc.cit.}~is a reformulation of the equivalence \eqref{eq-multiplicative-cover-parametrization}.) The section \eqref{eq-hilbert-cover-square-splitting-geometric} will, upon evaluating at $\Spec F$ and applying Tate duality, give rise to the section $\tau$:
\begin{align*}
\tau : F^{\times} &\rightarrow \Gamma(\Spec F, \mathbb G_m^{\dagger}) \\
& \xrightarrow{\simeq} F^{\times} \times \Gamma(\Spec F, \deloop^2\{\pm 1\}^{\otimes 2}) \rightarrow F^{\times} \times H^2(\Spec F, \{\pm 1\}^{\otimes 2}) \cong F^{\times }\times \{\pm 1\}.
\end{align*}

By construction, the projection of \eqref{eq-hilbert-cover-square-splitting-geometric} onto $\deloop^2\{\pm 1\}^{\otimes 2}$ is $\Psi\otimes p^*\Psi(-1)$, where $p : \mathbb G_m \rightarrow \Spec F$ is the projection. By \cite[Theorem 3.1.5]{zhao2022metaplectic}, the latter is isomorphic to the self cup-product of $\Psi$. Hence, the second component of $\tau$ sends $a\in F^{\times}$ to the image of $[\Psi(a)] \cup [\Psi(a)] \in H^2(\Spec F, \{\pm 1\}^{\otimes 2})$ under Tate duality, which is $\{a, a\}$.
\end{proof}
\end{void}

\begin{rem}
Proposition \ref{prop-sign-cover-identification} shows that \eqref{eq-sign-cover} is generally \emph{not} induced from a cover of the cokernel of $(\cdot)^2 : F^{\times} \rightarrow F^{\times}$, the obstruction being given by $\{-1, -1\} \in F^{\times}$. This element is nontrivial if and only if $F$ is an odd degree extension of $\rationals_2$.
\end{rem}

\begin{void}
\label{void-sign-cover-meta-weil-group}
Let us now apply Proposition \ref{prop-sign-cover-identification} to the Langlands duality for tori.

Denote by $\extension^1(F^{\times}_{/2}, \complexes^{\times})$ the fiber of the endomorphism of $\extension^1(F^{\times}, \complexes^{\times})$ defined by pre-composition with $(\cdot)^2 : F^{\times} \rightarrow F^{\times}$. The commutative diagram \eqref{eq-sign-cover} together with the tautological inclusion $\{\pm 1\} \subset \complexes^{\times}$, defines an object
$$
\widetilde F^{\times}_{\sgn, /2} \in \extension^1(F^{\times}_{/2}, \complexes^{\times}).
$$

Applying the functor \eqref{eq-class-field-theory-functor} for $T := \mathbb G_m$ and using its naturality with respect to $(\cdot)^2 : \mathbb G_m \rightarrow \mathbb G_m$, we obtain a functor
$$
\Langlands_{\deloop\{\pm 1\}} : \cycle^2(W_F, \{\pm 1\}) \rightarrow \extension^1(F^{\times}_{/2}, \complexes^{\times}).
$$

Let us view the meta-Weil group $\widetilde W_F$ (\emph{cf.}~\S\ref{eq-meta-weil-group}) as an object of $\cycle^2(W_F, \{\pm 1\})$.
\end{void}

\begin{cor}
\label{cor-sign-cover-dual-identification}
There is a canonical isomorphism in $\extension^1(F^{\times}_{/2}, \complexes^{\times})$:
\begin{equation}
\label{eq-sign-cover-dual-identification}
\Langlands_{\deloop\{\pm 1\}}(\widetilde W_F) \xrightarrow{\simeq} \widetilde F^{\times}_{\sgn, /2}.
\end{equation}
\end{cor}
\begin{proof}
For an abelian group $M$, denote by $\cycle^2_e(W_F, M)$ the fiber of the map $e^* : \cycle^2(W_F, M) \rightarrow \cycle^2(*, M)$ given by pullback along the neutral point $e : * \rightarrow */W_F$. Thus $\cycle^2_e(W_F, M)$ is canonically equivalent to the groupoid of central extensions of $W_F$ by $M$.

The restriction of $\Langlands_{\mathbb G_m}$ to $\cycle^2_e(W_F, \complexes^{\times})$ admits the following explicit description: Pulling back a commutative extension of $F^{\times}$ by $\complexes^{\times}$ along the Artin reciprocity map $W_F \rightarrow F^{\times}$ yields an equivalence of groupoids
\begin{equation}
\label{eq-pullback-by-artin-reciprocity}
\extension^1(F^{\times}, \complexes^{\times}) \xrightarrow{\simeq} \cycle^2_e(W_F, \complexes^{\times}),
\end{equation}
whose inverse coincides with the restriction of $\Langlands_{\mathbb G_m}$.

In what follows, we view $\widetilde F^{\times}_{\sgn}$ as an object of $\extension^1(F^{\times}, \complexes^{\times})$. It suffices to identify its image under \eqref{eq-pullback-by-artin-reciprocity} with the extension of $W_F$ by $\complexes^{\times}$ induced from $\widetilde W_F$, and match the $2$-torsion structures defined by $\tau$ and $\widetilde W_F$. The identification follows from the isomorphism \eqref{eq-sign-cover-isomorphism-hilbert-cover}. The matching of $2$-torsion structures follows from an explicit calculation, as we now perform.

Multiplication by $2$ on $\widetilde F^{\times}_{\sgn}$ factors through an isomorphism
\begin{equation}
\label{eq-sgn-pushout-pullback-isomorphism}
\complexes^{\times} \sqcup_{\complexes^{\times}} \widetilde F^{\times}_{\sgn} \xrightarrow{\simeq} \widetilde F^{\times}_{\sgn} \times_{F^{\times}} F^{\times}
\end{equation}
where the push-out is along $(\cdot)^2 : \complexes^{\times} \rightarrow \complexes^{\times}$ and the pullback is along $(\cdot)^2 : F^{\times} \rightarrow F^{\times}$. Using the isomorphism \eqref{eq-sign-cover-isomorphism-hilbert-cover}, we may represent an element of $\widetilde F^{\times}_{\sgn}$ by a pair $(a, z)$ with $a \in F^{\times}$ and $z\in\complexes^{\times}$. Its image under \eqref{eq-sgn-pushout-pullback-isomorphism} is
$$
((a^2, \{a, a\} z^2), a)
$$
which equals the product of $z^2$ with the $(\tau(a), a)$. Hence, splitting of the pushout induced from $\tau$ sends $(a, z)$ to $z^2$. The kernel of this map is the extension \eqref{eq-hilbert-cover}, as desired.
\end{proof}

\begin{void}
Given any $F$-torus $T$ equipped with a character $\epsilon : \Lambda \rightarrow \integers/2$, where $\Lambda$ is the sheaf of cocharacters of $T$, we have a commutative diagram
\begin{equation}
\label{eq-torus-sign-cover-functoriality}
\begin{tikzcd}[column sep = 1.5em]
	\cycle^2(W_F, \{\pm 1\}) \ar[r, "\Langlands_{\deloop\{\pm 1\}}"]\ar[d, "\epsilon^{\vee}"] & \extension^1(F^{\times}_{/2}, \complexes^{\times}) \ar[d, "\epsilon"] \\
	\cycle^2(W_F, \check T(\complexes)) \ar[r, "\Langlands_T"] & \extension^1(T(F), \complexes^{\times})
\end{tikzcd}
\end{equation}

Denote by $\mu$ the $\mathbb E_{\infty}$-monoidal morphism $\deloop T \rightarrow \deloop^4\{\pm 1\}^{\otimes 2}$ obtained by composing $\sgn$ with $\epsilon\otimes \deloop\Psi$ (\emph{cf.}~\eqref{eq-sign-cover-composition-with-character}) and by $\widetilde T$ the induced commutative extension of $T(F)$ by $\complexes^{\times}$.

Denote by $\widetilde{\check T}$ the $\check T(\complexes)$-gerbe over the pro-space $*/W_F$ induced from $\widetilde W_F$ along the dual character $\epsilon^{\vee} : \{\pm 1\} \rightarrow \check T(\complexes)$.
\end{void}

\begin{cor}
There is a canonical isomorphism in $\extension^1(T(F), \complexes^{\times})$ functorial in $(T, \epsilon)$:
\begin{equation}
\label{eq-torus-sign-cover-dual-identification}
\Langlands_T(\widetilde{\check T}) \xrightarrow{\simeq} \widetilde T.
\end{equation}
\end{cor}
\begin{proof}
The isomorphism \eqref{eq-torus-sign-cover-dual-identification} is defined as the image of \eqref{eq-sign-cover-dual-identification} under the right vertical functor of \eqref{eq-torus-sign-cover-functoriality}, using the commutativity of the latter.
\end{proof}

\subsection{The $\integers$-linear component}
\label{sec-torus-linear-cover}

\begin{void}
Let $A$ be a finite abelian group of order invertible in $F$ equipped with an injective character $\zeta : A \rightarrow \complexes^{\times}$. Let $T$ be an $F$-torus and $\mu$ be a $\integers$-linear morphism $\deloop T \rightarrow \deloop^4A(1)$.

The first goal of this subsection is to construct the isomorphism \eqref{eq-torus-commutative-cover-duality} for $(T, \mu)$, \emph{i.e.}~we shall produce an isomorphism in $\extension^1(T(F), \complexes^{\times})$ functorial in $(T, \mu)$:
\begin{equation}
\label{eq-torus-linear-cover-duality}
	\Langlands_T(\widetilde{\check T}) \xrightarrow{\simeq} \widetilde T_{\zeta}.
\end{equation}

Afterwards, we will combine \eqref{eq-torus-sign-cover-dual-identification} and \eqref{eq-torus-linear-cover-duality} to prove Theorem \ref{thm-torus-commutative-cover-duality}.
\end{void}

\begin{void}
Denote by $\Lambda$ the \'etale sheaf of cocharacters of $T$. Recall that tensor product with $\deloop\Psi$ defines an equivalence
\begin{equation}
\label{eq-torus-linear-cover-classification}
\Hom_{\integers}(\Lambda, \deloop^2 A) \xrightarrow{\simeq} \Hom_{\integers}(\deloop T, \deloop^4A(1)).
\end{equation}

Thus, $\mu$ corresponds under \eqref{eq-torus-linear-cover-classification} to a $\integers$-linear morphism $\Lambda \rightarrow \deloop^2 A$. Inducing the latter along $\zeta : A \rightarrow \complexes^{\times}$ and passing to adjoints, we obtain the $\check T(\complexes)$-gerbe $\widetilde{\check T}$.
\end{void}

\begin{void}\emph{Split tori.}
\label{void-torus-linear-cover-duality-split}
Let us first construct \eqref{eq-torus-linear-cover-duality} in the special case where $T$ is split. We view $\Lambda$ as an abelian group. The functor $\Langlands_T$ renders the following diagram commute:
$$
\begin{tikzcd}[column sep = 1em]
	\cycle^2(W_F, \check T(\complexes)) \ar[r, "\simeq"]\ar[d, "\Langlands_T"] & \Hom(\Lambda, \cycle^2(W_F, \complexes^{\times})) \ar[d, "{\Hom(\Lambda, \Langlands_{\mathbb G_m})}"] \\
	\extension^1(T(F), \complexes^{\times}) \ar[r, "\simeq"] & \Hom(\Lambda, \extension^1(F^{\times}, \complexes^{\times}))
\end{tikzcd}
$$
Here, the horizontal isomorphisms are induced from $\check T(\complexes) \cong \check{\Lambda} \otimes \complexes^{\times}$ and $T(F) \cong \Lambda\otimes F^{\times}$.

Since $\mu$ is the tensor product of $\deloop\Psi$ with a $\integers$-linear morphism $\Lambda \rightarrow \deloop^2 A$, the construction of \eqref{eq-torus-linear-cover-duality} reduces to the case $T = \mathbb G_m$, where $\mu$ corresponds to a section of $\deloop^2A$. It remains to identify the composition
\begin{equation}
\label{eq-multiplicative-group-linear-cover-dual}
\Gamma(\Spec F, \deloop^2 A) \xrightarrow{\simeq} \cycle^2(W_F, A) \xrightarrow{\zeta} \cycle^2(W_F, \complexes^{\times}) \xrightarrow{\Langlands_{\mathbb G_m}} \extension^1(F^{\times}, \complexes^{\times})
\end{equation}
with the composition
\begin{equation}
\label{eq-multiplicative-group-linear-cover-construction}
\Gamma(\Spec F, \deloop^2A) \xrightarrow{\otimes\deloop\Psi} \Hom_{\integers}(\deloop\mathbb G_m, \deloop^4A(1)) \xrightarrow{\int_F} \extension^1(F^{\times}, A) \xrightarrow{\zeta} \extension^1(F^{\times}, \complexes^{\times}).
\end{equation}
\end{void}

\begin{lem}
The maps \eqref{eq-multiplicative-group-linear-cover-dual} and \eqref{eq-multiplicative-group-linear-cover-construction} are canonically isomorphic.
\end{lem}
\begin{proof}
Both maps are natural in the finite subgroup $A$ of $\complexes^{\times}$. Thus they factor through the colimit of $\Gamma(\Spec F, \deloop^2A_1)$, taken over subgroups $A_1$ of $\complexes^{\times}$ containing $A$. Since the colimit of $H^2(\Spec F, A_1)$ vanishes, we have an isomorphism
$$
\colim_{A_1} */\Gamma(\Spec F, \deloop A_1) \xrightarrow{\simeq} \colim_{A_1} \Gamma(\Spec F, \deloop^2 A_1).
$$
Thus, it suffices to identify \eqref{eq-multiplicative-group-linear-cover-dual} and \eqref{eq-multiplicative-group-linear-cover-construction} over the neutral component of $\Gamma(\Spec F, \deloop^2A)$ for every finite subgroup $A$ of $\complexes^{\times}$, functorially in $A$.

By taking loop spaces, this reduces to the commutativity of the diagram
\begin{equation}
\label{eq-artin-symbol-via-kummer}
\begin{tikzcd}[column sep = 1em]
\Gamma(\Spec F, \deloop A) \ar[r, "\simeq"]\ar[d, "\otimes\Psi"] & \cycle^1(W_F, A) \ar[d, "\mathrm{Artin}"] \\
\Hom_{\integers}(\mathbb G_m, \deloop^2 A(1)) \ar[r] & \Hom(F^{\times}, A)
\end{tikzcd}
\end{equation}
where the right vertical map is Artin reciprocity and the lower horizontal map is the evaluation at $\Spec F$ followed by Tate duality $H^2(\Spec F, A(1)) \cong A$. The commutativity of \eqref{eq-artin-symbol-via-kummer} amounts to expressing Artin reciprocity as adjoint to the pairing
\begin{align*}
H^1(\Spec F, A) \otimes F^{\times} &\xrightarrow{\id\otimes\Psi} H^1(\Spec F, A) \otimes H^1(\Spec F, \hat{\integers}(1)) \\
& \xrightarrow{\cup} H^2(\Spec F, A(1)) \xrightarrow{\simeq} A,
\end{align*}
which is essentially its definition.
\end{proof}

\begin{void}\emph{Induced tori.}
Suppose that $T$ is the Weil restriction of a split $F_1$-torus $T_1$ for a finite Galois extension $F \subset F_1$. The definition of $\Langlands_T$ renders the following diagram commute
\begin{equation}
\label{eq-weil-restriction-class-field-theory-functor-compatibility}
\begin{tikzcd}[column sep = 1.5em]
	\cycle^2(W_F, \check T(\complexes)) \ar[r, "\Langlands_T"]\ar[d, "\simeq"] & \extension^1(T(F), \complexes^{\times}) \ar[d, "\simeq"] \\
	\cycle^2(W_{F_1}, \check T_1(\complexes)) \ar[r, "\Langlands_{T_1}"] & \extension^1(T_1(F_1), \complexes^{\times})
\end{tikzcd}
\end{equation}
Here, the right vertical isomorphism is induced from the identification $T(F) \cong T_1(F_1)$ and the left vertical isomorphism is induced from the identification between $\check T(\complexes)$ and the pushforward of $\check T_1(\complexes)$ along $*/W_{F_1} \rightarrow */W_F$.

Denote by $\nu : \Spec F_1 \rightarrow \Spec F$ the natural map. The \'etale sheaf $\Lambda$ coincides with $\nu_*\Lambda_1$, where $\Lambda_1$ is the (constant) sheaf of cocharacters of $\Lambda_1$. The adjunction between $\nu_*$ and $\nu^!$ as functors on \'etale sheaves yields an isomorphism
\begin{align}
\notag
\Hom_{\integers}(\Lambda, A[2]) &\xrightarrow{\simeq} \Hom_{\integers}(\nu_*\Lambda_1, A[2]) \\
\label{eq-etale-sheaf-shriek-adjunction}
& \xrightarrow{\simeq} \Hom_{\integers}(\Lambda_1, \nu^!A[2]) \xrightarrow{\simeq} \Hom_{\integers}(\Lambda_1, A[2]),
\end{align}
where the last isomorphism comes from the identification $\nu^!A \cong A$.

Under the equivalences \eqref{eq-artin-symbol-via-kummer} and \eqref{eq-torus-linear-cover-classification}, the $\integers$-linear morphism $\mu : \deloop T \rightarrow \deloop^4 A(1)$ corresponds to a $\integers$-linear morphism $\mu_1 : \deloop T_1 \rightarrow \deloop^4A(1)$ over $\Spec F_1$. Furthermore, \eqref{eq-etale-sheaf-shriek-adjunction} is compatible with the vertical isomorphisms of \eqref{eq-weil-restriction-class-field-theory-functor-compatibility}, so the desired isomorphism \eqref{eq-torus-linear-cover-duality} for $T$ follows from the one for $T_1$ (\emph{cf.}~\S\ref{void-torus-linear-cover-duality-split}).
\end{void}

\begin{void}\emph{General tori.}
We turn to the case where $T$ is any $F$-torus. Choose a finite Galois extension $F \subset F_1$ such that $T_1 := T\times_{\Spec F}\Spec F_1$ splits. Denote by $T'$ the Weil restriction of $T_1$ to $\Spec F$, so we have an injection $T \rightarrow T'$ of $F$-tori.

Note that $\mu$ extends to a $\integers$-linear morphism $\mu' : \deloop T' \rightarrow \deloop^4A(1)$: By \eqref{eq-torus-linear-cover-classification}, it suffices to prove that any $\integers$-linear morphism $\Lambda \rightarrow \deloop^2A$ extends to a $\integers$-linear morphism $\Lambda' \rightarrow \deloop^2 A$. The obstruction lies in the cohomology group
$$
H^3(\Spec F, (\Lambda'/\Lambda)^{\vee} \otimes A)
$$
which vanishes because $\Lambda'/\Lambda$ is torsion-free and $\Spec F$ has cohomological dimension $2$.

The desired isomorphism \eqref{eq-torus-linear-cover-duality} for $T$ thus follows from the one for $T'$ by functoriality along the map of $F$-tori $T \rightarrow T'$. We omit the verification that this isomorphism is independent of the choice of $F_1$ and the extension $\mu'$.
\end{void}

\begin{rem}
The cover $\widetilde T$ associated to a $\integers$-linear morphism $\mu : \deloop T \rightarrow \deloop^4A(1)$ has been constructed by Kaletha (\emph{cf.}~\cite[\S2.2]{kaletha2022}). However, his construction is effectively the left-hand-side of \eqref{eq-torus-linear-cover-duality}.\footnote{For this reason, the construction of $\widetilde G$ for a $\integers$-linear morphism $\mu : \deloop G_{\abelian} \rightarrow \deloop^4A(1)$ given in \emph{op.cit.}~requires $G$ to be quasi-split.} Therefore, one may also interpret \eqref{eq-torus-linear-cover-duality} as the comparison between our construction of $\widetilde T$, which does not invoke Langlands duality, with Kaletha's.
\end{rem}

\begin{void}
Let us now gather all ingredients to prove Theorem \ref{thm-torus-commutative-cover-duality}.
\begin{proof}[Proof of Theorem \ref{thm-torus-commutative-cover-duality}]
Consider the decomposition \eqref{eq-sharp-cover-decomposition} for $\mu$:
$$
\mu \xrightarrow{\simeq} \mu^{(1)} + \mu^{(2)},
$$
where $\mu^{(1)}$ is defined by a character $\epsilon : \Lambda \rightarrow \integers/2$ and $\mu^{(2)}$ is $\integers$-linear. (By convention, $\mu^{(1)}$ is trivial unless $|A|$ is even.)

Denote by $\widetilde T^{(1)}_{\zeta}$ and $\widetilde T^{(2)}_{\zeta}$ the extensions of $T(F)$ by $\complexes^{\times}$ induced from $\mu^{(1)}$, and by $\widetilde{\check T}^{(1)}$ and $\widetilde{\check T}^{(2)}$ the associated $\check T(\complexes)$-gerbes over $*/W_F$. The isomorphism \eqref{eq-torus-sign-cover-dual-identification} applied to $(T, \epsilon)$ and the isomorphism \eqref{eq-torus-linear-cover-duality} applied to $(T, \mu^{(2)})$ yield isomorphisms
\begin{align*}
	\Langlands_T(\widetilde{\check T}^{(1)}) & \xrightarrow{\simeq} \widetilde T_{\zeta}^{(1)}, \\
	\Langlands_T(\widetilde{\check T}^{(2)}) & \xrightarrow{\simeq} \widetilde T_{\zeta}^{(2)}.
\end{align*}

We sum them using the $\integers$-linearity of $\Langlands_T$:
\begin{align*}
\Langlands_T(\widetilde{\check T}) & \xrightarrow{\simeq} \Langlands_T(\widetilde{\check T}^{(1)} + \widetilde{\check T}^{(2)}) \\
& \xrightarrow{\simeq} \widetilde T_{\zeta}^{(1)} + \widetilde T_{\zeta}^{(2)} \xrightarrow{\simeq}  \widetilde T_{\zeta}.
\end{align*}
This is the desired isomorphism \eqref{eq-torus-commutative-cover-duality}.
\end{proof}
\end{void}

\subsection{Duality for the center}
\label{sec-duality-for-center}

\begin{void}
\label{void-duality-for-center-context}
Let $A$ be a finite abelian group of order invertible in $F$, equipped with an injective character $\zeta : A \rightarrow \complexes^{\times}$. Let $G$ be a reductive group $F$-scheme and $\mu : \deloop G_{\abelian} \rightarrow \deloop^4A(1)$ be an $\mathbb E_{\infty}$-monoidal morphism. (The results of this subsection will be applied to the pair $(G^{\sharp}, \mu_{G_{\abelian}^{\sharp}})$ defined by a general rigidified morphism $\deloop G \rightarrow \deloop^4A(1)$, \emph{cf.}~\S\ref{void-metaplectic-dual-data}.) We also use $\mu$ to denote its pullback to $\deloop G$, viewed as a pointed morphism.

Denote by $\widetilde G$ the image of $\mu$ under the construction functor \eqref{eq-construction-of-covers}. Its pullback along $Z(F) \rightarrow G(F)$ is a commutative extension
\begin{equation}
\label{eq-center-commutative-cover}
1 \rightarrow A \rightarrow \widetilde Z \rightarrow Z(F) \rightarrow 1.
\end{equation}

Denote by $\widetilde Z_{\zeta}$ extension of $Z(F)$ by $\complexes^{\times}$ induced from \eqref{eq-center-commutative-cover} along $\zeta$.

In this subsection, we shall construct L-parameters for the set $\Pi(\widetilde Z)$ of $\zeta$-genuine smooth characters of $\widetilde Z$ using Theorem \ref{thm-torus-commutative-cover-duality}.
\end{void}

\begin{void}
\label{void-duality-for-center-galois-data}
Let us first fix notation for data on the Galois side. The dual group $H$ of $(G, \mu)$ is the Langlands dual of $G$ and we have an object $\widetilde Z_H$ of $\cycle^2(W_F, Z_H(\complexes))$.

We also slightly extend the formalism of \S\ref{void-group-cochain-as-space}: Given a pro-group $\Sigma = \lim_{i\in I}\Sigma_i$ and a \emph{complex} of sheaves of abelian groups $\mathscr A$ over $*/\Sigma_i$, we write $\cycle^n(\Sigma, \mathscr A)$ for the $\integers$-linear space underlying
$$
\Gamma(*/\Sigma, \mathscr A[n]) := \colim_{j \in I_{/i}} \Gamma(*/\Sigma_j, \mathscr A[n]).
$$

In other words, $\cycle^n(\Sigma, \mathscr A)$ is the space of \emph{hyper}cocycles of degree $n$. We will still refer to objects of $\cycle^2(\Sigma, \mathscr A)$ as $\mathscr A$-gerbes over $*/\Sigma$.

Denote by $\widetilde H_{\abelian}$ the object of $\cycle^2(W_F, H_{\abelian}(\complexes))$ induced from $\widetilde Z_H$ by functoriality along the map of complexes $Z_H(\complexes) \rightarrow H_{\abelian}(\complexes)$. Denote by $\Phi(\widetilde H_{\abelian})$ the set of isomorphism classes of trivializations of $\widetilde H_{\abelian}$.
\end{void}

\begin{rem}
Note that the total space of the $H_{\abelian}(\complexes)$-gerbe $\widetilde H_{\abelian}$ is a pro-space over $*/W_F$ and $\Phi(\widetilde H_{\abelian})$ is the set of isomorphism classes of its sections. In particular, functoriality with respect to $\widetilde H \rightarrow \widetilde H_{\abelian}$ defines a map
\begin{equation}
\label{eq-langlands-parameter-functoriality-cocenter}
\Phi(\widetilde H) \rightarrow \Phi(\widetilde H_{\abelian}).
\end{equation}
\end{rem}

\begin{void}
Denote by $T_{H, \sconn}$ the Langlands dual torus of $T_{\adjoint}$. (It coincides with the maximal torus of $H_{\sconn}$ induced from $T_H$.) The functor \eqref{eq-class-field-theory-functor} for $T$ and $T_{\adjoint}$ fits into a commutative diagram
\begin{equation}
\label{eq-class-field-theory-functor-adjoint-torus}
\begin{tikzcd}[column sep = 1em]
	\cycle^2(W_F, T_{H, \sconn}(\complexes)) \ar[r]\ar[d, "\Langlands_{T_{\adjoint}}"] & \cycle^2(W_F, T_H(\complexes)) \ar[d, "\Langlands_T"] & \\
	\extension^1(T_{\adjoint}(F), \complexes^{\times}) \ar[r] & \extension^1(T(F), \complexes^{\times})
\end{tikzcd}
\end{equation}

Since $H^3(W_F, T_{H, \sconn}(\complexes))$ vanishes, the cofiber of the top row of \eqref{eq-class-field-theory-functor-adjoint-torus} is identified with $\cycle^2(W_F, H_{\abelian}(\complexes))$. Therefore, \eqref{eq-class-field-theory-functor-adjoint-torus} induces a functor
\begin{equation}
\label{eq-class-field-theory-functor-center}
	\Langlands_Z : \cycle^2(W_F, H_{\abelian}(\complexes)) \rightarrow \extension^1(Z(F), \complexes^{\times}).
\end{equation}

Since $\widetilde H_{\abelian}$ comes from an object of $\cycle^2(W_F, T_H(\complexes))$, the isomorphism \eqref{eq-torus-commutative-cover-duality} induces an isomorphism in $\extension^1(Z(F), \complexes^{\times})$:
\begin{equation}
\label{eq-center-commutative-cover-duality}
\Langlands_Z(\widetilde H_{\abelian}) \xrightarrow{\simeq} \widetilde Z_{\zeta}.
\end{equation}
\end{void}

\begin{void}
We shall use \eqref{eq-class-field-theory-functor-center} and \eqref{eq-center-commutative-cover-duality} to construct the local Langlands correspondence for $\widetilde Z$, in analogy with \S\ref{void-torus-commutative-cover-local-langlands-correspondence}:
\begin{equation}
\label{eq-local-langlands-correspondence-sharp-center}
	\LLC : \Pi(\widetilde Z) \xrightarrow{\simeq} \Phi(\widetilde H_{\abelian}).
\end{equation}

Indeed, the functor $\Langlands_Z$ carries trivializations of $\widetilde H_{\abelian}$ to trivializations of $\widetilde Z_{\zeta}$, which are in bijection with $\zeta$-genuine characters of $\widetilde Z$. This map intertwines the $\cycle^1(W_F, H_{\abelian}(\complexes))$-action on trivializations of $\widetilde H_{\abelian}$ with the $\Hom(Z(F), \complexes^{\times})$-action on trivializations of $\widetilde Z_{\zeta}$, via the $\Langlands_Z$-action on loop spaces. We shall argue that the latter induces a bijection on $\pi_0$: Indeed, it occurs as the last vertical arrow in the commutative diagram
$$
\begin{tikzcd}[column sep = 1.5em]
	H^1(W_F, T_{H, \sconn}(\complexes)) \ar[r]\ar[d, "\pi_1\Langlands_{T_{\adjoint}}"] & H^1(W_F, T_H(\complexes)) \ar[r]\ar[d, "\pi_1\Langlands_T"] & H^1(W_F, H_{\abelian}(\complexes)) \ar[r]\ar[d, "\pi_1\Langlands_Z"] & 1 \\
	\Hom(T_{\adjoint}(F), \complexes^{\times}) \ar[r] & \Hom(T(F), \complexes^{\times}) \ar[r] & \Hom(Z(F), \complexes^{\times}) \ar[r] & 1
\end{tikzcd}
$$
Here, the top row is exact because $H^2(W_F, T_{H, \sconn}(\complexes))$ vanishes (\emph{cf.}~\cite[Theorem 3.2.2]{MR3003999}) and the bottom row is exact because $\complexes^{\times}$ is divisible. Since $\pi_1\Langlands_{T_{\adjoint}}$ and $\pi_1\Langlands_T$ are isomorphisms (\emph{cf.}~Remark \ref{rem-class-field-theory-functor-loop}), so is $\pi_1\Langlands_Z$.

It follows that the action of $\Langlands_Z$ on trivializations of $\widetilde H_{\abelian}$ defines a bijection
\begin{equation}
\label{eq-sharp-center-map-on-trivializations}
\Phi(\widetilde H_{\abelian}) \xrightarrow{\simeq} \Pi(\widetilde Z).
\end{equation}
We define \eqref{eq-local-langlands-correspondence-sharp-center} to be the inverse to \eqref{eq-sharp-center-map-on-trivializations}.
\end{void}

\subsection{Duality for the cocenter}
\label{sec-duality-for-the-cocenter}

\begin{void}
We keep the notation of \S\ref{void-duality-for-center-context} and \S\ref{void-duality-for-center-galois-data}. The goal of this subsection is to address the following question: When does $\widetilde G$ admit a $\zeta$-genuine character?

We shall only answer this question when $G$ is quasi-split. For the remainder of this subsection, we fix a Borel subgroup $B$ and a section of the projection $B \rightarrow T$, realizing $T$ as a subgroup of $G$.

The results of this subsection will not be used in the sequel. We include them because they support the philosophy that the local Langlands correspondence for covers defined by $\mathbb E_{\infty}$-monoidal morphisms $\mu : \deloop G_{\abelian} \rightarrow \deloop^4A(1)$ is ``not too far'' from the local Langlands correspondence for linear algebraic groups. The case where $\mu$ is $\integers$-linear is due to Kaletha (\emph{cf.}~\cite[\S2]{kaletha2022}) and no new idea is needed to treat the $\mathbb E_{\infty}$-monoidal case.
\end{void}

\begin{void}
Under the split fiber sequence \eqref{eq-sharp-cocenter-fiber-sequence} (\emph{cf.}~\S\ref{void-sharp-cocenter-fiber-sequence-splitting}), the $\mathbb E_{\infty}$-monoidal morphism $\mu$ defines a character $\pi_1 G \rightarrow \integers/2$ (trivial unless $|A|$ is even) and $\integers$-linear morphism $\pi_1 G \rightarrow \deloop^2A$.

Denote by $\pi_1^t G$ the torsion subgroup of $\pi_1G$. We restrict the two maps above to $\pi_1^t G$ and apply one unit of Tate twist. This gives us two $\integers$-linear maps
\begin{align*}
	\epsilon : \pi_1^t G(1) &\rightarrow \mu_2,\\
	f : \pi_1^t G(1) &\rightarrow \deloop^2A(1).
\end{align*}

For any section $\mathscr G$ of $\deloop^2A(1)$ over $\Spec F$, we denote by $[\mathscr G] \in A$ the image of its isomorphism class under Tate duality $H^2(\Spec F, A(1)) \cong A$. Write $\{\cdot, \cdot\}$ for the quadratic Hilbert symbol---we view it as valued in the subgroup $\{\pm 1\} \subset A$ if $|A|$ is even and trivial if $|A|$ is odd. Then we may form the character
\begin{equation}
\label{eq-sharp-cover-genuine-character-obstruction}
\pi_1^t G(1)(F) \rightarrow A, \quad \theta \mapsto \{\epsilon(\theta), \epsilon(\theta)\} \cdot [f(\theta)].
\end{equation}

The following result is an analogue of \cite[Proposition 2.4.7]{kaletha2022}.
\end{void}

\begin{prop}
\label{prop-genuine-character-obstruction}
The following statements are equivalent.
\begin{enumerate}
	\item the class of $\widetilde Z_H$ in $H^2(W_F, Z_H(\complexes))$ vanishes;
	\item the homomorphism \eqref{eq-sharp-cover-genuine-character-obstruction} vanishes;
	\item $\widetilde G$ admits a $\zeta$-genuine character.
\end{enumerate}
\end{prop}

\begin{void}
Let us begin with an elementary observation: Quasi-splitness of $G$ implies that $T_{\sconn}$ is the Weil restriction of a split torus, so $H^1(\Spec F, T_{\sconn})$ vanishes. The groupoid of $F$-points of the cocenter $G_{\abelian}$ can thus be identified as
\begin{equation}
\label{eq-cocenter-quasisplit-field-valued-points}
G_{\abelian}(F) \xrightarrow{\simeq} T(F)/T_{\sconn}(F) \xrightarrow{\simeq} G(F)/G_{\sconn}(F),
\end{equation}
where the quotients are taken in the sense of groupoids.

In particular, the $\pi_1$ of $G_{\abelian}(F)$ is identified with $\pi_1^t G(1)(F)$, while its $\pi_0$ is identified with the cokernel of $T_{\sconn}(F) \rightarrow T(F)$, as well as the cokernel of $G_{\sconn}(F) \rightarrow G(F)$.
\end{void}

\begin{void}
\label{void-cocenter-cover-automorphism}
Let us write $\extension^1(G_{\abelian}(F), A)$ for the fiber of
$$
\extension^1(T(F), A) \rightarrow \extension^1(T_{\sconn}(F), A),
$$
so an object $\widetilde G_{\abelian}$ of $\extension^1(G_{\abelian}(F), A)$ can be thought of as a commutative extension $\widetilde T$ of $T(F)$ equipped with a splitting over $T_{\sconn}(F)$. (The case where $G_{\abelian}$ is replaced by $\deloop\{\pm 1\}$ has already appeared in \S\ref{void-sign-cover-meta-weil-group}.)

The $\mathbb E_{\infty}$-monoidal morphism $\mu$ defines an object $\widetilde G_{\abelian}$ of $\extension^1(G_{\abelian}(F), A)$. Restricting its splitting $T_{\sconn}(F) \rightarrow \widetilde T$ to $\pi_1^t G(1)(F)$, we obtain a character
\begin{equation}
\label{eq-cocenter-cover-automorphism}
\pi_1^tG(1)(F) \rightarrow A,
\end{equation}
which vanishes if and only if $\widetilde G_{\abelian}$ is the pullback of a commutative extension of $\pi_0(G_{\abelian}(F))$ by $A$. Let us calculate \eqref{eq-cocenter-cover-automorphism}.
\end{void}

\begin{lem}
\label{lem-cocenter-cover-automorphism}
The character \eqref{eq-cocenter-cover-automorphism} equals \eqref{eq-sharp-cover-genuine-character-obstruction}.
\end{lem}
\begin{proof}
We perform the decomposition \eqref{eq-sharp-cover-decomposition}: $\mu \cong \mu^{(1)} + \mu^{(2)}$, where $\mu^{(1)}$ is defined by $\epsilon$ and $\mu^{(2)}$ is $\integers$-linear. The character \eqref{eq-cocenter-cover-automorphism} attached to $\mu^{(1)}$ is given by $\theta\mapsto \{\epsilon(\theta), \epsilon(\theta)\}$ (\emph{cf.}~Proposition \ref{prop-sign-cover-identification}). It remains to identify the character \eqref{eq-cocenter-cover-automorphism} attached to $\mu^{(2)}$ with $\theta \mapsto [f(\theta)]$. Therefore, we may assume that $\mu$ is $\integers$-linear in what follows.

In this case, $\mu$ is the tensor product of $\deloop\Psi : \deloop \mathbb G_m \rightarrow \deloop^2\hat{\integers}(1)$ with the $\integers$-linear morphism $\pi_1 G \rightarrow \deloop^2A$ which defines $f$. Applying the loop space functor to $\mu$ and evaluating at $\Spec F$, we obtain a map of spaces
\begin{equation}
\label{eq-linear-cover-loop-space-evaluation}
G_{\abelian}(F) \rightarrow \Gamma(\Spec F, \deloop^3A(1)).
\end{equation}

The character \eqref{eq-cocenter-cover-automorphism} is obtained from \eqref{eq-linear-cover-loop-space-evaluation} by taking $\pi_1$ and identifying $H^2(\Spec F, A(1))$ with $A$ under Tate duality. This yields the character $\theta \mapsto [f(\theta)]$.
\end{proof}

\begin{void}
Let us consider the induced cover $\widetilde G_{\abelian, \zeta} \in \extension^1(G_{\abelian}(F), \complexes^{\times})$ of $\widetilde G_{\abelian}$ (\emph{cf.}~\S\ref{void-cocenter-cover-automorphism}). Define the functor $\Langlands_{G_{\abelian}}$ by the diagram of fiber sequences
$$
\begin{tikzcd}[column sep = 1em]
	\cycle^2(W_F, Z_H(\complexes)) \ar[r]\ar[d, "\Langlands_{G_{\abelian}}"] & \cycle^2(W_F, T_H(\complexes)) \ar[r]\ar[d, "\Langlands_T"] & \cycle^2(W_F, T_{H, \adjoint}(\complexes)) \ar[d, "\Langlands_{T_{\sconn}}"] \\
	\extension^1(G_{\abelian}(F), \complexes^{\times}) \ar[r] & \extension^1(T(F), \complexes^{\times}) \ar[r] & \extension^1(T_{\sconn}(F), \complexes^{\times})
\end{tikzcd}
$$

The isomorphism \eqref{eq-torus-commutative-cover-duality} for $T$ and $T_{\sconn}$ yields an isomorphism
\begin{equation}
\label{eq-cocenter-commutative-cover-duality}
\Langlands_{G_{\abelian}}(\widetilde Z_H) \xrightarrow{\simeq} \widetilde G_{\abelian, \zeta}.
\end{equation}
\end{void}

\begin{void}
We now prove Proposition \ref{prop-genuine-character-obstruction}.
\begin{proof}[Proof of Proposition \ref{prop-genuine-character-obstruction}]
(1) $\Leftrightarrow$ (2). According to Lemma \ref{lem-cocenter-cover-automorphism}, the commutative extension $\widetilde G_{\abelian, \zeta}$ splits if and only if the character \eqref{eq-sharp-cover-genuine-character-obstruction} vanishes. On the other hand, the functor $\Langlands_{G_{\abelian}}$ induces an isomorphism on the set of isomorphism classes
$$
H^2(W_F, Z_H(\complexes)) \xrightarrow{\simeq} \Hom(\pi_1^tG(1)(F), \complexes^{\times}),
$$
so the equivalence of (1) and (2) follows from \eqref{eq-cocenter-commutative-cover-duality}.

(2) $\Rightarrow$ (3). A splitting of $\widetilde G_{\abelian, \zeta}$ induces a $\zeta$-genuine character of $\widetilde G$.

(3) $\Rightarrow$ (2). Since the cover $\widetilde G$ is induced from $\widetilde G_{\abelian}$, we have a canonical section $G_{\sconn}(F) \rightarrow \widetilde G$, whose restriction to $\pi_1^t G(1)(F) \cong \Ker(G_{\sconn}(F) \rightarrow G(F))$ is the character \eqref{eq-sharp-cover-genuine-character-obstruction} (\emph{cf.}~Lemma \ref{lem-cocenter-cover-automorphism}). If $\widetilde G$ admits a $\zeta$-genuine character, its restriction to $G_{\sconn}(F)$ must vanish because $G_{\sconn}(F)$ is perfect by Platonov's theorem (\emph{cf.}~\cite[\S7.2]{MR1278263}). It follows that the character \eqref{eq-sharp-cover-genuine-character-obstruction} must also vanish.
\end{proof}
\end{void}

\begin{rem}
By Proposition \ref{prop-genuine-character-obstruction}, the obstruction to the existence of a $\zeta$-genuine character of $\widetilde G$ only has to do with the torsion subgroup $\pi_1^t G$ of $\pi_1 G$.

By taking a $z$-extension $G' \rightarrow G$, one can thus find a $\zeta$-genuine character of the induced cover $\widetilde G'$ of $G'(F)$ and effectively reduces the local Langlands correspondence for $\widetilde G$ to that for $G'(F)$. This is explained in \cite[Theorem 2.6.2]{kaletha2022}, so we shall not repeat it.
\end{rem}

\medskip

\section{Structures on $\mu$}

In this section, we work over an arbitrary base scheme $S$ and let $G$ be a reductive group $S$-scheme. We adopt the notation of \S\ref{void-reductive-group-notation} for objects associated to $G$. Let $A$ be a finite abelian group whose order is invertible over $S$.

Fix a rigidified morphism $\mu : \deloop G \rightarrow \deloop^4A(1)$. Write $Q$ for its associated quadratic form and $b$, $b_1$, $b_2$ for the induced pairings (\emph{cf.}~\S\ref{void-etale-metaplectic-cover-fiber-sequence}).

The goal of this section is construct the ``canonical quadratic structure" on $\mu$ with respect to the $\deloop Z$-action on $\deloop G$ (\emph{cf.}~Proposition \ref{prop-canonical-quadratic-structure}). This provides the key technical ingredient in our calculation of Weissman's obstruction in \S\ref{sec-weissman-obstruction}.

\subsection{The canonical quadratic structure}
\label{sec-canonical-quadratic-structure}

\begin{void}
\label{void-center-cocenter-pairing-geometric}
Consider the self-tensor product $\deloop\Psi^{\otimes 2} : \deloop\mathbb G_m \otimes \deloop \mathbb G_m \rightarrow \deloop^4\hat{\integers}(2)$ of the delooped Kummer map $\deloop\Psi$. Tensoring it with the pairing $b_2$ yields a bilinear pairing
\begin{equation}
\label{eq-center-cocenter-pairing-geometric}
b_2\otimes \deloop\Psi^{\otimes 2} : \deloop G_{\abelian} \otimes \deloop Z \rightarrow \deloop^4 A(1).
\end{equation}

We shall use the same notation $b_2\otimes \deloop\Psi^{\otimes 2}$ to denote the pullback of \eqref{eq-center-cocenter-pairing-geometric} to $\deloop G \times \deloop Z$. It is \emph{bi-rigidified} in the sense that it is equipped with trivializations over $e\times \deloop Z$ and $\deloop G \times e$ which are compatible over $e\times e$.
\end{void}

\begin{void}
\label{void-center-action-morphism}
Consider the morphisms $p_1$, $p_2$, $a$ in the diagram
\begin{equation}
\label{eq-center-action-diagram}
\begin{tikzcd}[column sep = 0em]
& \deloop G \times \deloop Z \ar[rr, "a"]\ar[dl, swap, "p_1"]\ar[dr, "p_2"] & & \deloop G \\
\deloop G & & \deloop Z
\end{tikzcd}
\end{equation}
which are, respectively, projections onto the first and the second factors and the action map. Denote by $\mu_Z$ the restriction of $\mu$ to $\deloop Z$.
\end{void}

\begin{prop}
\label{prop-canonical-quadratic-structure}
In reference to \eqref{eq-center-action-diagram}, there is a canonical isomorphism of bi-rigidified morphisms $\deloop G \times \deloop Z \rightarrow \deloop^4A(1)$:
\begin{equation}
\label{eq-canonical-quadratic-structure}
	a^*\mu - (p_1)^*\mu - (p_2)^*\mu_Z \xrightarrow{\simeq} b_2\otimes \deloop\Psi^{\otimes 2}.
\end{equation}
\end{prop}

\begin{void}
\label{void-bi-rigidified-morphism-one-truncated}
The proof of Proposition \ref{prop-canonical-quadratic-structure} will appear in \S\ref{void-proof-prop-canonical-quadratic-structure}. Let us make some preliminary remarks about its statement.

First, \eqref{eq-canonical-quadratic-structure} is supposed to be an isomorphism in a $1$-groupoid. Namely, the space of bi-rigidified morphisms $\deloop G\times \deloop Z \rightarrow \deloop^4A(1)$ is $1$-truncated.

To see this, we note that the space of bi-rigidified morphisms $\deloop G \times \deloop Z \rightarrow \deloop^4A(1)$ is equivalent to that of pointed morphisms $\deloop Z \rightarrow \SMaps_e(\deloop G, \deloop^4A(1))$. Because the third term in \eqref{eq-etale-metaplectic-cover-fiber-sequence} is discrete, such morphisms factor through $\SHom_{\integers}(\pi_1 G, \deloop^2A)$, so they correspond to monoidal morphisms
$$
Z \rightarrow \SHom_{\integers}(\pi_1G, \deloop A),
$$
which form a $1$-groupoid. (Moreover, this shows that any bi-rigidified morphism $\deloop G_{\sconn} \times \deloop Z \rightarrow \deloop^4A(1)$ is canonically trivial.)
\end{void}

\begin{void}
\label{void-canonical-quadratic-structure-cocycle-condition}
Next, we shall state a cocycle condition satisfied by \eqref{eq-canonical-quadratic-structure}. Given a $G$-torsor $\mathscr E$ and a $Z$-torsor $\mathscr Z$ over an $S$-scheme, \eqref{eq-canonical-quadratic-structure} supplies a functorial isomorphism of sections of $\deloop^4A(1)$
\begin{equation}
\label{eq-canonical-quadratic-structure-explicit}
\mu(\mathscr E\otimes \mathscr Z) - \mu(\mathscr E) - \mu(\mathscr Z) \xrightarrow{\simeq} (b_2\otimes \deloop\Psi^{\otimes 2})(\mathscr E, \mathscr Z).
\end{equation}

Furthermore, the isomorphism \eqref{eq-canonical-quadratic-structure-explicit} is compatible with the natural trivializations of the two sides, when either $\mathscr E$ or $\mathscr Z$ is the trivial torsor.

Now, given $\mathscr E$ along with two $Z$-torsors $\mathscr Z_1$, $\mathscr Z_2$ over an $S$-scheme, there are two isomorphisms between $\mu(\mathscr E\otimes\mathscr Z_1 \otimes\mathscr Z_2)$ and
\begin{align*}
\mu(\mathscr E) + \mu(\mathscr Z_1) &+ \mu(\mathscr Z_2) \\
&+ (b_2\otimes \deloop\Psi^{\otimes 2})(\mathscr E, \mathscr Z_1) + (b_2 \otimes \deloop\Psi^{\otimes 2})(\mathscr E, \mathscr Z_2) + (b_2\otimes \deloop\Psi^{\otimes 2})(\mathscr Z_1, \mathscr Z_2),
\end{align*}
given by iteratively applying \eqref{eq-canonical-quadratic-structure-explicit} in different orders. The cocycle condition states that these two isomorphisms are canonically identified. (We omit drawing this rather large commutative diagram.)

Note that this is indeed a \emph{condition} and not additional structure, because the space of pointed morphisms $\deloop Z \times \deloop Z \rightarrow \SMaps_e(\deloop G, \deloop^4A(1))$ is $1$-truncated (\emph{cf.}~\S\ref{void-bi-rigidified-morphism-one-truncated}).
\end{void}

\subsection{Construction of \eqref{eq-canonical-quadratic-structure}}
\label{sec-quadratic-structure-construction}

\begin{void}
We shall first construct \eqref{eq-canonical-quadratic-structure} in the case where $G$ is split and equipped with a Killing pair $T \subset B\subset G$.

Recall that any bi-rigidified morphism $\deloop G \times \deloop Z \rightarrow \deloop^4A(1)$ is canonically trivialized as such over $\deloop G_{\sconn} \times \deloop Z$ (\emph{cf.}~\S\ref{void-bi-rigidified-morphism-one-truncated}).

By restrictions along $\deloop T \rightarrow \deloop G$ and $\deloop T_{\sconn} \rightarrow \deloop G_{\sconn}$, the bi-rigidified morphism $a^*\mu - (p_1)^*\mu - (p_2)^*\mu_{Z_G}$ defines a bi-rigidified morphism
\begin{equation}
\label{eq-quadratic-form-torus-center}
\deloop T \times \deloop Z \rightarrow \deloop^4A(1),
\end{equation}
equipped with a trivialization $\tau$ as such over $\deloop T_{\sconn} \times \deloop Z$. 
\end{void}

\begin{void}
The bi-rigidified morphism \eqref{eq-quadratic-form-torus-center} extends to the bi-rigidified morphism
\begin{equation}
\label{eq-bi-rigidified-morphism-torus}
m^*\mu - (p_1)^*\mu - (p_2)^*\mu : \deloop T \times \deloop T \rightarrow \deloop^4A(1),
\end{equation}
where $m$, $p_1$, $p_2$ are the multiplication and projection morphisms from $\deloop T\times\deloop T$ to $\deloop T$.

By \cite[Proposition 4.7.3]{zhao2022metaplectic}, the bi-rigidified morphism \eqref{eq-bi-rigidified-morphism-torus} is identified with $b\otimes \deloop\Psi^{\otimes 2}$, where $b$ is the symmetric form attached to $Q$. By restricting to $\deloop T\times \deloop Z$, we obtain an isomorphism of bi-rigidified morphisms
\begin{equation}
\label{eq-quadratic-structure-torus-center}
a^*\mu - (p_1)^*\mu - (p_2)^*\mu_Z \xrightarrow{\simeq} b\otimes \deloop\Psi^{\otimes 2}
\end{equation}
from $\deloop T\times \deloop Z$ to $\deloop^4A(1)$.
\end{void}

\begin{void}
Since the restriction of $b$ to $\Lambda_{\sconn}\otimes \Lambda$ extends to $\Lambda_{\sconn} \otimes \Lambda_{\adjoint}$ as the bilinear pairing $b_1$, the restriction of $b\otimes \deloop\Psi^{\otimes 2}$ to $\deloop T_{\sconn} \times \deloop T$ likewise extends to $\deloop T_{\sconn} \times \deloop T_{\adjoint}$ as a bi-rigidified morphism. This endows $b\otimes \deloop\Psi^{\otimes 2}$ with a trivialization $\tau_1$ over $\deloop T_{\sconn} \times \deloop Z$.

We shall prove that the trivializations $\tau$, $\tau_1$ are intertwined by the isomorphism \eqref{eq-quadratic-structure-torus-center}. More precisely, consider the diagram of bi-rigidified morphisms $\deloop T_{\sconn} \times \deloop Z \rightarrow \deloop^4A(1)$:
\begin{equation}
\label{eq-sconn-center-compatibility}
\begin{tikzcd}[column sep = -4em, row sep = 1em]
	a^*\mu - (p_1)^*\mu - (p_2)^*\mu_Z|_{\deloop T_{\sconn}\times \deloop Z} \ar[dr, "\tau"]\ar[dd, swap, "\eqref{eq-quadratic-structure-torus-center}"] & \\
	& 0 \\
	b\otimes \deloop\Psi^{\otimes 2}|_{\deloop T_{\sconn} \times \deloop Z} \ar[ur, swap, "\tau_1"]
\end{tikzcd}
\end{equation}
\end{void}

\begin{lem}
\label{lem-sconn-center-compatibility}
The diagram \eqref{eq-sconn-center-compatibility} commutes.
\end{lem}
\begin{proof}
We first observe that this assertion involves no additional structure. Indeed, bi-rigidified morphisms $\deloop T_{\sconn} \times \deloop Z \rightarrow \deloop^4A(1)$ are classified by pointed morphisms
\begin{equation}
\label{eq-sconn-center-compatibility-space}
\deloop Z \rightarrow \SHom_{\integers}(\Lambda_{\sconn}, \deloop^2A),
\end{equation}
which form a $1$-groupoid (\emph{cf.}~\S\ref{void-bi-rigidified-morphism-one-truncated}).

Next, we reduce the commutativity of \eqref{eq-sconn-center-compatibility} to its pullback along $\deloop T_{\sconn} \times T_{\adjoint} \rightarrow \deloop T_{\sconn} \times \deloop Z$. Indeed, bi-rigidified morphisms $\deloop T_{\sconn} \times T_{\adjoint} \rightarrow \deloop^4A(1)$ also form a $1$-groupoid, isomorphic to $\Maps_e(T_{\adjoint}, \SHom(\Lambda_{\sconn}, \deloop^2A))$. We need to show that the pullback functor
\begin{equation}
\label{eq-adjoint-to-center-birigidified-pullback}
\Maps_e(\deloop Z, \SHom(\Lambda_{\sconn}, \deloop^2A)) \rightarrow \Maps_e(T_{\adjoint}, \SHom(\Lambda_{\sconn}, \deloop^2A))
\end{equation}
is faithful. Since \eqref{eq-adjoint-to-center-birigidified-pullback} is a functor of Picard groupoids, it suffices to prove that its induced map on $\pi_1$ is injective. The latter occurs as the bottom horizontal arrow of the following commutative square
\begin{equation}
\label{eq-adjoint-to-center-birigidified-linear-pullback}
\begin{tikzcd}[column sep = 1em]
\Hom_{\integers}(\deloop Z, \SHom_{\integers}(\Lambda_{\sconn}, \deloop A)) \ar[r]\ar[d, "\simeq"] & \Hom_{\integers}(T_{\adjoint}, \SHom_{\integers}(\Lambda_{\sconn}, \deloop A)) \ar[d, "\simeq"] \\
\Maps_e(\deloop Z, \SHom_{\integers}(\Lambda_{\sconn}, \deloop A)) \ar[r] & \Maps_e(T_{\adjoint}, \SHom_{\integers}(\Lambda_{\sconn}, \deloop A))
\end{tikzcd}
\end{equation}
Here, the vertical functors are the forgetful ones: The left one is an isomorphism for degree reasons, and the right one is an isomorphism by the \'etale cohomology of $T_{\adjoint}$. The kernel of the top horizontal arrow of \eqref{eq-adjoint-to-center-birigidified-linear-pullback} is identified with
$$
\Hom_{\integers}(\deloop T, \SHom_{\integers}(\Lambda_{\sconn}, \deloop A)) \xrightarrow{\simeq} \Hom_{\integers}(T, \SHom_{\integers}(\Lambda_{\sconn}, A))
$$
which vanishes because $\SHom_{\integers}(\Lambda_{\sconn}, A)$ is discrete. We may now prove the commutativity of \eqref{eq-sconn-center-compatibility} after replacing $\deloop T_{\sconn} \times \deloop Z$ by $\deloop T_{\sconn} \times T_{\adjoint}$.

Along the composite $\deloop T_{\sconn} \times T_{\adjoint} \rightarrow \deloop T_{\sconn} \times \deloop Z \rightarrow \deloop T \times \deloop T$, the restrictions of $m$ and $p_1$ coincide and the restriction of $p_2$ is trivial. This endows $a^*\mu - (p_1)^*\mu - (p_2)^*\mu_Z$ with a trivialization over $\deloop T_{\sconn} \times T_{\adjoint}$. Since \eqref{eq-quadratic-structure-torus-center} is compatible with the trivializations over $\deloop T\times e$, it intertwines this trivialization with the one of $b\otimes \deloop\Psi^{\otimes 2}|_{\deloop T_{\sconn} \times T_{\adjoint}}$ induced from $b(\cdot, 0) = 0$.

Using these two trivializations, the restriction of \eqref{eq-sconn-center-compatibility} to $\deloop T_{\sconn} \times T_{\adjoint}$ reads as follows:
\begin{equation}
\label{eq-sconn-center-compatibility-restriction-adjoint}
\begin{tikzcd}[column sep = 2em, row sep = 1em]
	0 \ar[dr, "\tau"]\ar[dd, swap, "\id"] & \\
	& 0 \\
	0 \ar[ur, swap, "\tau_1"]
\end{tikzcd}
\end{equation}
Here, $\tau$ is induced from the $G_{\adjoint}$-equivariance structure of the restriction of $\mu$ to $\deloop G_{\sconn}$ and $\tau_1$ is the map $\deloop T_{\sconn} \times T_{\adjoint} \rightarrow \deloop^3A(1)$ given by applying the loop space functor to the second factor in $b_1\otimes \deloop\Psi^{\otimes 2}$. The commutativity of \eqref{eq-sconn-center-compatibility-restriction-adjoint} is precisely \cite[Proposition 5.5.4]{zhao2022metaplectic}.
\end{proof}

\begin{void}
\label{void-proof-prop-canonical-quadratic-structure}
We are now ready to construct the isomorphism \eqref{eq-canonical-quadratic-structure}.
\begin{proof}[Proof of Proposition \ref{prop-canonical-quadratic-structure}]
Consider the monoidal morphism
\begin{equation}
\label{eq-quadratic-expression-classifying-data}
Z \rightarrow \SHom_{\integers}(\pi_1 G, \deloop A)
\end{equation}
classifying the bi-rigidified morphism $a^*\mu - (p_1)^*\mu - (p_2)^*\mu_Z$ (\emph{cf.}~\S\ref{void-bi-rigidified-morphism-one-truncated}). We need to construct an isomorphism between \eqref{eq-quadratic-expression-classifying-data} and the adjoint of the pairing
\begin{equation}
\label{eq-center-cocenter-pairing-kummer-second-factor}
b_2\otimes\Psi : \pi_1 G \otimes Z \rightarrow \deloop A
\end{equation}
defined by tensoring with $\Psi : \mathbb G_m \rightarrow \deloop\hat{\integers}(1)$ along the second factor of $b_2$.

Suppose first that $G$ is split and equipped with a Killing pair $T\subset B\subset G$. In this case, we have identified \eqref{eq-quadratic-form-torus-center} with $b\otimes \deloop\Psi^{\otimes 2}|_{\deloop T \times \deloop Z}$ via the isomorphism \eqref{eq-quadratic-structure-torus-center} and proved that the trivialization $\tau$ corresponds to the trivialization of $b\otimes \deloop\Psi^{\otimes 2}|_{\deloop T_{\sconn} \times \deloop Z}$ defined by $b_1 \otimes \deloop\Psi^{\otimes 2}$ (\emph{cf.}~Lemma \ref{lem-sconn-center-compatibility}). This yields a morphism of fiber sequences
\begin{equation}
\label{eq-quadratic-structure-fiber-sequence-map}
\begin{tikzcd}[column sep = 1em]
	Z \ar[r, "\eqref{eq-quadratic-expression-classifying-data}"]\ar[d] & \SHom_{\integers}(\pi_1G, \deloop A) \ar[d] \\
	T \ar[r, "b \otimes \Psi"]\ar[d] & \SHom_{\integers}(\Lambda, \deloop A) \ar[d] \\
	T_{\adjoint} \ar[r, "b_1\otimes \Psi"] & \SHom_{\integers}(\Lambda_{\sconn}, \deloop A)
\end{tikzcd}
\end{equation}
which gives an isomorphism between \eqref{eq-quadratic-expression-classifying-data} and the adjoint of \eqref{eq-center-cocenter-pairing-kummer-second-factor}.

We shall argue that this isomorphism is independent of the choice of the Killing pair $T\subset B\subset G$. For this, it suffices to show that the commutativity witness of the top square in \eqref{eq-quadratic-structure-fiber-sequence-map} is independent of the choice of the Killing pair. Given another Killing pair $T' \subset B' \subset G$, we need to show that the canonical identification $\deloop T \cong \deloop T'$ intertwines the isomorphism \eqref{eq-quadratic-structure-torus-center} defined for $\deloop T$, respectively $\deloop T'$. This follows because \eqref{eq-quadratic-structure-torus-center} is the restriction of an isomorphism between rigidified morphisms $\deloop T\times \deloop T \rightarrow \deloop^4A(1)$, and the latter form a discrete space classified by bilinear pairings $\Lambda \otimes \Lambda \rightarrow A(-1)$.

Since the isomorphism between \eqref{eq-quadratic-expression-classifying-data} and the adjoint of \eqref{eq-center-cocenter-pairing-kummer-second-factor} is constructed for any split $G$ without additional choices, the case for any reductive $G$ follows by \'etale descent.
\end{proof}
\end{void}

\subsection{$\deloop Z^{\sharp}$-equivariance}

\begin{void}
Recall the reductive group $S$-scheme $G^{\sharp}$ and its center $Z^{\sharp}$ (\emph{cf.}~\S\ref{void-sharp-lattices}). There is a natural map of group $S$-schemes of multiplicative type $Z^{\sharp} \rightarrow Z$. The $\deloop Z$-action on $\deloop G$ restricts to a $\deloop Z^{\sharp}$-action, which we record in the diagram
\begin{equation}
\label{eq-sharp-center-action-diagram}
\begin{tikzcd}[column sep = 0em]
& \deloop G \times \deloop Z^{\sharp} \ar[rr, "a^{\sharp}"]\ar[dl, swap, "p_1"]\ar[dr, "p_2"] & & \deloop G \\
\deloop G & & \deloop Z^{\sharp}
\end{tikzcd}
\end{equation}

Denote by $\mu_{Z^{\sharp}}$ the restriction of $\mu$ to $\deloop Z^{\sharp}$. Recall that $\mu_{Z^{\sharp}}$ has a canonical $\mathbb E_{\infty}$-monoidal structure (\emph{cf.}~Proposition \ref{prop-sharp-center-symmetric-monoidal}).
\end{void}

\begin{cor}
\label{cor-sharp-center-quadratic-structure}
In reference to \eqref{eq-sharp-center-action-diagram}, there is a canonical isomorphism of bi-rigidified morphisms $\deloop G \times \deloop Z^{\sharp} \rightarrow \deloop^4A(1)$:
\begin{equation}
\label{eq-sharp-center-quadratic-structure}
(a^{\sharp})^*\mu - (p_1)^*\mu - (p_2)^*\mu_{Z^{\sharp}} \xrightarrow{\simeq} 0.
\end{equation}
\end{cor}
\begin{proof}
The bilinear pairing $b_2$ (\emph{cf.}~\eqref{eq-center-cocenter-pairing}) restricts to the trivial pairing
$$
\pi_1 G \otimes \Fib(\Lambda^{\sharp} \rightarrow \Lambda^{\sharp}_{\adjoint}) \rightarrow A(-1),
$$
because the horizontal arrows of \eqref{eq-strict-weyl-invariant-form-extension} vanish over $\Lambda^{\sharp}$, respectively $\Lambda^{\sharp}_{\adjoint}$. This induces a trivialization of the restriction of $b_2\otimes \deloop\Psi^{\otimes 2}$ to $\deloop G_{\abelian} \otimes \deloop Z^{\sharp}$.

The isomorphism \eqref{eq-sharp-center-quadratic-structure} is the restriction of \eqref{eq-canonical-quadratic-structure} to $\deloop G \times \deloop Z^{\sharp}$, composed with the trivialization of the right-hand-side defined above.
\end{proof}

\begin{void}
\label{void-sharp-center-equivariance-structure}
The isomorphism \eqref{eq-sharp-center-quadratic-structure} induces an isomorphism of rigidified (\emph{not} bi-rigidified) morphisms $\deloop G \times \deloop Z^{\sharp} \rightarrow \deloop^4A(1)$:
\begin{equation}
\label{eq-sharp-center-equivariance-structure}
(a^{\sharp})^*\mu \xrightarrow{\simeq} (p_1)^*\mu + (p_2)^*\mu_{Z^{\sharp}},
\end{equation}
which may be regarded as the part of a $\deloop Z^{\sharp}$-equivariance structure on $\mu$ ``against $\mu_{Z^{\sharp}}$". The restrictions of \eqref{eq-sharp-center-equivariance-structure} to $\deloop G \times e$ and $e \times \deloop Z^{\sharp}$ are induced from the equality of maps $a^{\sharp} = p_1$, $a^{\sharp} = p_2$ over these loci.

The isomorphism \eqref{eq-sharp-center-equivariance-structure} is equipped with cocycle data. To be more transparent, let us formulate it in functorial terms: Given a $G$-torsor $\mathscr E$ and $Z^{\sharp}$-torsors $\mathscr Z_1$, $\mathscr Z_2$ over an $S$-scheme, the diagram of sections of $\deloop^4A(1)$ commute
\begin{equation}
\label{eq-sharp-center-equivariance-cocycle-diagram}
\begin{tikzcd}[column sep = 1.5em]
	\mu(\mathscr E \otimes \mathscr Z_1 \otimes \mathscr Z_2) \ar[r, "\simeq"]\ar[d, "\simeq"] & \mu(\mathscr E) + \mu_{Z^{\sharp}}(\mathscr Z_1\otimes\mathscr Z_2) \ar[d, "\simeq"] \\
	\mu(\mathscr E \otimes \mathscr Z_1) + \mu_{Z^{\sharp}}(\mathscr Z_2) \ar[r, "\simeq"] & \mu(\mathscr E) + \mu_{Z^{\sharp}}(\mathscr Z_1) + \mu_{Z^{\sharp}}(\mathscr Z_2)
\end{tikzcd}
\end{equation}
Here, the right vertical arrow appeals to the monoidal structure on $\mu_{Z^{\sharp}}$ and the remaining arrows are instances of \eqref{eq-sharp-center-equivariance-structure}.

The commutativity of \eqref{eq-sharp-center-equivariance-cocycle-diagram} follows from the cocycle condition on the canonical quadratic structure (\emph{cf.}~\S\ref{void-canonical-quadratic-structure-cocycle-condition}). Similarly to the latter, it is a \emph{condition} and not additional structure. Likewise, higher coherence (for triples of $Z^{\sharp}$-torsors, \emph{etc.}) is trivially satisfied.
\end{void}

\medskip

\section{Weissman's obstruction}
\label{sec-weissman-obstruction}

Let $F$ be a local field with a fixed algebraic closure $\bar F$. Let $G$ be a reductive group $F$-scheme. Let $A$ be a finite abelian group with order invertible in $F$, equipped with an injective character $\zeta : A \rightarrow \complexes^{\times}$. Let $\mu$ be an $A$-valued \'etale metaplectic cover of $G$.

In this section, we define Weissman's obstruction $\Omega_{\beta}(\sigma)$, starting with the case $\Omega(\sigma)$ for the trivial $G$-isocrystal, and explain why it obstructs the existence of fibers of the conjectural map $\LLC_{\beta}$ (\emph{cf.}~Conjecture \ref{conj-enhanced-local-langlands-correspondence}) at $\sigma$. Being conjectural, we need to assume something about $\LLC_{\beta}$ to make this precise: This is the compatibility with central core characters (\emph{cf.}~Lemma \ref{lem-weissman-obstruction-isocrystal}) which requires Theorem \ref{thm-torus-commutative-cover-duality} to state. Then we express $\Omega_{\beta}(\sigma)$ in terms of $\Omega(\sigma)$ and the Kottwitz invariant of $\beta$ (\emph{cf.}~Theorem \ref{thm-inclusion-difference-kottwitz-invariant}, Corollary \ref{cor-weissman-obstruction-kottwitz-invariant}).

The last two subsections, \S\ref{sec-example-tori} and \S\ref{sec-center-isocrystal}, can be considered supplements to the article. In \S\ref{sec-example-tori}, we prove that for tori, the vanishing of $\Omega_{\beta}(\sigma)$ is necessary and sufficient for $\LLC_{\beta}^{-1}(\sigma)$. In  \S\ref{sec-center-isocrystal}, we prove a ``dual version" of one of the ingredients in Theorem \ref{thm-inclusion-difference-kottwitz-invariant}: It identifies the cover $\widetilde G_{\beta}$ when $G_{\beta}$ is isomorphic to $G$, \emph{i.e.}~when $\beta$ comes from a $Z$-isocrystal.

\subsection{The case for $\widetilde G$}

\begin{void}
We shall associated to $G$ and $\mu$ a finite abelian group $K$ and a map
\begin{equation}
\label{eq-weissman-obstruction}
	\Omega : \Phi(\widetilde H) \rightarrow \Hom(K, \complexes^{\times}).
\end{equation}

For any $\sigma \in \Phi(\widetilde H)$, we shall refer to $\Omega(\sigma)$ as \emph{Weissman's obstruction} of $\sigma$. It has the property that $\Omega(\sigma)\neq 1$ implies that the fiber of the conjectural local Langlands correspondence \eqref{eq-local-langlands-correspondence} at $\sigma$ is empty, assuming ``compatibility with central core characters".

The obstruction $\Omega$ was first observed by Weissman when $G$ is a torus (\emph{cf.}~\cite[\S4]{MR2485462}, \cite[\S8.3]{MR3802419}).
\end{void}

\begin{void}\emph{Definition of $K$.}
\label{void-sharp-center-kernel}
We let $Q$ be the quadratic form associated to $\mu$ (\emph{cf.}~\S\ref{void-etale-metaplectic-cover-fiber-sequence}) and consider the induced \'etale sheaves $\Lambda^{\sharp}$, $\Lambda^{\sharp}_{\sconn}$, $\Lambda^{\sharp}_{\adjoint}$ (\emph{cf.}~\S\ref{void-sharp-lattices}). Tensoring with $\mathbb G_m$, we obtain $F$-tori $T^{\sharp}$, $T^{\sharp}_{\sconn}$, $T^{\sharp}_{\adjoint}$ fitting into a commutative diagram of isogenies
\begin{equation}
\label{eq-sharp-cocharacter-lattice}
\begin{tikzcd}[column sep = 2em]
T_{\sconn}^{\sharp} \ar[r]\ar[d] & T^{\sharp} \ar[r]\ar[d] & T^{\sharp}_{\adjoint} \ar[d] \\
T_{\sconn} \ar[r] &T \ar[r] & T_{\adjoint}
\end{tikzcd}
\end{equation}

Denote by $Z^{\sharp}$ the kernel of $T^{\sharp} \rightarrow T^{\sharp}_{\adjoint}$, so we have a natural map $Z^{\sharp} \rightarrow Z$. Define
$$
K := \Ker(Z^{\sharp}(F) \rightarrow Z(F)).
$$
\end{void}

\begin{void}
Denote by $\widetilde G$ the image of $\mu$ under \eqref{eq-construction-of-covers}. Write $\widetilde Z$ for its pullback along $Z(F) \rightarrow G(F)$ and $\widetilde Z^{\sharp}$ for its further pullback to $Z^{\sharp}(F)$.

The subgroup $\Ker(\widetilde Z^{\sharp} \rightarrow \widetilde Z)$ of $\widetilde Z^{\sharp}$ is identified with $K$ via the projection onto $Z^{\sharp}(F)$. Thus, we obtain a map
\begin{equation}
\label{eq-obstruction-subgroup-sharp-center}
i : K \rightarrow \widetilde Z^{\sharp}.
\end{equation}
\end{void}

\begin{lem}
\label{lem-sharp-center-commutative}
The group $\widetilde Z^{\sharp}$ is commutative and its image in $\widetilde G$ is central.
\end{lem}
\begin{proof}
The commutativity of $\widetilde Z^{\sharp}$ follows from the fact that $\mu_{Z^{\sharp}}$ is $\mathbb E_{\infty}$-monoidal (\emph{cf.}~Proposition \ref{prop-sharp-center-symmetric-monoidal}). It remains to prove that the image of $\widetilde Z^{\sharp}$ in $\widetilde G$ is central.

Denote by $a^{\sharp} : G(F) \times Z^{\sharp}(F) \rightarrow G(F)$ the multiplication map. Applying the construction functor \eqref{eq-construction-of-covers}, with $G \times Z^{\sharp}$ playing the role of $G$, to the isomorphism \eqref{eq-sharp-center-equivariance-structure}, we obtain an isomorphism of covers of $G(F) \times Z^{\sharp}(F)$:
\begin{equation}
\label{eq-sharp-center-cover-equivariance}
(a^{\sharp})^*\widetilde G \xrightarrow{\simeq} (p_1)^*\widetilde G + (p_2)^*\widetilde Z^{\sharp},
\end{equation}
whose restrictions to $G(F) \times e$ and $e \times Z^{\sharp}(F)$ are induced from the equality of maps $a^{\sharp} = p_1$, respectively $a^{\sharp} = p_2$ over these subgroups.

Equivalently, one may express \eqref{eq-sharp-center-cover-equivariance} as a morphism of short exact sequences
$$
\begin{tikzcd}[column sep = 1.5em]
	1 \ar[r] & A \times A \ar[r]\ar[d, "\sum"] & \widetilde G \times \widetilde Z^{\sharp} \ar[r]\ar[d, "\tilde a^{\sharp}"] & G(F) \times Z^{\sharp}(F) \ar[r]\ar[d, "a^{\sharp}"] & 1 \\
	1 \ar[r] & A \ar[r] & \widetilde G \ar[r] & G(F) \ar[r] & 1
\end{tikzcd}
$$
where $\tilde a^{\sharp}$ restricts to the identity on $\widetilde G \times e$ and the natural map on $e \times \widetilde Z^{\sharp}$. By expressing an element $(\tilde g, \tilde z) \in \widetilde G \times \widetilde Z^{\sharp}$ as $(\tilde g, 1) \cdot (1, \tilde z)$, respectively $(1, \tilde z)\cdot(\tilde g, 1)$, and using the fact that $\tilde a^{\sharp}$ is a group homomorphism, we see that $\tilde g$ commutes with the image of $\tilde z$.
\end{proof}

\begin{void}
\label{void-central-core-character}
By Lemma \ref{lem-sharp-center-commutative} and Schur's lemma, $\widetilde Z^{\sharp}$ acts by a character $\chi$ on any irreducible $\zeta$-genuine smooth representation $V$. The association of $\chi$ to $[V]$ defines a map
\begin{equation}
\label{eq-restriction-sharp-center}
	\Pi(\widetilde G) \rightarrow \Pi(\widetilde Z^{\sharp}),
\end{equation}
where the target stands for the set of $\zeta$-genuine smooth characters $\widetilde Z^{\sharp} \rightarrow \complexes^{\times}$.

We refer to the image of $[V] \in \Pi(\widetilde G)$ under \eqref{eq-restriction-sharp-center} as the \emph{central core character} of $[V]$.
\end{void}

\begin{rem}
Our notion of the ``central core character" is different from Weissman's (\emph{cf.}~\cite[\S6.3]{MR3802418}). Namely, the notion of \emph{op.cit.}~concerns only the maximal torus of $Z^{\sharp}$ whereas ours concerns the entire $Z^{\sharp}$.
\end{rem}

\begin{void}\emph{Definition of $\Omega$.}
The map \eqref{eq-weissman-obstruction} is defined to be the composition of \eqref{eq-langlands-parameter-functoriality-cocenter} with the inverse of \eqref{eq-local-langlands-correspondence-sharp-center} and the restriction along \eqref{eq-obstruction-subgroup-sharp-center}:
\begin{align*}
\Omega : \Phi(\widetilde H) &\rightarrow \Phi(\widetilde H_{\abelian}) \\
 &\xrightarrow{\simeq} \Pi(\widetilde Z^{\sharp}) \xrightarrow{i^*} \Hom(K, \complexes^{\times}).
\end{align*}
\end{void}

\begin{lem}[Weissman]
\label{lem-weissman-obstruction}
Suppose that there is a map $\LLC : \Pi(\widetilde G) \rightarrow \Phi(\widetilde H)$ satisfying the following compatibility with central core characters: It renders the diagram
\begin{equation}
\label{eq-local-langlands-compatibility-with-central-character}
\begin{tikzcd}
	\Pi(\widetilde G) \ar[r, "\LLC"]\ar[d, "\eqref{eq-restriction-sharp-center}"] & \Phi(\widetilde H) \ar[d, "\eqref{eq-langlands-parameter-functoriality-cocenter}"] \\
	\Pi(\widetilde Z^{\sharp}) \ar[r, "\eqref{eq-local-langlands-correspondence-sharp-center}"] & \Phi(\widetilde H_{\abelian})
\end{tikzcd}
\end{equation}
commutative. Then for any $\sigma \in \Phi(\widetilde H)$ with $\Omega(\sigma)\neq 1$, the set $\LLC^{-1}(\sigma)$ is empty.
\end{lem}
\begin{proof}
Fix $\sigma \in \Phi(\widetilde H)$ and let $V$ be an irreducible $\zeta$-genuine smooth representation of $\widetilde G$ belonging to fiber of $\LLC$ at $\sigma$.

Suppose that $\widetilde Z^{\sharp}$ acts on $V$ via some character $\chi$. By the commutativity of \eqref{eq-local-langlands-compatibility-with-central-character}, the subgroup $K$ of $\widetilde Z^{\sharp}$ acts by the character $\Omega(\sigma)$. However, since the image of $K$ in $\widetilde G$ is trivial, this implies that $\Omega(\sigma) = 1$.
\end{proof}

\begin{rem}
\label{rem-every-character-weissman-obstruction}
Every character $\chi : K \rightarrow \complexes^{\times}$ occurs as $\Omega(\sigma)$ for some $\sigma \in \Phi(\widetilde H)$.

Indeed, one may first extend $(\zeta, \chi)$ along the inclusion $A \times K \subset \widetilde T^{\sharp}$ to obtain a $\zeta$-genuine character of $\widetilde T^{\sharp}$. Under the local Langlands correspondence for $(T^{\sharp}, \mu_{T^{\sharp}})$ (\emph{cf.}~\S\ref{void-torus-commutative-cover-local-langlands-correspondence}), the latter defines an L-parameter $\sigma_T \in \Phi(\widetilde T_H)$ with respect to the canonical maximal torus $T_H$ of $H$. The image $\sigma\in\Phi(\widetilde H)$ of $\sigma_T$ satisfies $\Omega(\sigma) = \chi$, by construction of \eqref{eq-local-langlands-correspondence-sharp-center}.
\end{rem}

\subsection{The Pontryagin dual of $K$}
\label{sec-pontryagin-dual-kernel}

\begin{void}
Recall the finite abelian group $K$ associated to $G$ and $\mu$ (\emph{cf.}~\S\ref{void-sharp-center-kernel}). In this subsection, we shall construct a surjective map
\begin{equation}
\label{eq-sharp-kernel-quotient-of-fundamental-group}
	\gamma : (\pi_1 G)_{\Gal_F} \rightarrow \Hom(K, \complexes^{\times}).
\end{equation}

Let us note a consequence of Pontryagin duality.
\end{void}

\begin{lem}
\label{lem-adjoint-of-pairing-isomorphism}
Let $\Lambda_1, \Lambda_2$ be \'etale sheaves of finite free $\integers$-modules over $\Spec F$ equipped with a pairing $c : \Lambda_1 \otimes \Lambda_2 \rightarrow A$. Denote by $\Lambda_1^{\sharp} \subset \Lambda_1$, $\Lambda_2^{\sharp} \subset \Lambda_2$ the kernels of $c$. Then the adjoint of $c$ factors through an isomorphism of \'etale sheaves
\begin{equation}
\label{eq-adjoint-of-pairing-isomorphism}
\Lambda_1/\Lambda_1^{\sharp} \xrightarrow{\simeq} \SHom(\Lambda_2/\Lambda_2^{\sharp}, A).
\end{equation}
\end{lem}
\begin{proof}
It suffices to check that \eqref{eq-adjoint-of-pairing-isomorphism} is an isomorphism over a separable closure of $F$, so we may assume that $\Lambda_1$, $\Lambda_2$ are finite free $\integers$-modules rather than sheaves of such.

Since $\Lambda_2/\Lambda^{\sharp}_2$ is $N$-torsion for $N := |A|$, $\zeta$ induces an isomorphism
$$
\SHom(\Lambda_2/\Lambda_2^{\sharp}, A) \xrightarrow{\simeq} (\Lambda_2/\Lambda_2^{\sharp})^{\vee},
$$
where $(\cdot)^{\vee}$ denotes Pontryagin dual, \emph{i.e.}~continuous homomorphisms into the topological group $U_1$ of unit complex numbers.

We view the composite $\zeta\cdot c$ as a $U_1$-valued pairing and consider its adjoint
\begin{equation}
\label{eq-adjoint-of-pairing-zeta-induced}
	\Lambda_1 \rightarrow (\Lambda_2)^{\vee}.
\end{equation}
The kernel of \eqref{eq-adjoint-of-pairing-zeta-induced} equals $\Lambda_1^{\sharp}$. We claim that its cokernel is identified with $(\Lambda_2^{\sharp})^{\vee}$. Indeed, since Pontryagin duality is an exact involution, the dual of \eqref{eq-adjoint-of-pairing-zeta-induced} is $\Lambda_2 \rightarrow (\Lambda_1)^{\vee}$, which has kernel $\Lambda^{\sharp}_2$. The isomorphism \eqref{eq-adjoint-of-pairing-isomorphism} follows.
\end{proof}

\begin{void}
We shall apply Lemma \ref{lem-adjoint-of-pairing-isomorphism} to the pairings $b$ and $b_1$ associated to $Q$. More precisely, applying a Tate twist to \eqref{eq-strict-weyl-invariant-form-extension} and using Lemma \ref{lem-adjoint-of-pairing-isomorphism}, we find a commutative square
\begin{equation}
\label{eq-bilinear-forms-induced-isomorphism}
\begin{tikzcd}[column sep = 1em]
	(\Lambda/\Lambda^{\sharp})(1) \ar[r, "\simeq"]\ar[d] & \SHom(\Lambda/\Lambda^{\sharp}, A) \ar[d] \\
	(\Lambda_{\adjoint}/\Lambda^{\sharp}_{\adjoint})(1) \ar[r, "\simeq"] & \SHom(\Lambda_{\sconn}/\Lambda^{\sharp}_{\sconn}, A)
\end{tikzcd}
\end{equation}
where the horizontal maps are isomorphisms.
\end{void}

\begin{void}\emph{Construction of $\gamma$.}
Taking global sections of \eqref{eq-bilinear-forms-induced-isomorphism} over $\Spec F$, we obtain the commutative square
$$
\begin{tikzcd}[column sep = 1em]
	\Ker(T^{\sharp} \rightarrow T)(F) \ar[r, "\simeq"]\ar[d] & \Hom((\Lambda/\Lambda^{\sharp})_{\Gal_F}, A) \ar[d] \\
	\Ker(T_{\adjoint}^{\sharp} \rightarrow T_{\adjoint})(F) \ar[r, "\simeq"] & \Hom((\Lambda_{\sconn}/\Lambda_{\sconn}^{\sharp})_{\Gal_F}, A) 
\end{tikzcd}
$$
Taking kernels of the vertical maps, we obtain an isomorphism
\begin{equation}
\label{eq-sharp-kernel-pontryagin-dual}
K \xrightarrow{\simeq} \Hom((\pi_1G)_{\Gal_F}/(\pi_1G^{\sharp})_{\Gal_F}, A).
\end{equation}

Since $(\pi_1G)_{\Gal_F}/(\pi_1G^{\sharp})_{\Gal_F}$ is $N$-torsion (for $N := |A|$), its Pontryagin dual is identified with $K$ under \eqref{eq-sharp-kernel-pontryagin-dual}. Applying bi-duality yields a short exact sequence
\begin{equation}
\label{eq-sharp-kernel-pontryagin-dual-as-quotient}
(\pi_1G^{\sharp})_{\Gal_F} \rightarrow (\pi_1G)_{\Gal_F} \xrightarrow{\gamma} \Hom(K, \complexes^{\times}) \rightarrow 1.
\end{equation}

The map \eqref{eq-sharp-kernel-quotient-of-fundamental-group} is defined as the second map displayed in this short exact sequence.
\end{void}

\subsection{The case for $\widetilde G_{\beta}$}

\begin{void}
For each $\beta \in \Isoc_G$, we have a morphism of group $F$-schemes
\begin{equation}
\label{eq-center-mapping-to-extended-pure-inner-form}
Z \rightarrow G_{\beta},
\end{equation}
sending an $R$-point $z$ of $Z$ to the automorphism of the pullback of $\beta$ to $X\times \Spec R$ given by acting by $z$. The image of \eqref{eq-center-mapping-to-extended-pure-inner-form} is central in $G_{\beta}$.

Let $\widetilde G_{\beta}$ denote the image $\mu$ under the construction functor \eqref{eq-construction-of-covers-isocrystal}. Denote by $\widetilde Z_{\beta}$ the pullback of $\widetilde G_{\beta}$ along the map $Z(F) \rightarrow G_{\beta}(F)$ induced from \eqref{eq-center-mapping-to-extended-pure-inner-form} and by $\widetilde Z_{\beta}^{\sharp}$ its further pullback to $Z^{\sharp}(F)$. Thus $\Ker(\widetilde Z^{\sharp}_{\beta} \rightarrow \widetilde Z_{\beta})$ is identified with $K$ along the projection onto $Z^{\sharp}(F)$. This yields an injection
\begin{equation}
\label{eq-obstruction-subgroup-sharp-center-isocrystal}
i_{\beta} : K \rightarrow \widetilde Z^{\sharp}_{\beta}.
\end{equation}

This map specializes to \eqref{eq-obstruction-subgroup-sharp-center} when $\beta$ is the trivial $G$-isocrystal.
\end{void}

\begin{void}
Let us now identify $\widetilde Z_{\beta}$ for any $\beta \in \Isoc_G$.

Recall the bilinear pairing $b_2 \otimes \deloop\Psi^{\otimes 2} : \deloop G_{\abelian} \otimes \deloop Z \rightarrow \deloop^4A(1)$ (\emph{cf.}~\S\ref{void-center-cocenter-pairing-geometric}). Evaluating at the $G_{\abelian}$-isocrystal defined by $\beta$, we find a rigidified morphism
$$
(b_2\otimes \deloop\Psi^{\otimes 2})(\beta, \cdot) : X\times \deloop Z \rightarrow \deloop^4A(1),
$$
whch is canonically the pullback of a rigidified morphism $\deloop Z \rightarrow \deloop^4A(1)$ (\emph{cf.}~Lemma \ref{lem-weil-galois-torsion}), to be denoted using the same expression.

The following result is a consequence of the canonical quadratic structure (\emph{cf.}~Proposition \ref{prop-canonical-quadratic-structure}). Its statement invokes the construction functor \eqref{eq-construction-of-covers} for $Z$.
\end{void}

\begin{prop}
\label{prop-isocrystal-center-cover}
For any $\beta \in \Isoc_G$, there is a canonical isomorphism of covers of $Z(F)$:
\begin{equation}
\label{eq-center-cover-baer-sum}
\widetilde Z_{\beta} \xrightarrow{\simeq} \widetilde Z + \int_F (b_2\otimes \deloop\Psi^{\otimes 2})(\beta, \cdot).
\end{equation}
\end{prop}
\begin{proof}
Consider the action map $a : \deloop G \times \deloop Z \rightarrow \deloop G$. Its restriction along the $G$-isocrystal $\beta : X \rightarrow \deloop G$ yields a morphism
\begin{equation}
\label{eq-center-action-fixed-isocrystal}
a_{\beta} : X\times \deloop Z \rightarrow \deloop G.
\end{equation}
The loop space functor applied to \eqref{eq-center-action-fixed-isocrystal} recovers \eqref{eq-center-mapping-to-extended-pure-inner-form}. More precisely, taking fiber product of $X$ with itself over the two stacks in \eqref{eq-center-action-fixed-isocrystal}, we obtain a morphism from $X\times Z$ to the group $X$-sheaf of automorphisms of $\beta$, which is adjoint to \eqref{eq-center-mapping-to-extended-pure-inner-form}.

Let us pull back $\mu$ along the composition of the projection $p : X\times \deloop Z \rightarrow X$ and $\mu : X \rightarrow \deloop G$. By construction, we have an identification of covers of $Z(F)$
\begin{equation}
\label{eq-center-cover-isocrystal-expression}
\widetilde Z_{\beta} \xrightarrow{\simeq} \int_F (a_{\beta})^*\mu - p^*\beta^*\mu.
\end{equation}

It remains to identify the right-hand-sides of \eqref{eq-center-cover-baer-sum} and \eqref{eq-center-cover-isocrystal-expression}. We shall do so by identifying the ``integrands'', \emph{i.e.}~providing an isomorphism
\begin{equation}
\label{eq-quadratic-structure-restriction-to-isocrystal}
(a_{\beta})^*\mu - p^*\beta^*\mu \xrightarrow{\simeq} \mu_Z + (b_2\otimes \deloop\Psi^{\otimes 2})(\beta, \cdot)
\end{equation}
of rigidified morphisms $X\times \deloop Z \rightarrow \deloop^4A(1)$, where $\mu_Z$ denotes the restriction of $\mu$ along the composition $X\times\deloop Z \rightarrow \deloop Z \rightarrow \deloop G$.

The isomorphism \eqref{eq-quadratic-structure-restriction-to-isocrystal} is the restriction of \eqref{eq-canonical-quadratic-structure} along $(\beta, \id) : X\times \deloop Z \rightarrow \deloop G \times\deloop Z$.
\end{proof}

\begin{rem}
It follows from Proposition \ref{prop-isocrystal-center-cover} that the cover $\widetilde Z_{\beta}$ depends only on the $G_{\abelian}$-isocrystal induced from $\beta$.
\end{rem}

\begin{void}
Note that pairing $b_2 \otimes \deloop\Psi^{\otimes 2}$ is canonically trivialized over $\deloop G_{\abelian} \otimes \deloop Z^{\sharp}$ (\emph{cf.}~the proof of Corollary \ref{cor-sharp-center-quadratic-structure}). In particular, the pullback of \eqref{eq-center-cover-baer-sum} along $Z^{\sharp}(F) \rightarrow Z(F)$ yields an isomorphism of covers of $Z^{\sharp}(F)$:
\begin{equation}
\label{eq-sharp-center-cover-identification}
	\widetilde Z_{\beta}^{\sharp} \xrightarrow{\simeq} \widetilde Z^{\sharp}.
\end{equation}

Let us compose the inverse of \eqref{eq-sharp-center-cover-identification} with the natural map $\widetilde Z_{\beta}^{\sharp} \rightarrow \widetilde G_{\beta}$ to obtain a map:
\begin{equation}
\label{eq-sharp-center-map-to-isocrystal}
\widetilde Z^{\sharp} \rightarrow \widetilde G_{\beta},
\end{equation}
\end{void}

\begin{lem}
\label{lem-sharp-center-isocrystal}
The image of \eqref{eq-sharp-center-map-to-isocrystal} is central in $\widetilde G_{\beta}$.
\end{lem}
\begin{proof}
Denote by $\mu_{G_{\beta}} := T_{\beta}(\mu)$ the translation of $\mu$ by $\beta$ (\emph{cf.}~\S\ref{void-isocrystal-translation-functor}). Consider the $\deloop Z^{\sharp}$-action on $\deloop G_{\beta}$ via the inclusion \eqref{eq-center-mapping-to-extended-pure-inner-form}.

By the proof of Lemma \ref{lem-sharp-center-commutative}, it suffices to show that $\mu_{G_{\beta}}$ is $\deloop Z^{\sharp}$-equivariant against $\mu_{Z^{\sharp}} : \deloop Z^{\sharp} \rightarrow \deloop^4A(1)$ in the sense of \S\ref{void-sharp-center-equivariance-structure} and that, upon acting on the neutral point of $\deloop G_{\beta}$, this equivariance structure reduces to the isomorphism
\begin{equation}
\label{eq-sharp-center-isocrystal-restriction}
\mu_{G_{\beta}}|_{\deloop Z^{\sharp}} \xrightarrow{\simeq} \mu_{Z^{\sharp}}
\end{equation}
induced from \eqref{eq-quadratic-structure-restriction-to-isocrystal} and the trivialization of $b_2\otimes \deloop\Psi^{\otimes 2}$ over $\deloop G_{\abelian} \times \deloop Z^{\sharp}$.

By Lemma \ref{lem-weil-galois-torsion}, it suffices to construct the $\deloop Z^{\sharp}$-equivariance structure after base change along $X \rightarrow \Spec F$. The base change of $\mu_{G_{\beta}}$ to $X\times\deloop G_{\beta}$ is the pullback of $\mu$ along \eqref{eq-isocrystal-classifying-stack-translation-map} minus the constant section $p^*\beta^*\mu$. However, \eqref{eq-isocrystal-classifying-stack-translation-map} is $\deloop Z^{\sharp}$-equivariant, so the desired $\deloop Z^{\sharp}$-equivariance structure on $\mu_{G_{\beta}}$ follows from that of $\mu$ (\emph{cf.}~\S\ref{void-sharp-center-equivariance-structure}). The fact that acting on the neutral point of $\deloop G_{\beta}$ recovers the isomorphism \eqref{eq-sharp-center-isocrystal-restriction} is a consequence of the construction of \eqref{eq-sharp-center-equivariance-structure} (which uses the trivialization of $b_2\otimes \deloop\Psi^{\otimes 2}$ over $\deloop G\times \deloop Z^{\sharp}$).
\end{proof}

\begin{void}
By Lemma \ref{lem-sharp-center-isocrystal} and Schur's lemma, we have a map
\begin{equation}
\label{eq-restriction-sharp-center-isocrystal}
\Pi(\widetilde G_{\beta}) \rightarrow \Pi(\widetilde Z^{\sharp})
\end{equation}
sending the isomorphism class $[V]$ of an irreducible $\zeta$-genuine smooth representation $V$ of $\widetilde G_{\beta}$ to the character of $\widetilde Z^{\sharp}$ by which it acts on $V$ through \eqref{eq-sharp-center-map-to-isocrystal}.

The image of $[V]$ under \eqref{eq-restriction-sharp-center-isocrystal} can be viewed as the ``central core character" of $[V]$, generalizing the construction of \S\ref{void-central-core-character}.
\end{void}

\begin{void}
We now arrive at a crucial point: The isomorphism \eqref{eq-sharp-center-cover-identification} is generally \emph{incompatible} with the inclusions of $K$ via $i$ and $i_{\beta}$ (\emph{cf.}~\eqref{eq-obstruction-subgroup-sharp-center}, \eqref{eq-obstruction-subgroup-sharp-center-isocrystal}). In other words, the quotient $i_{\beta}/i$ factors through a character
\begin{equation}
\label{eq-isocrystal-sharp-kernel-difference}
K \rightarrow A.
\end{equation}

We express this character in terms of the Kottwitz invariant of $\beta$ (\emph{cf.}~\S\ref{void-kottwitz-invariant}).
\end{void}

\begin{thm}
\label{thm-inclusion-difference-kottwitz-invariant}
The character \eqref{eq-isocrystal-sharp-kernel-difference} equals the image of $\Kottwitz(\beta)$ under \eqref{eq-sharp-kernel-quotient-of-fundamental-group}, \emph{i.e.}
\begin{equation}
\label{eq-inclusion-difference-kottwitz-invariant}
\frac{i_{\beta}}{i} = \gamma(\Kottwitz(\beta)).
\end{equation}
\end{thm}
\begin{proof}
Using the isomorphism \eqref{eq-center-cover-baer-sum} and the $\integers$-linear structure on $\Cov(Z(F), A)$, we may express \eqref{eq-isocrystal-sharp-kernel-difference} as follows: Consider the cover
$$
\int_F(b_2 \otimes \deloop\Psi^{\otimes 2})(\beta, \cdot) \in \Cov(Z(F), A)
$$
equipped with the splitting over $Z^{\sharp}(F)$ defined by the trivialization of $(b_2\otimes \deloop\Psi^{\otimes 2})(\beta, \cdot)$ over $\deloop Z^{\sharp}$. The restriction of this splitting to $K$ is the character \eqref{eq-isocrystal-sharp-kernel-difference}.

The $\integers$-linear morphism $b_2\otimes \Psi^{\otimes 2} : \deloop G_{\abelian} \otimes \deloop Z \rightarrow \deloop^4A(1)$ is trivialized over $\deloop G_{\abelian} \otimes \deloop Z^{\sharp}$, so by taking fibers, we obtain a pairing
$$
\langle\cdot, \cdot\rangle : \deloop G_{\abelian}\otimes \Fib(Z^{\sharp} \rightarrow Z) \rightarrow \deloop^2A(1).
$$

This pairing encodes the character \eqref{eq-isocrystal-sharp-kernel-difference} in the following manner: Given $\beta : X \rightarrow \deloop G_{\abelian}$ and $a \in K \cong H^0(\Spec F, \Fib(Z^{\sharp} \rightarrow Z))$, the class of $\langle\beta, a\rangle$ in
\begin{equation}
\label{eq-tate-duality-for-isoc}
H^2(X, A(1)) \cong H^2(\Spec F, A(1)) \cong A
\end{equation}
is the image of $a$ under \eqref{eq-isocrystal-sharp-kernel-difference}. Here, the isomorphisms are given by pullback along $X \rightarrow \Spec F$ (\emph{cf.}~Lemma \ref{lem-weil-galois-torsion}) and Tate duality.

On the other hand, $\langle\cdot, a\rangle : \deloop G_{\abelian} \rightarrow \deloop^2A(1)$ is the tensor product of a $\integers$-linear map $\pi_1 G \rightarrow A$ with the Kummer map. By construction, this $\integers$-linear map is the image of $a\in K$ under \eqref{eq-sharp-kernel-pontryagin-dual}. The desired equality \eqref{eq-inclusion-difference-kottwitz-invariant} thus reduces to the following compatibility between Kottwitz invariant \eqref{eq-kottwitz-invariant} and Tate duality: Given any map of \'etale sheaves $f : \pi_1 G \rightarrow A$, the following diagram commutes
\begin{equation}
\label{eq-kottwitz-invariant-compatibility-with-tate-duality}
\begin{tikzcd}[column sep = 1em]
	\pi_0\Isoc_{G_{\abelian}} \ar[r, "\Kottwitz"]\ar[d, "f \otimes \Psi"] & (\pi_1 G)_{\Gal_F} \ar[d, "f"] \\
	H^2(X, A(1)) \ar[r, "\eqref{eq-tate-duality-for-isoc}"] & A
\end{tikzcd}
\end{equation}
Here, the left vertical arrow is induced from $f\otimes \Psi : G_{\abelian} \rightarrow \deloop A(1)$. The commutativity of \eqref{eq-kottwitz-invariant-compatibility-with-tate-duality} reduces to the case where $G$ is a torus by construction (\emph{cf.}~\S\ref{void-kottwitz-invariant}), then to the case where $G = \mathbb G_m$ by functoriality (\emph{cf.}~\cite[\S2]{MR809866}), where it follows from the definition.
\end{proof}

\begin{void}
Finally, let us define (the generalized) Weissman's obstruction
$$
\Omega_{\beta} : \Phi(\widetilde H) \rightarrow \Hom(K, \complexes^{\times})
$$
for an arbitrary $G$-isocrystal $\beta$.

We set $\Omega_{\beta}$ to be the composition
\begin{align*}
\Omega_{\beta} : \Phi(\widetilde H) & \rightarrow \Phi(\widetilde H_{\abelian}) \\
& \xrightarrow{\simeq} \Pi(\widetilde Z^{\sharp}) \xrightarrow{i_{\beta}^*} \Hom(K, \complexes^{\times}).
\end{align*}
The proof of Lemma \ref{lem-weissman-obstruction} also yields the following result.
\end{void}

\begin{lem}
\label{lem-weissman-obstruction-isocrystal}
Suppose that there is a map $\LLC_{\beta} : \Pi(\widetilde G_{\beta}) \rightarrow \Phi(\widetilde H)$ satisfying the following compatibility with central core characters: It renders the diagram
\begin{equation}
\label{eq-local-langlands-compatibility-with-central-character-isocrystal}
\begin{tikzcd}
	\Pi(\widetilde G_{\beta}) \ar[r, "\LLC_{\beta}"]\ar[d, "\eqref{eq-restriction-sharp-center-isocrystal}"] & \Phi(\widetilde H) \ar[d, "\eqref{eq-langlands-parameter-functoriality-cocenter}"] \\
	\Pi(\widetilde Z^{\sharp}) \ar[r, "\eqref{eq-local-langlands-correspondence-sharp-center}"] & \Phi(\widetilde H_{\abelian})
\end{tikzcd}
\end{equation}
commutative. Then for any $\sigma \in \Phi(\widetilde H)$ with $\Omega_{\beta}(\sigma)\neq 1$, the set $\LLC_{\beta}^{-1}(\sigma)$ is empty.\qed
\end{lem}

\begin{cor}
\label{cor-weissman-obstruction-kottwitz-invariant}
For each $\sigma \in \Phi(\widetilde H)$ and $\beta \in \Isoc_G$, the character $\Omega_{\beta}(\sigma)$ vanishes if and only if
\begin{equation}
\label{eq-weissman-obstruction-kottwitz-invariant}
\gamma(\Kottwitz(\beta)) = \Omega(\sigma)^{-1}.
\end{equation}
\end{cor}
\begin{proof}
The quotient $\Omega_{\beta}(\sigma)/\Omega(\sigma)$ is the character $i_{\beta}/i : K \rightarrow \complexes^{\times}$. By Theorem \ref{thm-inclusion-difference-kottwitz-invariant}, the latter equals $\gamma(\Kottwitz(\beta))$. Hence the equality \eqref{eq-weissman-obstruction-kottwitz-invariant} holds if and only if $\Omega_{\beta}(\sigma) = 1$.
\end{proof}

\begin{void}
\label{void-weissman-obstruction-vanishing-existence}
To conclude, let us rewrite the short exact sequence \eqref{eq-sharp-kernel-pontryagin-dual-as-quotient} using \eqref{eq-basic-isocrystal-kottwitz-invariant-bijection}:
$$
\pi_0(\Basic_{G^{\sharp}}) \rightarrow \pi_0(\Basic_G) \rightarrow \Hom(K, \complexes^{\times}) \rightarrow 1.
$$
Here, the middle arrow sends $\beta$ to $\gamma(\Kottwitz(\beta))$. The group structure on $\pi_0(\Basic_G)$ is induced from that on $(\pi_1G)_{\Gal_F}$, and similarly for $\pi_0(\Basic_{G^{\sharp}})$.

Given $\sigma \in \Phi(\widetilde H)$, Corollary \ref{cor-weissman-obstruction-kottwitz-invariant} shows that there exists a basic $G$-isocrystal $\beta$ for which \eqref{eq-weissman-obstruction-kottwitz-invariant} holds. Furthermore, the subset of $\pi_0(\Basic_G)$ consisting of isomorphism classes of such $\beta$ forms a torsor under the image of $\pi_0(\Basic_{G^{\sharp}})$.

By Remark \ref{cor-weissman-obstruction-kottwitz-invariant}, the character $\Omega(\sigma)^{-1}$ of $K$ is arbitrary as $\sigma$ varies. This means that to guarantee the equality \eqref{eq-weissman-obstruction-kottwitz-invariant}, one really needs to consider basic $G$-isocrystals spanning a full set of representatives of $\pi_0(\Basic_G)/\pi_0(\Basic_{G^{\sharp}})$.
\end{void}

\subsection{Example: tori}
\label{sec-example-tori}

\begin{void}
In this subsection, we specialize to the case $G = T$ is an $F$-torus. We shall construct the local Langlands correspondence (\emph{cf.}~Conjecture \ref{conj-enhanced-local-langlands-correspondence}) for $T$.

More precisely, for each $\beta \in \Isoc_T$, we shall construct a map
\begin{equation}
\label{eq-local-langlands-correspondence-tori-fixed-isocrystal}
\LLC_{\beta} : \Pi(\widetilde T_{\beta}) \rightarrow \Phi(\widetilde H).
\end{equation}

In fact, $\LLC_{\beta}$ is uniquely determined by the compatibility diagram \eqref{eq-local-langlands-compatibility-with-central-character-isocrystal} since \eqref{eq-langlands-parameter-functoriality-cocenter} becomes an isomorphism in this case, so let us turn this into a definition.
\end{void}

\begin{void}\emph{Construction of $\LLC_{\beta}$.}
The map \eqref{eq-sharp-center-map-to-isocrystal} specializes to a map
\begin{equation}
\label{eq-restriction-sharp-center-isocrystal-torus}
\widetilde T^{\sharp} \rightarrow \widetilde T_{\beta}
\end{equation}
whose image is central (\emph{cf.}~Lemma \ref{lem-sharp-center-isocrystal}). Thus, given any irreducible $\zeta$-genuine smooth representation $V$ of $\widetilde T_{\beta}$, the action of $\widetilde T^{\sharp}$ on $V$ through \eqref{eq-restriction-sharp-center-isocrystal-torus} is a $\zeta$-genuine smooth character. This defines a map
\begin{equation}
\label{eq-torus-central-core-character}
\Pi(\widetilde T_{\beta}) \rightarrow \Pi(\widetilde T^{\sharp}).
\end{equation}

The map $\LLC_{\beta}$ is the composition of \eqref{eq-torus-central-core-character} with the local Langlands correspondence for $T^{\sharp}$ equipped with the restriction $\mu_{T^{\sharp}}$ of $\mu$ (\emph{cf.}~\S\ref{void-torus-commutative-cover-local-langlands-correspondence}).
\end{void}

\begin{void}
The following description of $\LLC_{\beta}^{-1}(\sigma)$ generalizes Weissman's result for the trivial $T$-isocrystal $\beta$ (\emph{cf.}~\cite[Theorem 2.12]{MR3595494}), with the same proof.
\end{void}

\begin{prop}
\label{prop-local-langlands-correspondence-tori}
Given $\beta \in \Isoc_T$ and $\sigma \in \Phi(\widetilde H)$, the set $\LLC_{\beta}^{-1}(\sigma)$ is finite and nonempty if and only if \eqref{eq-weissman-obstruction-kottwitz-invariant} holds.
\end{prop}
\begin{proof}
Lemma \ref{lem-weissman-obstruction-isocrystal} and Corollary \ref{cor-weissman-obstruction-kottwitz-invariant} together imply that $\LLC_{\beta}^{-1}(\sigma)$ is empty when \eqref{eq-weissman-obstruction-kottwitz-invariant} fails. It remains to prove that when \eqref{eq-weissman-obstruction-kottwitz-invariant} holds, $\LLC_{\beta}^{-1}(\sigma)$ is nonempty and finite.

The L-parameter $\sigma$ corresponds, under the local Langlands correspondence for $(T^{\sharp}, \mu^{\sharp})$ (\emph{cf.}~\S\ref{void-torus-commutative-cover-local-langlands-correspondence}), to a $\zeta$-genuine smooth character $\chi_{\sigma}$ of $\widetilde T^{\sharp}$. By construction, a $\zeta$-genuine smooth representation $V$ of $\widetilde T_{\beta, \zeta}$ has L-parameter $\sigma$ if and only if $\widetilde T^{\sharp}$ acts on $V$ via $\chi_{\sigma}$. Since \eqref{eq-weissman-obstruction-kottwitz-invariant} holds, $\chi_{\sigma}$ annihilates the kernel of \eqref{eq-restriction-sharp-center-isocrystal-torus}, so it factors through a character
$$
\bar{\chi}_{\sigma} : \widetilde T^{\sharp}/K \rightarrow \complexes^{\times},
$$
where $\widetilde T^{\sharp}/K$ is identified with a subgroup of the center $\widetilde C_{\beta}$ of $\widetilde T_{\beta}$ (\emph{cf.}~Lemma \ref{lem-sharp-center-isocrystal}).

By the Stone--von Neumann theorem, $\Pi(\widetilde T_{\beta})$ is in bijection with genuine characters of $\widetilde C_{\beta}$. Hence $\LLC_{\beta}^{-1}(\sigma)$ is in bijection with extensions of $\bar{\chi}_{\sigma}$ along the inclusion
$$
\widetilde T^{\sharp}/K \subset \widetilde C_{\beta},
$$
which is of finite index. This implies that $\LLC_{\beta}^{-1}(\sigma)$ is nonempty and finite.
\end{proof}

\subsection{$Z$-isocrystals}
\label{sec-center-isocrystal}

\begin{void}
We return to the context where $G$ is a reductive group $F$-scheme. Given a $Z$-isocrystal $\beta$, we may consider the induced $G$-isocrystal, hence the group $F$-scheme $G_{\beta}$. There is a canonical isomorphism of group $F$-schemes
\begin{equation}
\label{eq-center-isocrystal-form-isomorphism}
G \xrightarrow{\simeq} G_{\beta},
\end{equation}
defined as follows: Restricting the action map $a : \deloop G \times \deloop Z \rightarrow \deloop G$ along the $Z$-isocrystal $\beta : X \rightarrow \deloop Z$ yields a morphism $a_{\beta} : \deloop G \times X \rightarrow \deloop G$, which induces \eqref{eq-center-isocrystal-form-isomorphism} on loop spaces.

In this subsection, we express the pullback of $\widetilde G_{\beta}$ along \eqref{eq-center-isocrystal-form-isomorphism} in terms of the cover $\widetilde G$.
\end{void}

\begin{void}
Evaluating the bilinear pairing $b_2 \otimes \deloop\Psi^{\otimes 2} : \deloop G_{\abelian} \otimes \deloop Z \rightarrow \deloop^4A(1)$ at $\beta : X \rightarrow \deloop Z$ and descending along $X\rightarrow \Spec F$ (\emph{cf.}~Lemma \ref{lem-weil-galois-torsion}), we obtain a $\integers$-linear morphism
\begin{equation}
\label{eq-bilinear-pairing-evaluation-at-center-isocrystal}
b_2\otimes \deloop\Psi^{\otimes 2}(\cdot, \beta) : \deloop G_{\abelian} \rightarrow \deloop^4A(1),
\end{equation}
which defines a rigidified morphism $\deloop G \rightarrow \deloop^4A(1)$ that we denote by the same expression.

The following result can be thought of as a ``dual version" of Proposition \ref{prop-isocrystal-center-cover}.
\end{void}

\begin{prop}
\label{prop-center-isocrystal-cover-identification}
For any $\beta \in \Isoc_Z$, there is a canonical isomorphism of covers of $G(F)$ with regard to the identification \eqref{eq-center-isocrystal-form-isomorphism}:
\begin{equation}
\label{eq-center-isocrystal-cover-identification}
\widetilde G_{\beta} \xrightarrow{\simeq} \widetilde G + \int_F (b_2\otimes \deloop\Psi^{\otimes 2})(\cdot, \beta).
\end{equation}
\end{prop}
\begin{proof}
Consider the pullback $p^*\beta^*\mu$ of $\mu$ along the projection $p : \deloop G \times X \rightarrow X$ and the $G$-isocrystal $\beta : X \rightarrow \deloop G$.

The isomorphism \eqref{eq-canonical-quadratic-structure}, restricted along $(\id, \beta) : \deloop G \times X \rightarrow \deloop G\times\deloop Z$, yields an isomorphism of rigidified section of $\deloop^4A(1)$ over $\deloop G\times X$:
\begin{equation}
\label{eq-center-isocrystal-parameter-identification}
(a_{\beta})^*\mu - p^*\beta^*\mu \xrightarrow{\simeq} \mu + (b_2\otimes \deloop\Psi^{\otimes 2})(\cdot, \beta),
\end{equation}
or equivalently, over $\deloop G$ (\emph{cf.}~Lemma \ref{lem-weil-galois-torsion}).

The isomorphism \eqref{eq-center-isocrystal-cover-identification} is the image of \eqref{eq-center-isocrystal-parameter-identification} under \eqref{eq-construction-of-covers}.
\end{proof}

\begin{rem}
\label{rem-kazhdan-patterson-linear-difference}
Proposition \ref{prop-center-isocrystal-cover-identification} expresses $\widetilde G_{\beta}$ as the image of the rigidified morphism $\mu + (b_2\otimes \deloop\Psi^{\otimes 2})(\cdot, \beta)$ under \eqref{eq-construction-of-covers}.

If $\mu$ is the \'etale realization of a central extension of $G$ by $\Ktheory_2$ (\emph{cf.}~\cite[\S2.3]{zhao2022metaplectic}), one may wonder whether the rigidified morphism $\mu + (b_2 \otimes \deloop\Psi^{\otimes 2})(\cdot, \beta)$ also comes from \'etale realization. This is generally \emph{not} the case.

For a ``naturally occurring" example, let us take $G := \GL_2$ endowed with the Kazhdan--Patterson cover, viewed as a central extension $E$ of $G$ by $\Ktheory_2$ (\emph{cf.}~\cite[\S13.2]{MR3802419}). Assuming $\characteristic F \neq 2$, the latter defines a rigidified morphism $\mu : \deloop G \rightarrow \deloop^4\{\pm 1\}^{\otimes 2}$ under \'etale realization. Identifying both $\pi_1 G$ and $\Fib(\Lambda \rightarrow \Lambda_{\adjoint})$ with $\integers$, the bilinear form $b_2$ is given by
$$
b_2 : \integers \otimes \integers \rightarrow \integers/2,\quad 1\otimes 1\mapsto 1.
$$

We argue that $(b_2\otimes \deloop\Psi^{\otimes 2})(\cdot, \beta)$ (hence its sum with $\mu$) does not lift to a central extension of $G$ by $\Ktheory_2$, unless $\beta$ is the trivial $Z$-isocrystal. Indeed, $(b_2\otimes \deloop\Psi^{\otimes 2})(\cdot, \beta)$ arises as the tensor product of $\Psi$ with a $\integers$-linear morphism
\begin{equation}
\label{eq-kazhdan-patterson-cover-linear-difference}
\pi_1 G \rightarrow \deloop^2\{\pm 1\},
\end{equation}
which sends the generator of $\pi_1 G\cong \integers$ to the Kummer gerbe $\deloop\Psi(\beta)$ of $\beta$---the latter represents a nontrivial class in $H^2(\Spec F, \{\pm 1\})$ when $\beta$ is nontrivial. If $(b_2\otimes\deloop\Psi^{\otimes 2})(\cdot, \beta)$ lifts to a central extension of $G$ by $\Ktheory_2$, then the $\mathbb E_1$-monoidal morphism $\Lambda \rightarrow \deloop^2\{\pm 1\}$, obtained by pre-composing \eqref{eq-kazhdan-patterson-cover-linear-difference} with the projection $\Lambda\twoheadrightarrow \pi_1 G$, can be expressed in terms of the second Brylinski--Deligne invariant of $E$, \emph{i.e.}~it factors as a monoidal morphism
$$
\Lambda \rightarrow \deloop \mathbb G_m \xrightarrow{\deloop\Psi} \deloop^2\{\pm 1\}.
$$
This implies that $\deloop\Psi(\beta)$ lifts to a section of $\deloop \mathbb G_m$ over $\Spec F$, hence trivial by Hilbert 90.
\end{rem}

\bibliographystyle{amsalpha}
\bibliography{bibliography.bib}

\end{document}